\documentclass{article}
\usepackage[utf8]{inputenc}
\usepackage[T1]{fontenc}
\usepackage{graphicx}
\usepackage{grffile}
\usepackage{longtable}
\usepackage{wrapfig}
\usepackage{rotating}
\usepackage[normalem]{ulem}
\usepackage{amsmath}
\usepackage{textcomp}
\usepackage{amssymb}
\usepackage{capt-of}
\usepackage{hyperref}
\usepackage{a4wide}
\usepackage[utf8]{inputenc}
\usepackage{amsmath}
\usepackage{amsthm}
\usepackage{tikz}
\usetikzlibrary{calc} \usetikzlibrary{arrows}
\makeatletter
\let\@fnsymbol\@arabic
\newtheorem{proposition}{Proposition}[section]
\newtheorem{corollary}[proposition]{Corollary}
\newtheorem{lemma}[proposition]{Lemma}
\newtheorem{theorem}[proposition]{Theorem}
\theoremstyle{definition}
\newtheorem{define}[proposition]{Definition}
\newtheorem{remark}[proposition]{Remark}
\newtheorem{model}{Model}
\newtheorem{examp}{Example}[section]
\usetikzlibrary{shapes}

\newcommand{\aut}{\operatorname{Aut}}
\newcommand{\caut}{\operatorname{CAut}}
\newcommand{\aaut}{\operatorname{AAut}}
\newcommand{\normal}{\unlhd}
\newcommand{\wl}{\operatorname{WL}}
\newcommand{\id}{\operatorname{Id}}
\newcommand{\identity}{{\mathbb I}}
\newcommand{\allone}{{\mathbb J}}
\newcommand{\gap}{{\sf GAP\ }}
\newcommand{\noproof}{\hfill \qedsymbol}
\newcommand{\setchoose}[2]{\genfrac {\{} {\}} {0pt} 1 {#1} {#2}}
\newcommand{\mtrx}[4] {\begin{pmatrix} #1 & #2 \\ #3 &
    #4 \end{pmatrix}}
\newcommand{\correction}[1]{{#1}}
\newcommand{\strike}[1]{}

\author{Mikhail Klin\thanks{Ben Gurion University, Beer Sheva, Israel, klin@math.bgu.ac.il}\\
Mikhail Muzychuk\thanks{Ben Gurion University, Beer Sheva, Israel, muzychuk@bgu.ac.il}\\
Sven Reichard\thanks{Dresden International University,  Dresden, Germany, sven.reichard@tu-dresden.de}}
\date{\today}
\title{Proper Jordan schemes exist. First examples, computer search, patterns of reasoning. An essay}
\hypersetup{
 pdfauthor={Mikhail Klin\thanks{Ben Gurion University, Beer Sheva, Israel, klin@math.bgu.ac.il}\\
Mikhail Muzychuk\thanks{Ben Gurion University, Beer Sheva, Israel, muzychuk@bgu.ac.il}\\
Sven Reichard\thanks{Dresden International University,  Dresden, Germany}},
 pdftitle={Proper Jordan schemes exist. First examples, computer search, patterns of reasoning. An essay.},
 pdfkeywords={},
 pdfsubject={},
 pdfcreator={Emacs 25.2.2 (Org mode 9.1.13)}, 
 pdflang={English}}
\begin{document}

\maketitle
\renewcommand{\theenumi}{\alph{enumi}}

\begin{center}
Dedicated to the memory of Luba Bron (1930 - 2019) 
\end{center}

\begin{abstract}
A special class of Jordan algebras over a field \(F\) of characteristic
zero is considered. Such an algebra consists of an \(r\)-dimensional
subspace of the vector space of all square matrices of a fixed
order \(n\) over \(F\). It contains the identity matrix, the all-one
matrix; it is closed with respect to \correction{matrix transposition}, Schur-Hadamard (entrywise)
multiplication \correction{and} the Jordan product \(A*B=\frac 12 (AB+BA)\),
where \(AB\) is the usual matrix product.

The suggested axiomatics (with some natural additional requirements)
implies an equivalent reformulation in terms of symmetric binary
relations on a vertex set of cardinality \(n\). The appearing
graph-theoretical structure is called a Jordan scheme of order \(n\) and
rank \(r\). A significant source of Jordan schemes stems from the
symmetrization of association schemes. Each such structure is called a
non-proper Jordan scheme. The question about the existence of proper
Jordan schemes was posed a few times by Peter J. Cameron.

In the current text an affirmative answer to this question is
given. The first small examples presented here have orders
\(n=15,24,40\). Infinite classes of proper Jordan schemes of rank 5 and
larger are introduced. A prolific construction for schemes of rank 5
and order \(n=\binom{3^d+1}{2}\), \(d\in {\mathbb N}\), is outlined. 

The text is written in the style of an essay. The long exposition relies
on initial computer experiments, a large amount of diagrams, and
finally is supported by a number of patterns of general theoretical
reasonings. The essay contains also a historical survey and an
extensive bibliography.
\end{abstract}

\begin{small}
\gap
\end{small}

\pagebreak
\tableofcontents
\pagebreak
\section{Introduction}
\label{sec:orga22b9b1}
\label{orgb33f855} 
A Jordan algebra \(J\) over a field \(F\) is a \correction{commutative} non-associative algebra
with binary operation \(*\) which \strike{is commutative and which} satisfies
the following Jordan identity: 
\[
  (x*y) * (x*x) = x*(y*(x*x)).
  \]

This concept was introduced by Pascual Jordan (1933) in the
framework of his attempts to axiomatize observables in quantum
mechanics.

Soon after, Jordan algebras became a subject of independent interest
in mathematics. This field is still developing and flourishing, more
or less independently of its links with physics.

In the current text we are interested in a special kind of Jordan
matrix algebras over a field \correction{$F$} of characteristic zero, as a rule
\({\mathbb C}\), which we call \emph{coherent Jordan algebras}. This
concept appears as an amalgamation of Jordan algebras with coherent
algebras, which were introduced by D.G. Higman (1970).

By a coherent Jordan algebra \(J\) over the field \correction{$F$} we
understand a vector space of square matrices of order \(n\) which
contains the identity matrix \({\mathbb I}_n\), the all-one matrix
\({\mathbb J}_n\), and which is closed under transposition as well as under
two binary operations: the Schur-Hadamard (entrywise) product
\(A\circ B\) and the Jordan product \(A*B\), defined as
\[
  A*B = \frac 12 (AB + BA),
  \]
where \(AB\) is the usual product of the matrices \(A\) and \(B\).

Following a classical pattern of reasoning in Algebraic Graph Theory
(AGT for short) for coherent algebras, we associate to each coherent
Jordan algebra a certain relational structure. 

Indeed, the introduced axioms for a coherent Jordan algebra \(J\) of
order \(n\) and rank \(r\) (that is, its dimension as a vector
space over \correction{the base field}) easily imply that the algebra \(J\) has a
unique standard basis \(\{A_0, A_1, \ldots, A_{r-1}\}\) consisting of
(0,1)-matrices. Here \(\sum_{i=0}^{r-1} A_i = {\mathbb J}_n\), and
the sum of a suitable subset of basic matrices \(A_i\) is equal to
\({\mathbb I}_n\).

Let us denote by \(\Omega=\Omega_n\) a fixed set of cardinality
\(n\). Then each basic matrix \(A_i\) can be interpreted as the
adjacency matrix of a basic graph \(\Gamma_i=(\Omega, R_i)\). The set
\(\{R_0, R_1, \ldots, R_{r-1}\}\) forms a partition of the cartesian
square \(\Omega^2\) of \(\Omega\). Each of those basic relations \(R_i\)
is either symmetric or antisymmetric.

Establishing such bijections between the sets of basic matrices
\(\{A_i\}\), basic relations \(\{R_i\}\), and basic graphs
\(\{\Gamma_i\}\), \(i\in\{0,1,\ldots, r-1\}\), we are able to freely
switch between the three parallel languages of matrices, relations
and graphs. Frequent exploitation of this standard paradigm in AGT
will be also quite essential in the current text.

Thus, following this pattern, let us associate with a given Jordan
algebra \(J\) with standard basis \(\{A_0, A_1, \ldots, A_{r-1}\}\) a
combinatorial structure \({\frak X} = (\Omega, {\cal R})\), were
\({\cal R} = \{R_0, \ldots, R_{r-1}\}\) is the set of basic relations
on the vertex set \(\Omega\). We will call the arising structure a
\emph{coherent Jordan configuration}. Note that in the classical case in
AGT considered above, the relational analogue of a coherent algebra
is a coherent configuration (CC).

Recall also that a very significant particular case of CC's
corresponds to the symmetric association schemes (AS's). Here, each
basic relation \(R_i\) is symmetric, while each basic graph \(\Gamma_i\)
is undirected. Also, exactly one of the basic relations, typically
denoted by \(R_0\), is the identity relation \(\id_n\) on the set
\(\Omega\), while all other basic relations are antireflexive; the
amount \(r-1\) of such relations is sometimes denoted by \(d\).

Let us call the corresponding special case of coherent Jordan
configurations with unique reflexive basic relations \emph{Jordan
schemes} (JS's). Although at the beginning of the text some general
properties of coherent Jordan configurations will be considered,
mainly this paper deals with JS's.

At this moment we are prepared  to discuss the motivation, obtained
results in the considered project, as well as the style of the
current text. (Note that below in this section, as a rule, exact
references are postponed to the main body of the paper.)

In fact, the concept of a Jordan scheme goes back to a few
publications and recent presentations by Peter J. Cameron. The main
source of such objects is the symmetrization of AS's (homogeneous
CC's in alternative terminology). Assume that \({\frak X}=(\Omega,
  {\cal R})\) is an arbitrary AS, \({\cal R}=\{R_0, R_1, \ldots,
  R_d\}\). Recall that \(R_i^T = \{(y,x) \mid (x,y)\in R_i\}\) is also a
class (basic relation) of \({\frak X}\). If \(R_i\) is an antisymmetric
relation, then \(R_i^T\cap R_i = \emptyset\); otherwise, \(R_i^T =
  R_i\). Denote by \(\tilde{\frak X}=(\Omega, \correction{\widetilde{\mathcal R}})\) the
symmetrization of \({\frak X}\). Here \(\tilde {\cal R} = \{\tilde
  R_i\mid R_i\in{\cal R}\}\), while \(\tilde R_i=R_i\cup R_i^T\) for all
\(i\), \(0\le i\le d\). If \({\frak X}\) is a symmetric AS then \(\tilde
  {\frak X} = {\frak X}\) and thus \(\tilde {\frak X}\) is an AS and  hence
also a Jordan scheme. Otherwise, \(\tilde {\frak X}\) is not
obligatorily an AS, however it is easy to observe that \(\tilde
  {\frak X}\) is still a JS. Such a JS will be called the
symmetrization of the AS \({\frak X}\). Let us call such a
symmetrization a \emph{non-proper JS}.

In the terms introduced above, the main question posed by P. Cameron
is: Are there examples of \emph{proper Jordan schemes}, that is, JS's
which do not appear as symmetrizations of suitable AS's.

In this text we provide an affirmative answer to this question,
introducing a number of such objects of small orders, as well as
some infinite classes of proper JS's. To the best of our knowledge
this is a pioneering result; no such structures were known before.

The present text is quite long, more than one hundred pages in the
selected format, supplemented by an extensive bibliography.

Sections \ref{orgf276f8a} and  \ref{org4a0fd08} fulfill the role
of preliminaries, introducing basic facts in three areas: AGT,
computer algebra, and the newly developed theory of \correction{Jordan schemes}. The presentation here is not self-contained. Some extra facts
are postponed to further sections. A lot of classical material with
more details can be borrowed from a number of provided sources in
the three touched areas. We will also refer a few times to the paper
\cite{KliMR} in preparation, where a more formal, refined version of
the current text will be presented.

The first, in fact the smallest, example of a proper JS, denoted by
\(J_{15}\), has order 15, rank 5, its automorphism group has order 12
and is isomorphic to the alternating group \(A_4\). The group has
orbits of lengths 3 and 12 on the vertex set \(\Omega_{15}\). 

The proper JS \(J_{15}\) is subject of an unproportionally detailed
consideration, see Sections \ref{orgeb5dfbd}-\ref{org4553fba}. We start
from the consideration of computer data, related to one of the two
so-called Siamese color graphs of order 15. Both structures were
discovered by the author S.R. in his Ph.D. thesis, fulfilled under
the guidance of M.K. The automorphism group of this graph is \((A_4,
  \Omega_{15})\). The proper JS \(J_{15}\) appears as a merging of colors
in the CC associated to \((A_4, \Omega_{15})\). At the beginning the
structure appears quite suddenly, like a jack-in-the-box.

The ongoing further presentation related to this new structure
\(J_{15}\) is arranged deliberately at a particularly slow pace. Each
step of the considerations is supplied by related computer data and,
in addition, by diagrams of the appearing graphs.

In Section \ref{orgaf4339e} the antipodal distance transitive graph
\(\Delta\) of valency 4 and diameter 3 is treated. This graph \(\Delta
  = L(\Pi)\) is the line graph of the Petersen graph \(\Pi\); it appears
as a 3-fold cover of the complete graph \(K_5\). Classical properties
of the number 6 are extensively exploited. This allows to construct
a non-proper rank 5 Jordan scheme \(NJ_{15}\) as the symmetrization of a
non-commutative imprimitive rank 6 AS \({\cal M}_{15}\) of
order 15. It turns out that \(\aut(NJ_{15}) = \aut({\cal M}_{15})
  \cong A_5\); here \(A_5\) is the smallest non-abelian simple group
of order 60. After revisiting the AS \({\cal M}_{15}\) in Section
\ref{org105b323}, the Siamese combinatorial structures of order 15 are
discussed in Section \ref{org4e457e4}. Though in the context of the
current text this is a kind of deviation, a brief acquaintance of
the reader with the Siamese combinatorial objects may shed extra
light on the origins of the discovery to which the whole text is
devoted. 

In Section \ref{org8e3100c} a Cayley color graph over the group \(A_4\) of
order 12 is considered. In its terms an auxiliary structure \(Y_{12}\)
is defined, consisting of three copies \(TT_i\), \(0\le i\le 2\), which
share a common spread \(4\circ K_3\), and the identity relation
\(\id_{12}\). Here the graphs \(TT_i\) are isomorphic copies of the
truncated tetrahedron of valency 3.

In Section \ref{orga11617b} the intransitive action \((A_4, \Omega_{15})\),
as well as its related rank 21 CC \(Y_{15}\) are reconsidered once
more. The symmetrization \(\tilde Y_{15}\) of rank 13  is
discussed. In the next Section \ref{org94d2fe9} in terms of \(\tilde
  Y_{15}\) one more color graph \(Pre_{15}\), called the pregraph of
order 15, is introduced. Its vertex set \(\Omega_{15}\) is split to
the continent \(\Omega_1\) of size 12 and the island \(\Omega_2\) of
size 3. The colors are: the identity relations \(\id_{12}\) and
\(\id_3\) on the continent and island, respectively, the graph \(K_3\)
on the island, the graphs
\(4\circ K_3\), \(TT_i\), \(0\le i\le 2\) on the continent, which is
identified with \(Y_{12}\). Altogether already seven colors
(relations) are defined. In addition, three more relations, called
bridges, denoted by \(Br_i\), \(0\le i \le 2\), are defined. The union
of these disjoint relations corresponds to the complete bipartite
graph \(K_{3,12}\) with the island and continent as vertex sets. In
each bridge any vertex on the island has valency 4, while any vertex
on the continent has valency 1. Thus the defined pregraph \(Pre_{15}\)
has rank 10. All further considerations are fulfilled in the
framework of this pregraph. 

In Section \ref{orgaf6cddd} we describe switching procedures \(Sw_i\)
between \(NJ_{15}\) and \(J_{15}\). For example, in the switching \(Sw_0\)
the induced graphs on the continent and the island are not changed,
nor is the bridge \(Br_0\), while the bridges \(Br_1\) and \(Br_2\) are
simply transposed. The switchings \(Sw_1\) and \(Sw_2\) are defined
similarly. It turns out that in the framework of \(Pre_{15}\) six
color graphs of rank 5 are defined as mergings of colors of the
pregraph. Three of these graphs are isomorphic to \(NJ_{15}\), the
other three to \(J_{15}\). Each of the three defined switchings maps
a non-proper Jordan scheme to a proper one. A sketch of a formal
proof of this crucial fact is given in Section \ref{org4553fba}.

We wish to stress that the huge part of the text from Sections
\ref{orgeb5dfbd} to \ref{org5fceb33} is designed as an instructive example for a
reader who is sufficiently motivated and patient. Such a reader will
be able to follow how \(J_{15}\) was initially discovered, how,
relying on data obtained by computer and on a redundant amount of
visual aids based on these data, the initial ideas for the future
general proof were formed.

The next order for which we know a proper Jordan scheme is 24. The
corresponding structure \(J_{24}\) shares some common features with
\(J_{15}\). It has rank 5, three basic distance regular graphs of
valency 7, each isomorphic to the classical Klein graph \(Kle_{24}\)
and having a common spread of type \(8\circ K_3\). We really enjoy
this structure \(J_{24}\), and in Section \ref{org5e0938c} share our
pleasure with the reader, constructing the graph \(Kle_{24}\) from
scratch and repeating the proof of its uniqueness up to isomorphism,
due to A. Jurišić (1995). This allows to present in Section
\ref{org508eb9e}  a self-contained description of \(J_{24}\) and to
investigate its symmetry. As in the first part of the text, visual
images play a significant role in the justifications. Nevertheless,
hopefully the reader this time will be more tolerant to such mode of
reasoning: fewer diagrams, and each diagram plays a concrete clearer role,
which can be comprehended immediately. 

Sections \ref{org68efaa1} to \ref{org495439e} (about 15 pages) form the
main part of the current text. The style of presentation here
differs essentially from the Sections \ref{orgeb5dfbd} to \ref{org508eb9e},
following the usual standards of a research paper. In Section
\ref{org68efaa1} a general formulation of the model "island and continent"
is described for orders \(n=3(m+1)\), \(3|(m-1)\). The starting
structure is a suitable rank 6 non-commutative imprimitive AS with
three thin (valency 1) relations and three thick relations of
valency \(m\), each defining an antipodal distance regular cover of
the complete graph \(K_{m+1}\). It is proved that the symmetrization
of this AS \({\frak X}_n\) defines a non-proper Jordan scheme \(NJ_n\)
of order \(n\). The selection of island corresponds to the selection
of a fiber, that is a connected component of the spread \((m+1)\circ
  K_3\) of the AS \({\frak X}_n\). 
 The union of the remaining fibers forms the continent. In
Section \ref{org68efaa1}, relying on such a selection in \({\frak X}_n\), a
pregraph \(Pre_n\) is defined, still having ten colors. Some
properties of \(Pre_n\) and its automorphism groups are formulated and
justified. This allows in Section \ref{orgd4c6594} to formulate and prove
conditions for the existence of proper Jordan schemes of order
\(n\). In particular, Theorem \ref{org54afef5} and Corollary
\ref{org2853378} claim one of the central results of the current text:
the existence of three combinatorially isomorphic proper Jordan
schemes of rank 5 and order \(n\).

Now in Section \ref{org3941a9a} an infinite series of non-commutative rank
6 AS's of order \(n\) and related proper Jordan schemes of rank 5 and
order \(n\) are constructed. Here, \(n=3(q+1)\), where \(q\equiv 1\mod 3\)
is a prime power. The rank 6 AS \({\frak X}_n\) appears as the
centralizer algebra of a suitable imprimitive action of
\(PSL(2,q)\). The corresponding results are Theorem \ref{orga416f9e} and
Corollary \ref{orgf5082e9}. The cases \(q=4,7\) lead to the proper JS's
\(J_{15}\) and \(J_{24}\), respectively. A few next values of \(n\) are
also mentioned.

Section \ref{org30a0d76} deals with the further generalization of already
formulated and rigorously proved results for proper rank 5 JS's. 
\correction{This time we start} from a rank \(2l\) non-commutative AS of order
\(n=l(m+1)\), with \(l\) thin and \(l\) thick relations, the latter of
valency \(m\). A suitable multiplication table for such AS \({\cal
  M}_n\) is presented. A more general mode of switching is defined. A
more general theoretical result is formulated: still proper Jordan
schemes of order \(n\) do appear. All details are postponed to
\cite{KliMR}. Relying on this generalization, in Section \ref{orgb36d93b}
a list of proper Jordan schemes of order \(n\) is given, where \(n\)
varies from 63 to 748. As a rule, the automorphism group of such a
scheme \(J_n\) is isomorphic to \(AGL(1,m)\) of order \(m(m-1)\). Note
that \(A_4\cong AGL(1,4)\), while \(\aut(J_{24})\) has order 42.

In fact, Section \ref{orgb36d93b} starts from the consideration of Siamese
color graphs of order 40. In this framework two proper Jordan AS's
of rank 7 and order 40 are constructed, both with valencies 1, 1, 2,
9, 9, 9, 9. 

In Section \ref{org682bc3a} a prolific construction for proper rank
5 Jordan schemes of order \(n={{3^d+1}\choose {2}}\), \(d\in{\mathbb
  N}\), is outlined. The construction is based on a particular case of
the classical \strike{WFDF} strongly regular graphs (SRG's), due to
W.D. Wallis and D. Fon-Der-Flaass \correction{(a WFDF-graph, for short)}. Three basic graphs here are
primitive SRG's of the same valency \(k=3^{d-1}\cdot \frac{3^d-1}{2}\)
and with \(\lambda=\mu={{3^{d-1}}\choose 2}\).  Again, a more detailed
consideration, including accurate proofs, is postponed to
\cite{KliMR}. Section \ref{orgcbc9bab} deals with the smallest case \(d=1\),
which corresponds to the symmetrization of the centralizer algebra
of the regular action of the smallest non-abelian group (of order
6). The appearing non-proper JS of order 6 was known already to
B.V. Shah (1959). We provide four isomorphic models of this almost
trivial structure \(NJ_6\).

In Section \ref{orgf5cb3c0} the first non-trivial case \(d=2\), leading to
schemes of order 45, is treated, mainly with the aid of a
computer. Here the theoretical results from Section \ref{org3f2a5c7} serve
just as a motivation of research. In fact, we rely on the full
catalogue of SRG's with the parameters \((45, 12, 3, 3)\), due to Ted
Spence. Up to isomomorphism, there are 78 such SRG's. Some natural
conditions are considered for these SRG's to appear as 
basic graphs of a JS of rank 5.

Three graphs, together with the required spread \(5\circ K_9\) in
their complements, passed a formulated test, which was checked with
the aid of a computer. Though computations were not finished, at
least 340 Jordan schemes of rank 5 (up to isomorphism) of order 45
were enumerated this way. Striking news are that all of them are
proper JS's.

Section \ref{org5dcbc7f} reports on the computer-aided enumeration of
Jordan schemes of small orders \(n\). At this stage the developed
computer programs are not sophisticated. An almost naive, brute
force version was tested for orders \(n=8,9,10\). The cases \(n=11,13\)
(where \(n\) is a prime) were covered, using some additional simple
theoretical arguments. Finally, the cases \(n=12,14\) required more
efforts; here computations still are not finished. The current result
is that there is no proper scheme of order $n$,
$n\in [1,11]\cup\{13\}$. Our wild guess is that there are no proper
schemes of orders $12$ and $14$ neither; this is subject to rigorous
confirmation in the future.

The final Sections \ref{org495439e}-\ref{org7fa6e54} (about 15 pages)
are very specific even for the selected genre of the entire
text. Diverse miscellanea are postponed from the main body of the
paper to Section \ref{org495439e}. Among other discussed topics
here the reader faces issues such as:
\begin{itemize}
\item computational aspects of diverse groups associated to color graphs;
\item difference in formats used in {\sf GAP} and COCO;
\item remarkable properties of some actions of \(A_5\) as well as extra
beauties of the Klein graph \(Kle_{24}\);
\item the crucial role of  Hoffman colorings of the graphs under consideration;
\item delicate issues in the  accepted terminology and notation in AGT.
\end{itemize}

Section \ref{org03b7322} presents a short historical survey, which is
literally split into a few subsections. In Subsection
\ref{org7b1a622} the authors' acquaintance and accumulated vision of
the history of Jordan algebras is reflected. We think this might be
of help for all experts in AGT. Subsection \ref{orgf71f8d9}, devoted
to \correction{CCs} and WL-stabilization, is oriented mainly to the experts in
\correction{algebra} and Jordan algebras, in particular those who may be
challenged to enter to AGT. Subsection \ref{org94cd7b7} touches
Siamese combinatorial objects and Tatra AS's. Though both these
concepts remain in a relative shadow in the current text, familiarity
with them reveals the origins of the introduced discovery, which are
pretty clear to the authors, while they were not so visible to the
reader. 
Finally, Subsection \ref{org6b70302} is briefly summing up the
significant stages in the reported project: from many years of
implicit preparation to the last half year brain-storming explicit
attack. 

The concluding discussion in Section \ref{org7fa6e54} is quite informal,
touching a lot of issues related to many facets in the text, from
its style and genre to open questions. Probably, it is a good idea
for any reader to look first through the sections
\ref{orgb33f855} and \ref{org7fa6e54}, postponing the decision how to
digest the entire text.

A few more remarks: We frequently use abbreviations. Usually they
are evidently introduced at the first occurence, like AGT; a few
times they are clear in the considered context.

The reader is already able to understand unusual features of the
created text. This is \correction{neither} a regular paper nor a small
book. Probably closer to an extended preprint. We prefer to
attribute to this genre the word "essay" At this stage the text will
be posted on arXiv, hopefully a couple of updates will be created,
depending on many factors.  

Last but not least in the attempts to accurately attribute the style
of the created essay is reference to the concept of
\href{https://en.wikipedia.org/wiki/Abductive\_reasoning}{abductive reasoning}. Indeed, we start from the construction of
a lucky set of pioneering examples. After that we seek for the most
simple and likely explanations of the detected nice mathematical
creations. At the beginning, that is in the current essay, too
accurate and rigorous argumentation is regarded, as a rule, a bit
annoying and not always appropriate.

Of course, finally, being mathematicians, we switch to the normal
style of reasonings, to be presented in \cite{KliMR}.

One may regard George Pólya among the creators of such an approach
in mathematics. Our interpretation of abduction seems to be a bit
less accepted in mathematics and closer to diverse branches of
natural sciences.

We finish this Section \ref{orgb33f855} with a kind of informal
disclaimer. The roles, distribution of activities, vision of the
desired structure of the text, attention to its diverse parts, are
pretty different between the members of our trio. In principle, it
was possible to write a kind of table, showing somehow the inputs of
the authors to each of the sections. It seems that both the
preparation and the consumption of the obtained matrix would be
annoying to everyone involved.

This is why we finally claim that the entire trio feels collective
\correction{involvement} for the major part of the created text in spite of the fact of evident
distinctions in the authors inputs and attention to its diverse
parts.

Disclaimer of the second author: Much of the
  material of this essay was prepared by my coauthors. My
  participation in the text writing was focused on the Sections 3.2,
  12, 16, 18 and 20.
\section{Preliminaries from AGT and computer algebra}
\label{sec:org2d1b29f}
\label{orgf276f8a}

It is assumed that the major part of the expected audience of the
current text is familiar with the background concepts from
AGT. Nevertheless, in order to make the preprint reasonably
self-contained, below  a brief outline of the most significant
notions is provided. Many texts, which were written at different
times in diverse styles, might be quite helpful. Here is a small
sample: \cite{Bai04}, \cite{BanI84}, \cite{Cam99}, \cite{KliG15},
\cite{KliPRWZA10}, \cite{KliRW09}. A number of other significant texts
will be mentioned later on in an ad hoc manner. 

Our initial concept is a finite graph \(\Gamma\) with the \emph{vertex set} \(\Omega\)
of cardinality \(|\Omega| = n\). Sometimes for a concrete value of \(n\)
the set will be denoted as \(\Omega_n\) to stress which specific set
is currently considered. The number \(n\) is usually called the
\emph{order} of \(\Gamma\). A \emph{directed graph} \(\Gamma=(\Omega,R)\) is a
pair consisting of the vertex set \(\Omega\) and the \emph{arc set}
\(R\subseteq \Omega^2\), where \(\Omega^2 = \{(x,y)\mid
  x,y\in\Omega\}\). 

An \emph{undirected graph} \(\Gamma=(\Omega, E)\) consists of the vertex
set \(\Omega\) and the \emph{edge set} \(E\subseteq\setchoose \Omega 2\),
here \(\setchoose \Omega 2 = \{\{x,y\}\mid x,y\in \Omega, x\neq
  y\}\). An arc \((x,x)\) of a directed graph \(\Gamma\) is called a
\emph{loop}. In case when \(\Gamma\) does not contain loops it is sometimes
convenient to identify \(\Gamma = (\Omega, R\cup R^T)\) with the
undirected graph \(\tilde \Gamma=(\Omega, E)\), where \(E=\{\{x,y\}\mid
  (x,y)\in R\}\) for \(\Gamma=(\Omega,R)\). The graph \(\tilde\Gamma\) is
usually called the \emph{symmetrization} of \(\Gamma=(\Omega, R)\). 

We also consider a \emph{full color graph} \(\Gamma=(\Omega, {\cal R})\),
where \({\cal R} = \{ R_0, R_1, \ldots, R_d\}\) is a partition of the
set \(\Omega^2\). Then the number \(d+1\) of different \emph{binary
relations} \(R_i\), \(0\le i \le d\), is called the \emph{rank} of
\(\Gamma\). Frequently the rank of a color graph will be denoted by
\(r\). Each graph \(\Gamma_i=(\Omega, R_i)\), \(0\le i \le d\), will be
called a \emph{unicolor graph} of the considered color graph \(\Gamma\).

To each directed graph \(\Gamma = (\Omega, R)\) we associate its
\emph{adjacency matrix} \(A(\Gamma)\) of order \(n\). It is defined as
follows: \(A(\Gamma) = (a_{ij})_{1\le i,j\le n}\), where
\[
  a_{ij} = \begin{cases}
  0 & \mbox{if $(i,j)\not\in R$}\\
  1 & \mbox{if $(i,j)\in R$.}
  \end{cases}
  \]
For a color graph \(\Gamma=(\Omega,{\cal R})\), defined as above, we may also
consider its adjacency matrix
\[
  A(\Gamma) = \sum_{i=0}^d iA(\Gamma_i),
  \]
where \(\Gamma_i = (\Omega, R_i)\) is a unicolor graph of \(\Gamma\),
\(0\le i \le d\).

In many situations, when this does not imply any misunderstandings,
it might be convenient to identify  the graph \(\Gamma\),
its arc set \(R\), and its adjacency matrix \(A(\Gamma)\), freely
switching between these three parallel languages.

A \emph{coherent configuration} (CC) \({\frak X}=(\Omega, {\cal R})\) is a
color graph of rank \(r\) with the set \({\cal R}=\{R_0, R_1, \ldots,
  R_{r-1}\}\) of colors (relations) which satisfies the following extra conditions: 
\begin{description}
\item[CC1] There exists a subset $I_0$ of the set $I=\{0,\ldots, r-1\}$
   of indices of colors, such that $\bigcup_{i\in I_0} R_i = \operatorname{Id}_n$,
   here $\operatorname{Id}_n=\{(x,x)\mid x\in \Omega\}$ is the {\em identity
   relation} on the $n$-element set $\Omega$;
\item[CC2] For each $i\in I$ there exists $i'\in I$ such that $R_{i'}=R_i^T$,
    where $R_i^T = \{(y,x)\mid (x,y)\in R_i\}$;
\item[CC3] There exists a set of constant numbers $p_{ij}^k$, $i,j,k\in I$,
    such that 
    \[
    \left| \left\{ z \in \Omega \mid
      (x,z)\in R_i \wedge (z,y)\in R_j \right\} \right |
      = p_{ij}^k
    \]
    provided that $(x,y)\in R_k$.
\end{description}
The numbers \(p_{ij}^k\) are usually called \emph{intersection numbers} (or
\emph{structure constants}) of the CC  \({\frak X}=(\Omega, {\cal R})\).

To each CC  \({\frak X}=(\Omega, {\cal R})\) we associate an algebra
\(W\) over the field \(\mathbb R\) or \(\mathbb C\) which has a standard
basis consisting of the matrices \(A_i=A(\Gamma_i)\), \(i\in I\),
denoted by \(W=\left<A_0, \ldots, A_{r-1}\right>\). The elements of \(W\)
are all linear combinations of the basic matrices of \({\frak X}\). The
definitions of a CC and its corresponding coherent algebra \(W\) (of
order \(n\)) imply that each coherent algebra \(W\) contains the
identity matrix \({\mathbb I}={\mathbb I}_n\) of order \(n\), the matrix
\({\mathbb J}_n\) with all entries equal to 1. A coherent algebra
\(W\) is also closed with respect to transposition of
matrices. Besides the usual multiplication, the algebra \(W\) is also
closed with respect to \emph{Schur-Hadamard} (SH) product of matrices,
that is, entrywise multiplication.

The algebra \(W\) defined above, relying on a given CC \(\frak X\), is
called a \emph{coherent algebra}. Clearly, a coherent algebra can be
defined axiomatically, thus CCs and coherent algebras are equivalent
concepts.

\correction{Given} an arbitrary CC  \({\frak X}=(\Omega, {\cal R})\), the identity
relation \(\id_n\) appears as a union of identity relations on
disjoint subsets
of \(\Omega\), which uniquely define a partition of \(\Omega\). These
subsets are called \emph{fibers} of \(\frak X\).

A CC with one fiber is called an \emph{association scheme} (AS) or a
\emph{homogeneous CC}. A coherent algebra corresponding to an AS is
called its \emph{adjacency algebra}. An AS
\({\frak X}=(\Omega, {\cal R})\) is called \emph{symmetric} if all its
basic relations are symmetric. The adjacency algebra of a symmetric
AS is usually called its \emph{Bose-Mesner algebra} (BM-algebra).  Each symmetric AS has a
commutative adjacency algebra, however, there are some classes of
non-symmetric ASs which also have commutative adjacency algebras. 

A classical result by Higman \cite{Hig75} guarantees that each AS of rank at most
5 is commutative.

In the case of ASs usually the labelling of basic relations starts
from \(R_0=\id_n\). 

One of the main sources of CCs comes from permutation groups. Let
\((G,\Omega)\) be a \emph{permutation group} acting on the finite set
\(\Omega\), that is, \(G\) is a subgroup of the symmetric group
\(S_n=S(\Omega)\), consisting of all permutations of the set
\(\Omega\). Following H. Wielandt, let us define the set
\(2-orb(G,\Omega) = \{R_i\mid i\in I\}\). By definition,
\(2-orb(G,\Omega)\) is the partition of the cartesian square
\(\Omega^2\) into \emph{2-orbits} of \((G,\Omega)\). Each 2-orbit \(R_i\)
appears as \(\{(x,y)^g\mid g\in G\}\) for a suitable pair
\((x,y)\in\Omega^2\). Here for \(g\in G\), \((x,y)^g = (x^g,y^g)\), where
\(x^g\) is the \emph{image} of \(x\in\Omega\) under the action of \(g\in
  G\). The 2-orbits of the kind \(\{(x,x)^g\mid g\in G\}\) are called
\emph{reflexive} 2-orbits, all others are \emph{non-reflexive} 2-orbits. The
cardinality \(r=|I|\) of the set of 2-orbits of \((G,\Omega)\) is called
the \emph{rank} of \((G,\Omega)\).

For a transitive permutation group, which has exactly one reflexive
2-orbit, the rank coincides with the \correction{number of orbits of the point stabilizer \(G_x=\{g\in G\mid x^g = x\}\) where \(x\in\Omega\) is an arbitrary point. } In such a case the 2-orbits of
\((G,\Omega)\) are frequently called its \emph{orbitals.}

It is easy to check that the pair \((\Omega, 2-orb(G,\Omega))\) forms
a CC. Each such CC is called \emph{Schurian} CC, otherwise it is called
\emph{non-Schurian}. According to \cite{KliZA17} all CCs of order \(n\le
  13\) are Schurian, while there exists an example of a non-Schurian CC
of order 14. Similar wording is used for the particular case of ASs.

The introduced concepts can be expressed in the parallel language of
matrices. Recall that to an arbitrary permutation \(g\in S_n\) we may
associate a \emph{permutation matrix} \(P=P(g)\) of order \(n\), with rows
and columns labeled by elements of \(\Omega\). Here for \((\alpha,
  \beta)\in\Omega^2\), the entry \(P_{\alpha,\beta}\) is defined as 
\[
  P_{\alpha,\beta} = \begin{cases}
  1 &  \mbox{if $\beta = \alpha^g$}\\
  0 & \mbox{otherwise.}
  \end{cases}
  \]
Clearly, each permutation matrix is a (0,1)-matrix containing
exactly one entry equal to 1 in each row and each column. 

Now we may define the \emph{centralizer algebra} 
\[
  \frak{V}(G,\Omega) = \{A\in M_n(F)\mid AP(g)=P(g)A, \forall
  g\in G\},
  \]
where \(F\) is the field under consideration, usually \(\mathbb R\) or
\(\mathbb C\), and \(M_n(F)\) is the algebra of all square matrices of
order \(n\) over \(F\). 

It is easy to check that the defined set \(\frak{V}(G,\Omega)\) of all
square matrices of order \(n\) which are commuting with all
permutation matrices \(P(g)\), \(g\in G\), is indeed an
algebra. Moreover, such an algebra has a \emph{standard basis} consisting
of the adjacency matrices \(A(\Gamma_i)\), \(\Gamma_i=(\Omega, R_i)\) of
the basic graphs \(\Gamma_i\), defined by the 2-orbits \(R_i\), \(i\in
  I\), of \((G,\Omega)\). Thus, the CC defined by 2-orbits of
\((G,\Omega)\) and the centralizer algebra \(\frak{V}(G,\Omega)\) are
equivalent concepts. Therefore we can speak of Schurian coherent
algebras, that is, coherent algebras corresponding to Schurian
coherent configurations, in other words, centralizer algebras of
suitable permutation groups.

Because each CC  \({\frak X}=(\Omega, {\cal R})\) is a particular case
of a color graph, we may consider its symmetrization \correction{\(\tilde{\frak
  X}\).} As we will see later on, the symmetrization of CC is not
obligatorily a CC.

The concept of a coherent algebra \(W\) can be introduced as a
non-empty set of square matrices of order \(n\) (over a field, say
\({\mathbb R}\)) which is closed with respect to a few operations of
arity 0, 1, and 2. This implies that for an arbitrary non-empty set
\(X\) of matrices there exists the unique \emph{smallest} coherent algebra
\(W=WL(X)\) which contains all matrices from \(X\). This coherent algebra  \(WL(X)\) is
nowadays called the \emph{coherent closure} of \(X\) or the
\emph{Weisfeiler-Leman closure} (briefly WL-closure) of \(X\). There exists
an efficient (polynomial time) algorithm for the computation of
\(WL(X)\) from \(X\), which goes back to \cite{WeiL68}. There is a
number of texts where algorithmic aspects of this procedure, as well
as its complexity, are discussed, see e.g., \cite{KliRRT99},
\cite{BabCKP10}. We note also that a reasonable historical survey of
related activities appears in \cite{FarKM94}.

The introduced concept of WL-closure allows, in particular, to
provide clear and compact definitions of a few main classes of
structures in AGT. First, we will call a unicolor graph \(\Gamma\)
\emph{coherent} if it is not "ruined" by WL-stabilization; in other words
if \(\Gamma\) is a basic graph of \(WL(\Gamma)\).

A regular connected graph \(\Gamma\) of diameter \(d\) is called a
\emph{distance regular graph} (DRG) if the rank of \(WL(\Gamma)\) is equal
to \(d+1\). A DRG of diameter 2 is called \emph{strongly regular}
(SRG). Usually to a DRG \(\Gamma\) a full or short set of parameters
is associated, which is called the \emph{intersection array} of
\(\Gamma\). The advantage of the intersection array is, in particular,
that the 3-dimensional tensor of structure constants of
\(WL(\Gamma)\) can be calculated relying on the intersection array
(which may be described with the aid of a special diagram), see
\cite{BroCN89} for details.

For an SRG \(\Gamma\) we consider its short \emph{parameter set} in a
notation which goes back to R.C. Bose \cite{Bos63}. This is
\((n,k,\lambda,\mu)\), where \(n=|\Omega|\) is the order of \(\Gamma\),
\(k\) its valency, \(\lambda\) and \(\mu\) are the valencies of arbitrary
edges and non-edges, respectively.

A DRG \(\Gamma\) is \emph{primitive} if all non-reflexive basic graphs of
\(WL(\Gamma)\) are connected, otherwise, \(\Gamma\) is imprimitive. Note
that there is a tradition that for the connected imprimitive SRG
\(\Gamma\) its complement \(\bar\Gamma\) is also regarded as an
SRG. Clearly such an SRG \(\bar \Gamma\) has the form \(t\circ K_m\),
\(n=t\cdot m\), where \(t\circ K_m\) denotes the disjoint union of \(t\)
copies of the complete graph \(K_m\) with \(m\) vertices.

As was mentioned in Section \ref{orgb33f855}, the current
preprint heavily depends on the use of a few computer packages.

The oldest one of them is COCO (COherent COnfigurations); it was
created in 1990-91 by Igor Faradžev with algorithmic and logistical
support of M.K. The main programs are written in C, while a few ones
even in assembler. Its definite advantage is very simple behavior of
a user: just "pushing buttons". Unfortunately, due to the collapse
of the USSR, the full version of COCO, as it was "dreamed" by the
authors, never was created. The main programs, appearing in COCO,
are discussed with enough details in \cite{FarKM94}, though the name
COCO does not appear there. Formally, the package was announced in
\cite{FarK91}, a short discussion appears in
\cite{KliPRWZA10}. Note that the source files of COCO can be
downloaded from the homepage \cite{Bro} of A. Brouwer.

For a long while a few younger colleagues of M.K. have been thinking of
the creation of a modern version of COCO, denoted by COCO II. There
is a common understanding that such a version should be finally
regarded as a computer package associated with the computer algebra
system {\sf GAP}, see \cite{GAP}, and then will also rely on
standard tools in {\sf GAP}, related to processing of graphs
\cite{GRAPE}, and B. McKay's classical program nauty \cite{nauty}.

The ideology of the future COCO II is also briefly mentioned in
\cite{KliPRWZA10}. It seems that the current version COCO IIR,
created by S.R., is closer to the one which may be exploited by
other users. It is a pity, however, still the project COCO IIR is
not finished. One of the reasons is the absence of a permanent
employment of S.R., see the end of the paper for more comments.

In this context we have to mention a few groups which will be
frequently considered in the current text. 

First, for a color graph \({\frak X} = (\Omega, {\cal R})\), \({\cal R} =
  \{R_i \mid i \in I\}\), the \emph{automorphism group} \(\aut({\frak X}) =
  \bigcap_{i\in I} \aut(\Gamma_i)\) is defined, where \(\aut(\Gamma_i) = \{g\in
  S(\Omega) \mid R_i^g = R_i\}\), \(\Gamma_i = (\Omega, R_i)\). The
\emph{color group} is defined as \(\caut({\frak X}) = \{g\in S(\Omega)\mid {\cal R}^g =
  {\cal R}\}\). Here \({\cal R}^g={\cal R}\) means that the permutation
\(g\in S(\Omega)\) preserves the set of relations \({\cal R}\), that is
for each \(i\in I\) the image \(R_i^g=R_j\) for a suitable \(j\in
  I\). Clearly, \(\aut(\frak X) \normal \caut(\frak X)\). 

Finally for an algebra \({\frak X}\) (coherent or Jordan) with the
tensor of structure constants \(p_{ij}^k\) we also consider the
\correction{\emph{group of algebraic automorphisms}} \(\aaut({\frak X})\) which acts on the set \(I\) of indices
of basic elements of \({\frak X}\), aka basic graphs/relations of the
corresponding relational configuration. Here,
\[
  \aaut({\frak X}) = \{f\in S(I)\mid p_{i^fj^f}^{k^f} = p_{ij}^k,
  i,j,k\in I\}.
  \]

It is easy to understand that the quotient group \(\caut({\frak
  X})/\aut({\frak X})\) can be embedded as a subgroup of \(\aaut({\frak
  X})\). If \(\aaut({\frak X})\) also contains elements which do not
appear via such an embedding, then they are called \emph{proper algebraic
automorphisms}. 

Note that COCO contains a program for the computation of \(\aut({\frak
  X})\). It works surprisingly efficiently for primitive
CCs. Unfortunately, programs for the calculation of \(\caut({\frak
  X})\) and \(\aaut({\frak X})\) were not created. 

To help the reader to get a very simple idea of computer-aided
activities and further manipulation with the obtained data, we
provide below a simple example. Some results, fulfilled in COCO,
intentionally are described in a manner which should be more
comfortable for the reader. 

\begin{examp}
\begin{itemize}
\item We start from the dihedral group \(D_5=\left<g_1,g_2\right>\), where
\(g_1 = (0,1,2,3,4)\), \(g_2=(1,4)(2,3)\).
\item Using the command \emph{ind} of COCO, we obtain an intransitive group
\((G,\Omega)\) of degree 10. As an abstract group, \(G\cong
    D_5\). Here \(G=\left<\tilde g_1, \tilde g_2\right>\), where \(\tilde
    g_1 = (0,1,2,3,4)(5,6,7,8,9)\), \(\tilde g_2=(1,4)(2,3)(6,9)(7,8)\),
\(\Omega=[0,9]\), that is the set of integers from 0 to 9.
\item Now with the aid of the command \emph{cgr} of COCO we describe all
2-orbits of \((G,\Omega)\). It turns out that \(|2-orb(G,\Omega)| =
    12\); the information about the representatives of this set looks
as follows:
\begin{center}
\begin{tabular}{llr}
R\(_{\text{i}}\) & (x,y)\(\in\) R\(_{\text{i}}\) & val(R\(_{\text{i}}\))\\
\hline
R\(_{\text{0}}\) & (0,0) & 1\\
R\(_{\text{1}}\) & (0,1) & 2\\
R\(_{\text{2}}\) & (0,2) & 2\\
R\(_{\text{3}}\) & (0,5) & 1\\
R\(_{\text{4}}\) & (0,6) & 2\\
R\(_{\text{5}}\) & (0,7) & 2\\
R\(_{\text{6}}\) & (5,0) & 1\\
R\(_{\text{7}}\) & (5,1) & 2\\
R\(_{\text{8}}\) & (5,2) & 2\\
R\(_{\text{9}}\) & (5,5) & 1\\
R\(_{\text{10}}\) & (5,6) & 2\\
R\(_{\text{11}}\) & (5,7) & 2\\
\hline
\end{tabular}
\end{center}

Here \(R_0, R_9\) are reflexive 2-orbits, the pairs \((R_3,R_6)\), \((R_4,
    R_7)\), \((R_5,R_8)\) provide antisymmetric 2-orbits, all other
2-orbits are symmetric.
\item In principle, COCO does not provide in the standard command
sequence the lists of all 2-orbits, though this job can be
obtained via a sequence of extra inducings. For the reader's
convenience we provide lists of some of 2-orbits, as follows:

\begin{align*}
R_1 &= \{ \{  0,  1 \},
         \{  0,  4 \},
         \{  1,  2 \},
         \{  2,  3 \},
         \{  3,  4 \}\}, \\
R_2 &= 
       \{ \{  0,  2 \},
         \{  0,  3 \},
         \{  1,  3 \},
         \{  1,  4 \},
         \{  2,  4 \}\}, \\
R_3 &= 
   \{ (  0,  5 ),
      (  1,  6 ),
      (  2,  7 ),
      (  3,  8 ),
      (  4,  9 ) \}, \\
R_4 &= 
  \{ (  0,  6 ),
    (  0,  9 ),
    (  1,  5 ),
    (  1,  7 ),
    (  2,  6 ),
    (  2,  8 ),
    (  3,  7 ),
    (  3,  9 ),
    (  4,  5 ),
    (  4,  8 ) \}, \\
 R_{10} &= 
    \{ \{  5,  6 \},
      \{  5,  9 \},
      \{  6,  7 \},
      \{  7,  8 \},
      \{  8,  9 \}\}.
\end{align*}
\item On the next step we are using the command \emph{inm}, which provides
the intersection numbers of the rank 12 CC \(W=(\Omega,
    2-orb(G,\Omega))\). Typically, the user does not need to work
directly with this tensor; it is processed by COCO automatically.
\item Now we are interested in all coherent \emph{proper} subalgebras \({\frak
    X}_i\) of the coherent algebra defined by \(W\). Here proper means
that the rank of \({\frak X}\) is smaller than the rank of \(W\) and
larger than 2.

Recall that in the language of CCs, coherent subalgebras
correspond to \emph{mergings} (or fusions) of \(W\). 

In fact, COCO is able to determine only those mergings which are
ASs. There are also options to find only symmetric or primitive
mergings, as well as some combinations of such properties.

Such a job is fulfilled with the aid of the command \emph{sub}, the
input for which is the output of \emph{inm}. The obtained mergings are
ordered according to their rank. Each merging is described as a
partition of the set of indices \(i\in I\) of basic relations
\(R_i\). In our cases, \(I=[0,11]\).

Thus COCO returns the following ten proper mergings:

\begin{itemize}
\item m1 = \{\{0,9\},\{5,8\},\{2,10\},\{1,11\},\{3,7\},\{6,4\}\}, r=6,
\item m2 = \{\{0,9\},\{5,8\},\{2,11\},\{3,6\},\{4,7\},\{1,10\}\}, r=6,
\item m3 = \{\{0,9\},\{5,8\},,\{2,3,6,11\},\{4,7,1,10\}\}, r=4,
\item m4 = \{\{0,9\},\{5,8\},\{2,1,10,11\},\{3,6,4,7\}\}, r=4,
\item m5 = \{\{0,9\},\{5,8,3,6,4,7\},\{2,10\},\{1,11\}\}, r=4,
\item m6 = \{\{0,9\},\{5,8,3,6,7\},\{2,11\},\{1,10\}\}, r=4,
\item m7 = \{\{0,9\},\{5,8,3,6,4,7\},\{2,1,10,11\}, r=3,
\item m8 = \{\{0,9\},\{5,8\},\{2,3,6,4,7,1,10,11\}\}, r=3,
\item m9 = \{\{0,9\},\{5,8,2,10\},\{3,6,4,7,11,11\}\}, r=3,
\item m10 = \{\{0,9\},\{5,8,1,11\},\{2,3,6,4,7,10\}\}, r=3.
\end{itemize}
\end{itemize}

Denote by \({\frak X}_i\) the AS which appears as merging of \(W\)
defined by the partition \(mi\) of the set \(I\), \(1\le i \le
    10\). COCO informs us that all mergings, besides \({\frak X}_1\), are
symmetric and that the rank 3 mergings \({\frak X}_7\) and \({\frak
    X}_{10}\) are the only two primitive mergings, which are generated
by an SRG with the parameters \((10,3,0,1)\).
\begin{itemize}
\item On the last step COCO returns the order and rank of the
automorphism groups \(G_i = \aut({\frak X_i})\). Comparing the ranks
of the groups \(G_i\) with the ranks of the ASs \({\frak X}_i\), we
conclude that all obtained mergings are Schurian.

The orders of the groups obtained by the command \emph{aut} (including
the 2-closure of the starting group \(G\) denoted by \(G^{(2)}\), that
is \(G^{(2)} = \aut(W)\)), are 10, 20, 20, 240, 320, 200, 200, 120, 28800, 3840, 120.
\item At this stage COCO finishes its job. Additional information
discussed below was obtained with the aid of COCOIIR.

First, we obtain that \(\caut(W)\) has order 40, while \(\aaut(W)\)
has order 4 and is isomorphic to the elementary abelian group
\(E_4\). Analysing additional generators of \(\caut(W)\) versus \(G\),
it is easy to understand that they have a very natural
explanation: \(g_3=(0,5)(1,6)(2,7)(3,8)(4,9)\) transposes two
isomorphic fibers of \(W\), while \(g_4=(1,2,4,3)(6,7,9,8)\) together
with the induced "natural" cycle \(\tilde g_1\) generates the group
\(AGL(1,5)\) of order 20.
\item On the next step we obtain that 10 mergings of \(W\) split with
respect to \(\caut(W)\) into 9 orbits. Namely, we have the
non-trivial orbits  \(\{{\frak
    X}_7,{\frak X}_{10}\}\), while all other mergings lie in orbits of
length 1.

In principle one may get, using COCOIIR, which of the mergings are
isomorphic and which are algebraically isomorphic, though not
combinatorially isomorphic. In the considered quite simple example
the latter options are not relevant.
\item The next stage is the explanation (interpretation) of the obtained
computer-aided results, see the end of the text for a discussion
of these issues.

Currently this step is fulfilled without essential use of a
computer. A crucial possible automatical procedure would rely on
the analysis of the Hasse diagram of all obtained mergings. Then
we will need for each considered merging \({\frak X}_i\) to detect a
suitable graph \(\Delta_i\), such that \({\frak X}_i =
    WL(\Delta_i)\). Ideally, \(\Delta_i\) might be a basic graph, or a
union of a "few" basic graphs, or (not full) color graph with a
few colors. 

In this example the results were obtained in a reasonably
human-friendly creative manner. They are presented below with the
aid of diagrams. Hopefully, the reader will agree that these
diagrams add a bit to the declared flavour of essay.
\item Thus below for each merging \({\frak X}_i\) its WL-generator is
presented and sometimes discussed.

\begin{figure}
\begin{center}
\input{figure.delta.1b.tex}    
\end{center}
\caption{\(\Delta_1\)}
\end{figure}

\begin{figure}
\begin{center}
\input{figure.delta.2.tex}    
\end{center}
\caption{\(\Delta_2\)}
\end{figure}

\begin{figure}
\begin{center}
\input{figure.non-matching.tex}    
\end{center}
\caption{\(\Delta_3\)}
\end{figure}

\begin{figure}
\begin{center}
\input{figure.delta.3.tex}    
\end{center}
\caption{\(\Delta_4\)}
\end{figure}

\begin{figure}
\begin{center}
\input{figure.delta.5.tex}    
\end{center}
\caption{\(\Delta_5\)}
\end{figure}

\begin{figure}
\begin{center}
\input{figure.delta.6.tex}    
\end{center}
\caption{\(\Delta_6\)}
\end{figure}

\begin{figure}
\begin{center}
\input{figure.delta.9.tex}    
\end{center}
\caption{\(\Delta_7\)}
\end{figure}

\begin{figure}
\begin{center}
\input{figure.delta.8.tex}    
\end{center}
\caption{\(\Delta_8\)}
\end{figure}

\begin{figure}
\begin{center}
\input{figure.matching.tex}    
\end{center}
\caption{\(\Delta_9\)}
\end{figure}

\begin{figure}
\begin{center}
\input{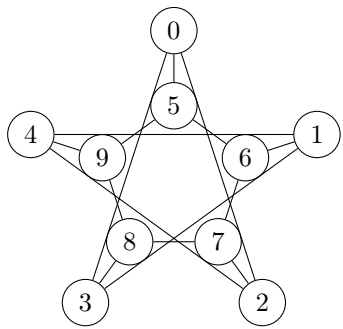}    
\end{center}
\caption{\(\Delta_{10}\)}
\end{figure}

Consideration of the diagrams immediately implies the description
of the groups \(G_i=\aut({\frak X_i})\). Indeed, we obtain that
\(G_2=D_{10}\),
\(G_3=S_2\times S_5\),
\(G_4=D_5\wr S_2\),
\(G_5 = G_6 \cong S_2\wr D_5\),
\(G_7\cong G_{10} \cong S_5\). 
\(G_8 \cong S_2\wr S_5\),
\(G_9 = S_5\wr S_2\),

Note that here \(K\times L\) is the direct product of permutation
groups \((K,\Omega_1)\) and \((L,\Omega_2)\), while \(K\wr L\) is the
wreath product of the corresponding groups. (We use the so-called
"orthodox" notation for the wreath product, due to L.A. Kalužnin,
cf \cite{FarKM94}.)

\item It remains to pay special attention to the merging \({\frak X}_1\) and its group
\(G_1\), which are less evident and thus are of reasonable
independent interest.

Recall that, according to the online catalog \cite{HanM} of small
ASs, the smallest order of a non-commutative AS, which is not
\emph{thin}, is 10. Recall that a thin AS is coming as a set of
2-orbits of a suitable regular permutation group.

It turns out that \({\frak X}_1\) is exactly this AS. It is Schurian
and it appears from the induced action of the group
\(AGL(1,5)=\left<g_1,g_3\right>\), where \(g_3=(1,2,4,3)\), on the set
\(\setchoose{[0,4]}{2}\) of 2-element subsets of \([0,4]\). Because
\(AGL(1,5)\) acts 2-transitively, this induced acction \(G_1\) is
transitive.

To adjust to the currently used labelling, we need to establish a
bijection between the sets \([0,9]\) and \(\setchoose{[0,4]}{2}\). In
presentation below it remains behind the scenes. We simply list
below the diagrams of the non-reflexive basic graphs \(\Gamma_i\) of
\({\frak X}_1\).
\end{itemize}

The non-reflexive basic graphs of \({\frak X}_1\) are: \(\Delta_1\),
\correction{\(\Delta_1^T\)}, \(\Delta_5\), \(\Delta_5^2\), \(\Delta_9\).

We observe that \({\frak X}_1\) is not symmetric, thus indeed it can
be, in principle, non-commutative.

To observe that \({\frak X}_1\) is in fact non-commutative, let us
trace paths of length 2 from vertex 0 via \(\Delta_1\) and then
\(\Delta_9\): We can reach 1 and 4. Now we walk from 0 first via
\(\Delta_9\) and then via \(\Delta_1\): We reach 2 and 3. This
observation is enough to confirm non-commutativity.
\end{examp}

\begin{remark}
The reader is welcome to compare nice and more ugly diagrams
\(\Delta_7\) and \(\Delta_{10}\) of the Petersen graph. With more
attention this situation is discussed in \cite{Sta18}, see also at
the end of our essay.
\end{remark}

\section{Coherent Jordan configurations: Basic definitions and simple facts}
\label{sec:org6af43ec}
\label{org4a0fd08}

As in the case of CC's the starting definitions are formulated in
parallel in terms of relations and matrices. The reader is expected
to be ready to easily switch from one language to another and vice
versa without any difficulties, adjusting to such slightly loose
mode of behaviour.

Let us start from the consideration of the concept of a \emph{Jordan
algebra} as it is usually considered on an abstract level. This is a
vector space over the field \(F\) of characteristic zero,  with a
bilinear operation \(*\) which satisfies two 
requirements: 
\begin{equation}\tag{JA1}\label{equation:JA1}
A* B = B* A
\end{equation}
\begin{equation}\tag{JA2}\label{equation:JA2}
(A* B)* (A* A) = A*(B*(A* A)).
\end{equation}
In other words we keep commutativity of the multiplication, however
we relax the traditional axiom of associativity, which may or may
not be fulfilled for concrete classes of Jordan algebras.

As a rule in this text, \(F\) \strike{may be} \correction{is} assumed to be \strike{the field \({\mathbb
  C}\) of complex numbers, or even the field \({\mathbb R}\) of real
numbers}\correction{${\mathbb C}$ or ${\mathbb R}$}. 

The origins of this concept go back to the physicist Pascual Jordan,
see Section \ref{org03b7322} for a brief historical review and some
significant references.

\subsection{Jordan matrix algebra with a fixed basis}
\label{sec:org62c22b0}

Let \(J\) be a vector subspace of dimension \(r\) of the vector space
\(M_n(F)\) of square matrices of order \(n\) over the field \(F\) of
characteristic zero. We require that \(J\) have a special basis
\(\{A_0,A_1, \dots, A_{r-1}\}\) consisting of symmetric
\((0,1)\)-matrices. We require that the identity matrix
\(\identity_n\) belong to \(J\), and that
\(\sum_{i=0}^{r-1}A_i=\allone_{n}\), where \(\allone_{n}\) is the
square matrix of order \(n\) all of whose entries are equal to 1.

Clearly, all matrices in \(J\) are symmetric. Finally we require that
\(J\) will be closed with respect to the \emph{Jordan product} \(A*B\) of
matrices \(A\) and \(B\) defined as follows:
\[
   A*B = \frac 12 (A\cdot B + B\cdot A).
   \]
(Of course, for commuting matrices \(A\) and \(B\) the usual product
\(A\cdot B\) and the Jordan product \(A*B\) coincide.)

\begin{proposition}
The defined operation of \emph{Jordan multiplication} \(A*B\) of matrices
\(A\) and \(B\) satisfies the axioms \eqref{equation:JA1} and
\eqref{equation:JA2}. 
\end{proposition}
\begin{proof}
Trivial inspection.
\end{proof}

\begin{define}
A vector subspace \(J\) together with a special basis \(\{A_0, \ldots,
   A_{r-1}\}\) which is closed with respect to the introduced operation
of Jordan multiplication of matrices is called a \emph{(matrix) Jordan
algebra of order \(n\) and rank \(r\).}
\end{define}

In analogy to coherent algebras we define the following

\begin{define}
  A \emph{coherent Jordan algebra} is a Jordan algebra consisting of
  symmetric matrices of order $n$ which is closed under Schur-Hadamard
  multiplication and which contains the identity matrix ${\mathbb
    I}_n$ and the all-one matrix ${\mathbb J}_n$.
\end{define}

\begin{examp}
Each Bose-Mesner algebra of order \(n\) is a coherent Jordan algebra.
\end{examp}

As in the previous section for the CCs and ASs considered there, we
may translate the concept of a coherent Jordan algebra to the parallel
language of  binary relations over the set \(\Omega=\Omega_n\) of
cardinality \(n\). Then our requirements start from the consideration
of a pair \((\Omega, \{R_0, \dots, R_{r-1}\})\), where \(\{R_0,
   \dots, R_{r-1}\})\) provides a partition of the set \(\Omega^2\) into
a set of  binary relations over \(\Omega\). Then we talk of
\emph{basic relations} \(R_i\) instead of \emph{basic matrices} \(A_i\).

The general relational structure \strike{CJC} which corresponds to a
coherent Jordan
algebra \(J\) may be called a \emph{coherent Jordan configuration} \correction{(CJC or JC, for short)} of order \(n\) and
rank \(r\), or, following the terminology of D.G. Higman for the case
of CC's, a coherent Jordan \emph{rainbow} (see extra discussion at the end of the
paper). At this stage we prefer the term CJC.

The definitions for CJC's presented above provide the most general background
for the considerations which, in principle, are possible in the
current text. They will be briefly discussed in the forthcoming
paper \cite{KliMR}.

From this moment we however are restricting our attention to the
particular case when all basic relations \(R_i\) (and hence the
corresponding basic matrices \(A_i\)) are symmetric. Once more, this
assumption is now fulfilled till the very end of the current text. 

 A very particular case of CJC, like for CC's previously, corresponds
to the case when \(A_0=\identity_n\) is one of the basic matrices of
the Jordan algebra \(J\). In other words, the identity relation \(\id\)
on \(\Omega\) is equal to the basic relation \(R_0\). In this case we
will talk of a \emph{homogeneous} coherent Jordan algebra, or a 
\emph{coherent Jordan scheme} in
the parallel relational language. (We can also talk of a Jordan
scheme (JS) if the extra feature "coherent" is clear from the
context.) 

In what follows we will mainly consider Jordan schemes, though at
this general section both concepts of CJC and JS are equally
significant.

Note that the existence of a CJC implies the existence of its tensor
of \emph{structure constants}, aka \emph{intersection numbers}, which this
time are introduced in terms of Jordan multiplication instead of
the usual multiplication of matrices.

Assume that \(W\) is a non-symmetric CC, in particular, a
non-symmetric AS. For each antisymmetric basic matrix \(A_i\) let us
denote \(\tilde A_i = A_i + A_i^T\). Then we can consider the
\emph{symmetrization} \(\tilde W\). Its basic matrices are the matrices
\(\tilde A_i\) instead of the antisymmetric basic matrices, while
symmetric basic matrices remain unchanged.

In particular, the symmetrization of a BM-algebra coincides with
itself. It is an easy exercise to prove the following claim:

\begin{proposition}
The symmetrization \(\tilde W\) of an arbitary CC \(W\) is a Jordan
configuration. 
\end{proposition}

In what follows we will call the symmetrization of any CC a
\emph{non-proper} CJC. A \emph{proper} CJC is one that cannot be obtained by
symmetrization. Similar definitions clearly work for a particular
case of CJCs, that is, for Jordan schemes. 

As was claimed in Section \ref{orgb33f855}, the main goal of this text is to
present first examples of proper Jordan schemes.

\subsection{Jordan schemes of small rank}
\label{sec:orgedec6a7}

The theory of general Jordan algebras is developed reasonably
well. We refer for a number of classical initial facts to 
texts such as \cite{Mal86} and \cite{See72}. Below is one of them.

\begin{proposition}
Let \(Sym_n(F)\) be the vector space of all symmetric matrices of
order \(n\) over the field \(F\). Then any subspace \(W\) of it is a
Jordan algebra iff it is closed with respect to taking squares.
\end{proposition}

Note that \(Sym_n(F)\) is the largest possible coherent Jordan algebra of
order \(n\) over \(F\), because it is the symmetrization of \(M_n(F)\),
the CC of order n and rank \(n^2\), which appears as the largest CC
of order \(n\) (i.e., the centralizer algebra of the identity
permutation group of degree \(n\)). The coherent configuration of
rank \(n^2\) is usually called the \emph{discrete configuration}.

The statement below was proved in \cite{Bai04} (part of Proposition 12.5). We
provide here the proof to make the text self-contained. See also
Section \ref{sec:org9d08348} for further discussion.

\begin{proposition}
\label{org163c63e} Let \({\frak X} = (\Omega, {\cal R} = \{R_0 , R_1
   , R_2 , \dots, R_{r-1} \})\) be a symmetric regular coloring of the
complete graph on \(\Omega\), \(A_i\) an adjacency matrix of \(R_i\) and
\({\cal A} = \left< A_0 , \dots, A_{r-1} \right>\) the linear span of
the adjacency matrices \(A_i\) . Then \({\frak X}\) is a (coherent)
Jordan scheme iff \((A_i + A_j )^2 \in A\) holds for all \(i, j \in
   \{1, \dots, r-1\}\).
\end{proposition}
\begin{proof}
 The statement follows from two simple identities: \(X * X = X^ 2\)
and \((X + Y )^ 2 = X^ 2 + 2X *  Y + Y^2\).
\end{proof}

Let us now consider Jordan schemes of small rank.

\textbf{Rank 2}: A Jordan scheme of rank 2 and order \(n\) has two trivial
basic graphs: the identity graph \(\Gamma_0=(\Omega_n, \id_n)\),
which corresponds to the full reflexive relation on \(\Omega_n\),
and the complete graph \(K_n\) (without loops) on the vertex set
\(\Omega_n\). 

\textbf{A Jordan scheme of rank 3} is also pretty clear. It corresponds to
a pair of complementary SRG's \(\Gamma_1\) and \(\overline\Gamma_1\)
with common vertex set \(\Omega\). The formal proof of such a claim
is an easy exercise on the edge of the theory of SRG's and simple
facts treated in the current section. We prefer even not to
evidently formulate  the corresponding claim. An interested reader
is referred to a helpful introductory text \cite{CarV04}.

The next propositions seem to be less trivial. Recall that each
time we use the term \emph{regular symmetric color graph of rank \(r\)} on
a vertex set \(\Omega\) for any partition of \(\Omega^2\) into
symmetric regular relations including the identity relation.

\begin{proposition}
Let \({\frak X} = (\Omega, \{R_0, R_1, R_2, R_3\})\) be a regular
symmetric graph of rank 4. Then \({\frak X}\) is a JS if and only if
it is an AS.
\end{proposition}

\begin{proof}
The direction AS \(\Rightarrow\) JS is trivial. To prove the opposite
implication, consider \({\cal A} = \left<A_0, A_1, A_2,
   A_3\right>\), that is the space spanned by the matrices \(A_i\). Then
\(A_i^k\in{\cal A}\) for each \(k\in{\mathbb N}\). Therefore, the
minimal polynomial of each \(A_i\) has degree at most 4. If this
bound is reached for some \(i\) then we get that \({\cal A} = {\mathbb
   R}[A_i]\), and \({\cal A}\) is commutative, therefore \({\frak X}\) is a
(commutative) AS.

Otherwise, the minimal polynomials of all \(A_i\) have degree at
most 3. This means that all basic graphs \(\Gamma_i\), \(i=1,2,3\), are
strongly regular. It is a well known fact in AGT, see, e.g.,
\cite{vDam03, vDamM10}, that a partition of the edge set of a complete
graph into three strongly regular graphs defines an AS; namely
amorphic rank 4 AS in the sense of \cite{GolIK94}. 
\end{proof}

The presented simple facts allow to formulate efficient sufficient
conditions for the existence of a rank 5 JS such that at least one
of its basic graphs is an SRG. These conditions will be considered
later on in the text (starting from Section \ref{org5fceb33}), when their
emergence will become reasonably clear to the reader.

\subsection{Jordan closure and WL closure}
\label{sec:org98eb02f}

Because the polynomial-time procedure of the Weisfeiler-Leman
stabilization may be described in terms of the usual multiplication
of matrices, a similar procedure may be defined also for Jordan
algebras. The result, which should be called the Jordan closure of
a given set \(X\) of symmetric square matrices of order \(n\) over a
field, say, \({\mathbb R}\), is the smallest Jordan algebra \(JA(X)\)
which contains the prescribed set \(X\) of matrices.

 The definitions presented above provide the most general background
for the considerations which, in principle, are possible at the
current text. They will be briefly discussed in the forthcoming
paper \cite{KliMR}.

The fact of the existence of the coherent Jordan closure \(JA(X)\)
of a given non-empty set \(X\) of symmetric square matrices of order
\(n\) is pretty clear. Indeed, \(JA(X)\) is defined as the smallest
coherent Jordan algebra which contains the set \(X\). Take into
account that a coherent Jordan algebra is defined as a set of
square matrices of order \(n\) over the field \({\mathbb R}\) which is
closed with respect to two operations of arity zero, one unary
operation and two binary operations. This immediately implies the
correctness of the definition of \(JA(X)\). 

A brief justification of the correctness of such a definition can
be found in the slides \cite{Cam18} of a lecture by P.J. Cameron,
given at Będlewo, a Polish analogue of Oberwolfach.

\begin{remark}
In fact, a simple version of the Jordan stabilization procedure of
a given set of square matrices of order \(n\) was implemented and was
quite frequently used in the course of our project. Some of the
discussed theoretical claims were helpful to test the correctness
of the created program, especially in the initial stages of its
development. 
\end{remark}

One may also introduce the definition of a \emph{genuine} JS
\(J\). This is a Jordan scheme \(J\) which is coarser than its
WL-stabilization. Each proper \(JS\) clearly is genuine. Small
examples of non-proper genuine JS's of orders \(n=6,10,12\) will be discussed
later on in this text. In principle, other more sophisticated classes
of genuine JS's may be considered, though this will not become subject
of a special evident attention at the current preprint. At this stage
we just provide a simple helpful claim.

\begin{proposition}
Let \({\frak X}\) be an AS, \(\tilde{\frak X}\) its
symmetrization. Assume that \(\tilde{\frak X}\) is a genuine
\(JS\). Then \({\frak X}\) is non-commutative.
\end{proposition}

\begin{proof}
Assume that \({\frak X}\) is commutative, then the symmetrization
\(\tilde{\frak X}\) is an AS. Therefore, \(WL(\tilde{\frak
   X})=\tilde{\frak X}\). This contradicts the assumption that
\(\tilde{\frak X}\) is genuine.
\end{proof}

\subsection{Walk regular graphs}
\label{sec:org22e8ce6}

The first examples of proper Jordan schemes, presented below,
consist of unicolor graphs, which are "very regular". For example,
the scheme \(J_{15}\) consists of distance regular graphs of
valencies 2 and 4.

This is why it is natural from the beginning to consider quite
strong necessary conditions for the basic graphs of new structures
we are searching for.

Recall, see, e.g., \cite{BroH12}, \cite{GodR01}, that a simple
graph \(\Gamma\) is called \emph{walk regular} if for any \(l\ge 1\) each
power \(A(\Gamma)^l\) of its adjacency matrix has a constant diagonal
(depending on the value of \(l\)). The concept goes back to
\cite{GodM80}, who introduced a couple equivalent definitions and
provided some feasibility conditions for the parameters of a
walk regular graph (WRG). Clearly each WRG is regular.

It is easy to understand that each coherent (simple) graph is a
WRG. The same holds for any basic graph of a Jordan scheme.

Let us call a WRG \(\Gamma\) \emph{properly walk regular} (PWRG) if it is
a WRG and the coherent closure \(WL(\Gamma)\) has at least two
fibers. Any union \(\Delta\) of basic graphs of a putative Jordan
scheme \(\frak{S}\) should be a WRG. However in principle, provided
\(\frak{S}\) is proper, the result of \(WL(\Delta)\) may be a
non-homogeneous CC, in other words, \(\Delta\) might be a PWRG.

Let us discuss some examples. A classical example of a PWRG is the
Hoffman graph \(H_{16}\) of order 16, which is cospectral with the
4-dimensional cube \(Q_4\). As will be mentioned below its Jordan
closure coincides with the symmetrisation \(SWL(H_{16})\) of
\(WL(H_{16})\). 

Taking into account that the proper scheme \(J_{15}\) was initially
approached via a lucky guess, there was sense to try to construct
similar structures of smaller orders. We took two different
approaches. 

The second, systematical approach will be discussed elsewhere. In
particular, we expect to consider all small PWRG's.

The first approach is more empirical, though again it relies on the
use of a computer.

The paper \cite{GodM80} contains a diagram of a graph \(GM_{13}\) of
order 13 and valency 4, presented there in Figure 2. From the
 context of the paper it was natural to assume that this graph
\(GM_{13}\) is walk regular. Surprisingly, the computer did not
confirm such an assumption, and afterwards the claim became
immediately clear via visual arguments. Indeed, the vertex at the
center does not belong to any triangle, while all 12 other vertices
do.

Being disappointed by such an observation, we started looking for
further information. On this way we faced a report by Mark Farrell
\cite{www:farrell}    about his search for smaller WRGs. A slight
difficulty was that Mark's email address is currently not working,
thus he, in a sense, remains "invisible".

Finally, a small ad hoc program was written in order to search for
all PWRGs of order 12. It turns out that up to isomorphism and
complements there are exactly two such graphs, denoted by
\(\Gamma_{12,4}\) and \(\Gamma_{12,5}\), of valencies 4 and 5,
respectively.

We further processed the first of these graphs, of
valency 4. We created the diagram presented in Figure \ref{figure:3.1}.

\begin{figure}
\begin{center}
\input{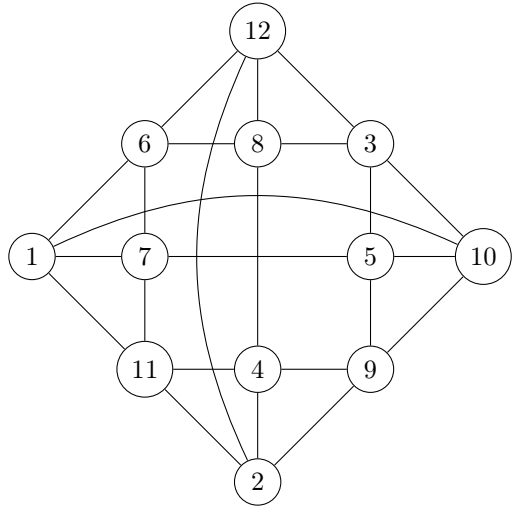}    
\end{center}
\caption{The proper walk regular graph \(\Gamma_{12,4}\) \label{figure:3.1}}
\end{figure}

The reader is welcome to compare it with the diagram of the graph
\(GM_{13}\) and to conclude that the current one appears by removing
of the central vertex of \(GM_{13}\), and interpreting four edges
incident to it as one horizontal and one vertical edge on the new
diagram. This pretty well explains the origin of the mistake: A kind
of typographical error.

Playing with the graph \(\Gamma_{12,4}\), its automorphism group was
first described with the aid of nauty and \gap and later on in a
more human-friendly manner.

On the next stage one more ad-hoc program was written. The goal was
to search for a "nice" imitation of a rank 5 proper Jordan scheme of
order 12. The program produced a color graph \(\frak{S}\) with the
following adjacency matrix:

\[
  \left(\begin{array}{rrrrrrrrrrrr}
  0&1&1&2&4&3&3&4&4&3&3&4\\
  1&0&1&3&2&4&4&4&3&4&3&3\\
  1&1&0&4&3&2&4&3&4&3&4&3\\
  2&3&4&0&1&1&4&3&3&4&3&4\\
  4&2&3&1&0&1&3&4&3&3&4&4\\
  3&4&2&1&1&0&3&3&4&4&4&3\\
  3&4&4&4&3&3&0&1&1&2&3&4\\
  4&4&3&3&4&3&1&0&1&4&2&3\\
  4&3&4&3&3&4&1&1&0&3&4&2\\
  3&4&3&4&3&4&2&4&3&0&1&1\\
  3&3&4&3&4&4&3&2&4&1&0&1\\
  4&3&3&4&4&3&4&3&2&1&1&0\\
  \end{array}\right)
  \]

In this rank 5 color graph all unicolor graphs are WRG's. The graphs
\(\Gamma_1=4\circ K_3\) and \(\Gamma_2=6\circ K_2\) are
vertex-transitive. The graphs \(\Gamma_3\) and \(\Gamma_4\) are both
isomorphic to the PWRG \(\Gamma_{12,4}\). In fact, \(\Gamma_3\) literally
coincides with the copy depicted in Figure \ref{figure:3.1}.

\begin{proposition}
\begin{enumerate}
\item The color graph \(\frak{S}_{12}\) provides an example of a rank 5
graph with all unicolor graphs walk regular; the two graphs of
valency 4 are isomorphic PWRGs.
\item \label{orgba815bb} \(\frak{S}_{12}\) is not a Jordan scheme.
\item \label{org69b2e24}\(\aut(\Gamma_{12,4})\cong E_{16}:{\mathbb Z}_2\) is a group of
order 32, with orbits of lengths 4 and 8 on the vertices.
\item \(\aut(\frak{S}_{12})\) is a group of order 2 generated by the
permutation \[h_1=(0,4)(1,3)(2,5)(6,9)(7,11)(8,10)\].
\item \(\caut(\frak{S}_{12})=\left<h_1,h_2\right> \cong E_4\), where
\(h_2=(0,1)(3,4)(6,9)(7,10)(8,11)\).
\end{enumerate}
\end{proposition}

\begin{proof}
All results were obtained with the aid of a computer.

Part \ref{org69b2e24} can be easily justified by hand. Indeed, the group
\(D_4\) of order 8, acting on the orbit \(\{3,6,9,11\}\) of
\(\aut(\Gamma_{12,4})\) is pretty visible from the diagram. Making
extra small efforts, we find the kernel of the action of
\(\aut(\Gamma_{12,4})\) on the above set. The kernel is isomorphic to
\(E_4\). Finally, the group acts transitively and faithfully on the set
\(V(\Gamma_{12,4})\setminus\{3,6,9,11\}\).

In principle the other parts can also be justified without the use of a
computer. 
\end{proof}

We see a few goals of this concluding part of Section
\ref{org4a0fd08}:
\begin{itemize}
\item To show elements of our technology, which combines an initial wild
guess, further computer-aided efforts, and finally a human glance
on the obtained result;
\item to introduce the structure \(\frak{S}_{12}\) as a kind of "exotic
flower", which is not a (proper) Jordan scheme, but imitates it in
some extent;
\item to announce special kind of activities which in future might be
exploited much more seriously.
\end{itemize}

\section{Proper Jordan scheme \(J_{15}\): An initial glance}
\label{sec:orgde4e04e}
\label{orgeb5dfbd}
In this section, in principle, we follow the style and notation in
\cite{KliRW09}. Some small adjustments however are done in a fashion
which is more convenient to the current presentation.

Thus we start from the group \(A_4=\left<g_1, g_2\right>\), where \(g_1=(0,1,2)\),
\(g_2=(1,2,3)\); \(A_4\) acts 2-transitively on the set
[0,3]=\{0,1,2,3\}. Using COCO (see \cite{FarK91}), we construct an
induced intransitive action of \(A_4\) on the set \(\Omega\) of
cardinality 15. Here \(\Omega=\Omega_1\cup\Omega_2\), where
\(\Omega_{\text{1}}\)=(0,1)\(^{\text{A}_{\text{4}}}\), \(\Omega_{\text{2}}\)=\{\{0,1\},\{2,3\}\}\(^{\text{A}_{\text{4}}}\) are sets of size
12 and 3, respectively. Clearly the restriction \((A_4, \Omega_1)\)
coincides with the regular action of \(A_4\). In the COCO labelling
\(\Omega\)=[0,14], here we have the following COCO map, see Table
\ref{table:15aa.map}. 

\begin{table}
\begin{center}
$
\begin{array}{|r|r|}
\hline
  0&(0,1)\\
  1&(1,2)\\
  2&(0,2)\\
  3&(2,0)\\
  4&(2,3)\\
  5&(1,0)\\
  6&(0,3)\\
  7&(3,0)\\
  8&(3,1)\\
  9&(2,1)\\
 10&(1,3)\\
 11&(3,2)\\
 12&\{\{0,1\},\{2,3\}\}\\
 13&\{\{0,3\},\{1,2\}\}\\
 14&\{\{0,2\},\{1,3\}\}\\
 \hline
\end{array}
$
\end{center}
\caption{Elements of \(\Omega\) \label{table:15aa.map}}
\end{table}

It turns out that the action \((A_4, \Omega)\) has rank 21. This was
obtained with the aid of COCO, though it may be easily confirmed
using the CFB (orbit counting) Lemma (cf. \cite{KliPosRos88b}).

In Table \ref{table:2-orbits} we list representatives of the
2-orbits of the intransitive group (A\(_{\text{4}}\), \(\Omega\)), as they were
automatically generated by COCO.

\begin{table}
\begin{center}
\begin{tabular}{|r|c|c||r|c|c|}
\hline
2-Orbit & Subdegree & Representative&
2-Orbit & Subdegree & Representative\\
\hline
0&1&(0,0)&12&1&(0,12)\\
1&1&(0,1)&  13&1&(0,13)\\

2&1&(0,2)&  14&1&(0,14)\\

3&1&(0,3)&  15&4&(12,0)\\
4&1&(0,4)&  16&4&(12,1)\\
5&1&(0,5)&  17&4&(12,2)\\
6&1&(0,6)&  18&1&(12,12)\\
7&1&(0,7)&  19&1&(12,13)\\
8&1&(0,8)&  20&1&(12,14)\\
9&1&(0,9)& &&\\
10&1&(0,10)&&&\\
11&1&(0,11)& &&\\ 
\hline
\end{tabular}
\end{center}
\caption{2-orbits of \((A_4,\Omega)\) \label{table:2-orbits}}
\end{table}

COCO returns twelve rank 6 mergings, all having a group of
order 60. Using \gap we may confirm that all these mergings are
isomorphic to the unique rank 6 AS, which appears from the
transitive action of A\(_{\text{5}}\) of degree 15. In the framework of the
current presentation these AS's are not immediately needed; they
will be reconsidered in further sections.

Using \gap we obtain 24 coherent Jordan schemes of rank 5. They are
split into two isomorphism classes with 12 elements each. One class
consists of non-proper schemes corresponding to the symmetrization
of the rank 6 scheme above. The other class consists of proper
Jordan schemes; we denote a representative of this class by
\(J_{15}\). 
Each scheme consists of a spread \(5\circ K_3\) and three
isomorphic copies of the unique DTG \(\Delta\) of valency 4 and
diameter 3. The graph \(\Delta\) is nothing else but the antipodal cover
of K\(_{\text{5}}\), which is a DTG. Again, \(\Delta\) will be considered with more
details in the next sections.

Here we just present the adjacency matrix A=A(J\(_{\text{15}}\)) of the object,
which from this moment will be denoted by \(J_{15}\) and called the
\emph{proper Jordan scheme of order 15.} In fact, A=\(\sum_{\text{i=0}}^{\text{4}}\) i\(\cdot\)
A\(_{\text{i}}\), where A\(_{\text{i}}\) is the adjacency matrix of the corresponding basic
graph \(\Gamma_{\text{i}}\)=(\(\Omega\), R\(_{\text{i}}\)), forming the entire color graph
\(\Gamma\). Moreover, \(\Gamma_1\cong 5\circ K_3\), while the three
graphs \(\Gamma_2\), \(\Gamma_3\), \(\Gamma_4\) are isomorphic to the DTG \(\Delta\).

\[
  A=A(J_{15}) = \left(\begin{array}{rrrrrrrrrrrrrrr}%
0&1&1&2&3&4&2&3&4&2&3&4&2&3&4\\%
1&0&1&3&4&2&3&4&2&3&4&2&3&4&2\\%
1&1&0&4&2&3&4&2&3&4&2&3&4&2&3\\%
2&3&4&0&1&1&3&2&4&2&4&3&4&3&2\\%
3&4&2&1&0&1&2&4&3&4&3&2&3&2&4\\%
4&2&3&1&1&0&4&3&2&3&2&4&2&4&3\\%
2&3&4&3&2&4&0&1&1&4&3&2&2&4&3\\%
3&4&2&2&4&3&1&0&1&3&2&4&4&3&2\\%
4&2&3&4&3&2&1&1&0&2&4&3&3&2&4\\%
2&3&4&2&4&3&4&3&2&0&1&1&3&2&4\\%
3&4&2&4&3&2&3&2&4&1&0&1&2&4&3\\%
4&2&3&3&2&4&2&4&3&1&1&0&4&3&2\\%
2&3&4&4&3&2&2&4&3&3&2&4&0&1&1\\%
3&4&2&3&2&4&4&3&2&2&4&3&1&0&1\\%
4&2&3&2&4&3&3&2&4&4&3&2&1&1&0\\%
\end{array}\right)
  \]

Having this matrix, any sufficiently  motivated reader
may in principle confirm that \(\left<A_i\mid 0\le i\le 4\right>\) is indeed a proper
Jordan algebra. We however warn the reader that this fact will be
reconfirmed in the forthcoming sections a couple of times and in a
more comfortable manner in comparison with a brute force inspection
possible now.

\section{The antipodal distance transitive graph of valency 4 and order 15}
\label{sec:org0fc2b73}
\label{orgaf4339e}
In this section we introduce and investigate the graph \(\Delta\) of
valency 4, diameter 3 and order 15, which is unique up to
isomorphism and is an antipodal distance transitive graph. 

Though part of the arguments initally rely on the use of a computer,
finally they can be justified in a human-friendly manner. The use of
suitable diagrams will be an essential feature of our style of
presentation. 

Recall the classical model for the famous Petersen graph \(\Pi\). Its
vertices are 2-element subsets of the 5-set [0,4], adjacency means
that the 2-subsets are disjoint. In other words,
\(\Pi=\overline{T_5}=\overline{L(K_5)}\), where \(T_5\) is the
\emph{triangular graph}, i.e., the line graph of the complete graph
\(K_5\). Figure \ref{figure:petersen} presents the well-known diagram of \(\Pi\).

\begin{figure}
\begin{center}
\input{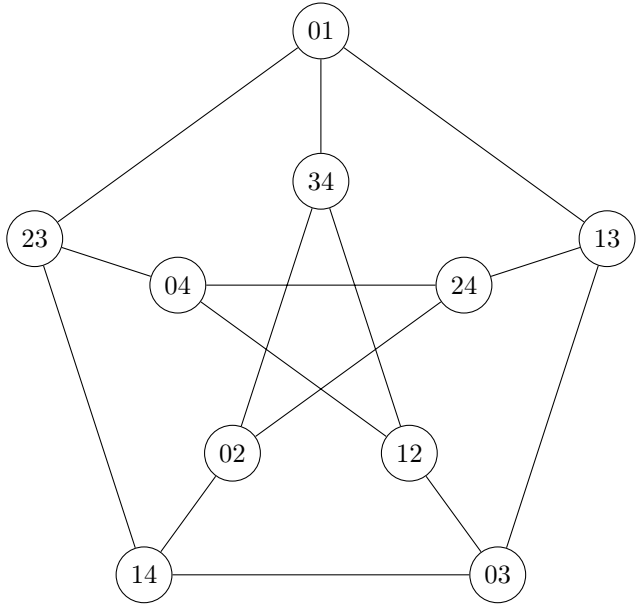}
\end{center}
\caption{Petersen graph \label{figure:petersen}}
\end{figure}

\begin{proposition}
\begin{enumerate}
\item The Petersen graph \(\Pi\) is a primitive SRG with the parameter
set \[(n,k,\lambda,\mu) = (10,3,0,1)\].
\item The parameter set of \(\Pi\) defines it uniquely up to isomorphism.
\item \(\operatorname{Aut}(\Pi)\cong S_5\), the symmetric group of order 120.
\end{enumerate}
\end{proposition}

\begin{proof}
The proof is straightforward; see also \cite{HolS93} as well as
\cite{Har69} for the formulation of the classical Whitney-Young
Theorem about the automorphism group of the line graph L(X) of a
given connected graph X.
\end{proof}

Now we introduce the graph \(\Delta\)=L(\(\Pi\)). Its vertex set of
cardinality 15 consists of structures \{\{a,b\},\{c,d\}\}, where \{a,b,c,d\}
is a subset of [0,4] of size 4. Adjacency means that two structures
have exactly one common 2-subset. Clearly, this implies that
\(\Delta\)=L(\(\Pi\)). In what follows, for brevity \{\{a,b\},\{c,d\}\} will be
substituted by ab,cd, written in the lexicographical order.

\begin{proposition}
\begin{enumerate}
\item \label{orgf1af0d7} The graph \(\Delta\) is a DRG of diameter 3 with the parameter set
(4,2,1;1,1,4).
\item \label{org3952143} \(\operatorname{Aut}(\Delta)\cong S_5\).
\item \label{org1c683c0} \(\Delta\) is a DTG.
\item \label{org71c050c} \(\Delta\) is an antipodal 3-fold cover of K\(_{\text{5}}\).
\item \label{org15c435f} The DRG \(\Delta\) is uniquely determined by its parameter set.
\end{enumerate}
\end{proposition}

\begin{figure}
\begin{center}
\input{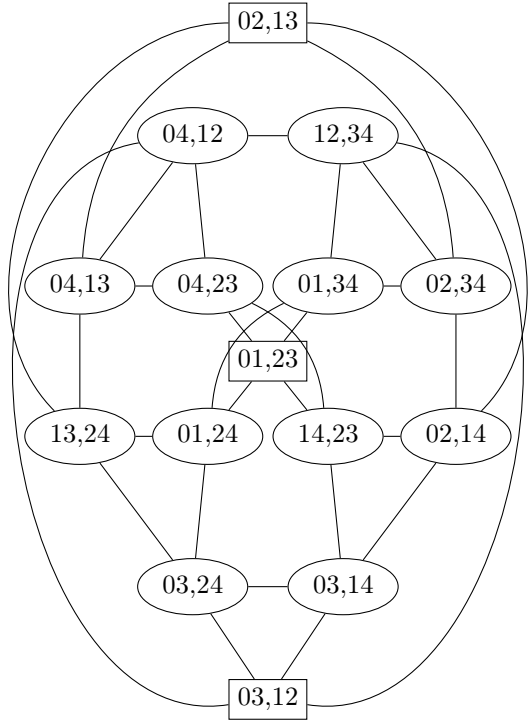}
\end{center}
\caption{Distance diagram of \(\Delta\) with respect to \(01,23\)}
\end{figure}

\begin{proof}
All the claimed results are well-known, see \cite{BroCN89} for exact
credits, in particular to D.H. Smith and A. Gardiner. For
completeness we discuss short outlines of the proofs below. For the
proof of \ref{org3952143} use again the Whitney-Young Theorem. Check that
\(\Delta\) has diameter 3, confirm that the vertex stabilizer of
Aut(\(\Delta\)) is a group of order 8, isomorphic to D\(_{\text{4}}\), which acts
transitively on the sets \(\Delta_{\text{1}}\)(x), \(\Delta_{\text{2}}\)(x) and \(\Delta_{\text{3}}\)(x) of
the vertices at distance i from x, i\(\in\)\{1,2,3\}. This implies \ref{org1c683c0}
and also \ref{orgf1af0d7}. The proof of \ref{org71c050c} follows from Figure
\ref{figure:delta-antipodal}.  For the proof of \ref{org15c435f} consider the diagram of \(\Delta\) presented
in Figure \ref{figure:delta-classical} and confirm that such a structure is easily
reconstructed from the parameter set.
\end{proof}

\begin{figure}
\begin{center}
\input{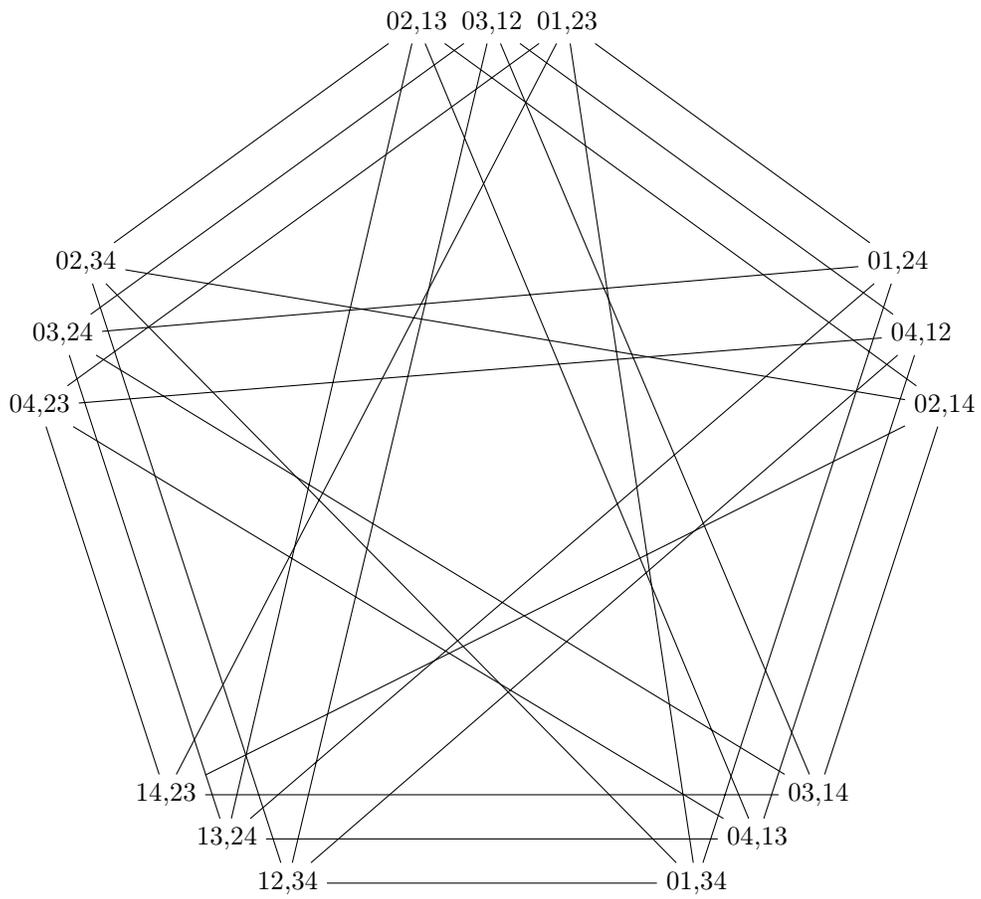}
\end{center}
\caption{\(\Delta\) as antipodal cover of \(K_5\) \label{figure:delta-antipodal}\label{figure:5.3}}
\end{figure}

\begin{figure}
\begin{center}
\input{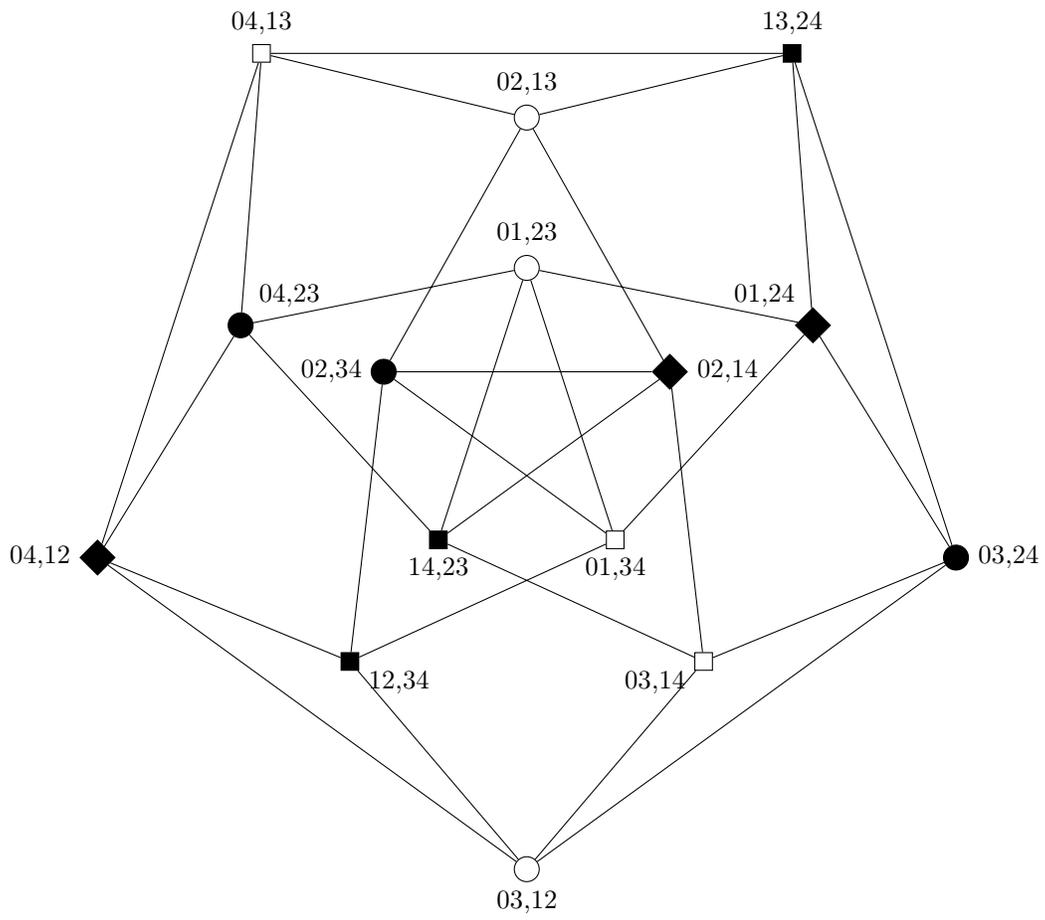}
\end{center}
\caption{Classical diagram of \(\Delta\) \label{figure:delta-classical}}
\end{figure}

Note that Figure \ref{figure:delta-classical} presents a nice
symmetric depiction of \(\Delta\); exactly this diagram appears on p.2
of  \cite{BroCN89}, the classical encyclopedic text in AGT.

\begin{proposition}
Let \(\Delta=\Delta_1=(\Omega,R_1)\) be our DTG of order 15. Let
\(\Delta_3=(\Omega, R_3)\) be the distance 3 graph of \(\Delta\). Define
\(\Sigma=(\Omega, R_1\cup R_3)\). Then \(\Sigma\) is an SRG with the
parameters (n,k,\(\lambda\),\(\mu\))=(15,6,1,3) which is isomorphic to the
complement \(\overline{T_6}\) of the triangular graph T\(_{\text{6}}\).
\end{proposition}

\begin{proof}
Here we exploit a classical remarkable property of the number 6, see
Chapter 6 of \cite{CamvL91}. Namely, the group \(S_5\) has two doubly
transitive actions of degree 5 and 6. The first action is the
natural action, while the second action is defined on the set \({\cal
  F}\) of all one-factorizations of the complete graph K\(_{\text{6}}\). There exist
exactly six such factorizations, let us fix one of them,  F\(_{\text{1}}\), which will
be called the \emph{special factorization}. It turns out that that
Aut(F\(_{\text{1}}\))\(\cong\) S\(_{\text{5}}\). Having F\(_{\text{1}}\), we first define the graph
\(\overline{T_6}\) on the set \({\genfrac {\{} {\}} {0pt} 1 {[0,5]} 2}\) of
all 2-subsets of [0,5], two subsets are adjacent if they are
disjoint. Clearly we get the unique SRG of valency 6 and order 15,
which is isomorphic to \(\overline{T_6}\). Now we remove from the edge
set of \(\overline{T_6}\) five cliques K\(_{\text{3}}\) (five triangles in
\(\overline{T_6}\)), which are defined by one-factors in F\(_{\text{1}}\). The
diagram in Figure \ref{figure:deleted} allows to finish the proof. Note that here we
use the special one-factorization as follows:
\[
  F_1 = \{01|25|34,02|13|45,03|15|24,04|12|35,05|14|23\}.  
  \]
(This time we use the full notation for F\(_{\text{1}}\) which is presented in
lexicographical order.)
\end{proof}

Now we aim to split the edge set of the graph T\(_{\text{6}}\) into two copies,
say \(\Delta_{\text{1}}\) and \(\Delta_{\text{2}}\), of the graph \(\Delta\). For this purpose we
exploit one more classical combinatorial structure: The smallest
non-trivial block design \({\cal D}\) with parameters
(v,b,k,r,\(\lambda\))=(6,10,3,5,2). 

This combinatorial structure may be introduced with the aid of one
more nice diagram, see Figure \ref{figure:5.6}. Note that according to COCO
methodology, here the same numbers are used for a few purposes. We
start with the natural action of A\(_{\text{5}}\) on [0,4] and select a pentagon,
denoted by O, which appears in the center of Figure \ref{figure:5.6}. Its
automorphism group D\(_{\text{5}}\) of order 10 consists of even permutations
only. This implies that the orbit of the pentagon O under the action
of A\(_{\text{5}}\) has length 6. Five more cycles of length 5 (as subgraphs of
K\(_{\text{5}}\)) also appear at the same figure. These cycles are regarded as
points of the coming structure \({\cal D}\). The blocks are
bijectively associated with 10 edges of K\(_{\text{5}}\). A cycle is incident to
an edge if the edge appears in the cycle. In this fashion we get the
following list B of the blocks of \({\cal D}\):
 \begin{align*}
 B = &
 \{\{0,1,3\},\{0,1,4\},\{0,2,4\},\{0,2,5\},\{0,3,5\},\\
& \{1,2,3\},\{1,2,5\},\{1,4,5\},\{2,3,4\},\{
 3,4,5\}\}.
 \end{align*}
The remaining 10 subsets of size 3 of the set [0,5] form the set
\(\overline B\) of the structure \({\overline{\cal D}}\). Note that here
\({\overline{\cal D}}\) is the complement to \({\cal D}\) both in the
sense of \cite{HugP85} and \cite{CamvL91}. 

\begin{proposition}
\begin{enumerate}
\item \label{orgb39d308} The incidence structure \({\cal D}\) is a simple BIBD with
parameters 
\[(v,b,k,r,\lambda)=(6,10,3,5,2).\]
\item \label{org787ed61} The designs \({\cal D}\) and \(\overline{\cal D}\) are isomorphic.
\item \label{orgbffcb51} Aut(\{\cal D\})\(\cong\) A\(_{\text{5}}\)
\item \label{org11b7773} \(Aut(\{{\cal D}, \overline{\cal D}\}) \cong S_5\).
\end{enumerate}
\end{proposition}

\begin{proof}
The proof of \ref{orgb39d308} is straightforward. The result \ref{org787ed61} follows from
the general theory of 2-designs, though it is also quite evident. By
construction, both \({\cal D}\) and \(\overline{\cal D}\) are invariant
with respect to A\(_{\text{5}}\). (Here the group in Part \ref{org11b7773} either preserves
\({\cal D}\) and \(\overline{\cal D}\) or swaps them.) Because the
designs \({\cal D}\)  and \(\overline{\cal D}\) are isomorphic, the
group in \ref{org11b7773} is twice larger than in \ref{orgbffcb51}. Finally we take into
account that the pair \(\{{\cal D}, \overline{\cal D}\}\) is
reconstructible from F\(_{\text{1}}\) and vice versa, while Aut(F\(_{\text{1}}\)) \(\cong\) S\(_{\text{5}}\),
see again \cite{CamvL91}.
\end{proof}

At this moment we are prepared to pursue our aim, as was claimed above.

\begin{figure}
\begin{center}
\input{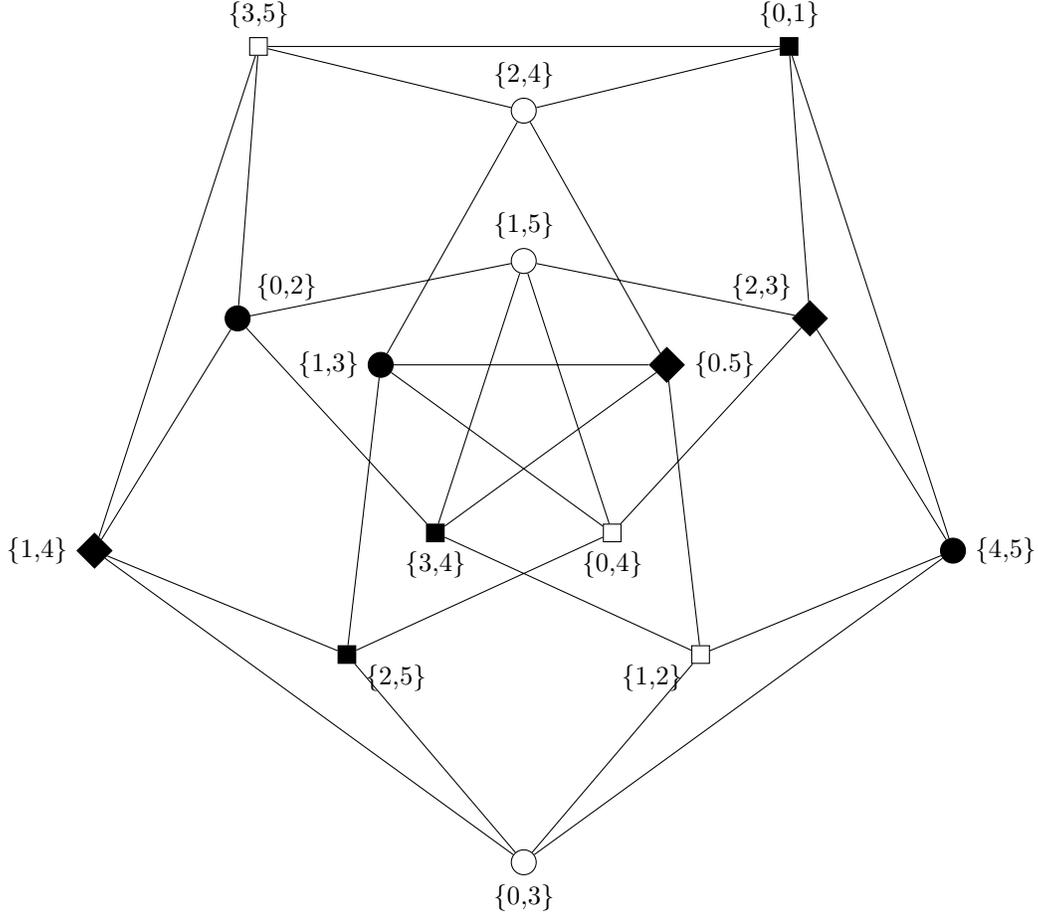}
\end{center}
\caption{Graph \(\Delta\) as \(\overline{T_6}\) with a one-factorization of \(K_6\) deleted \label{figure:deleted}}
\end{figure}

\begin{figure}
\begin{center}
\input{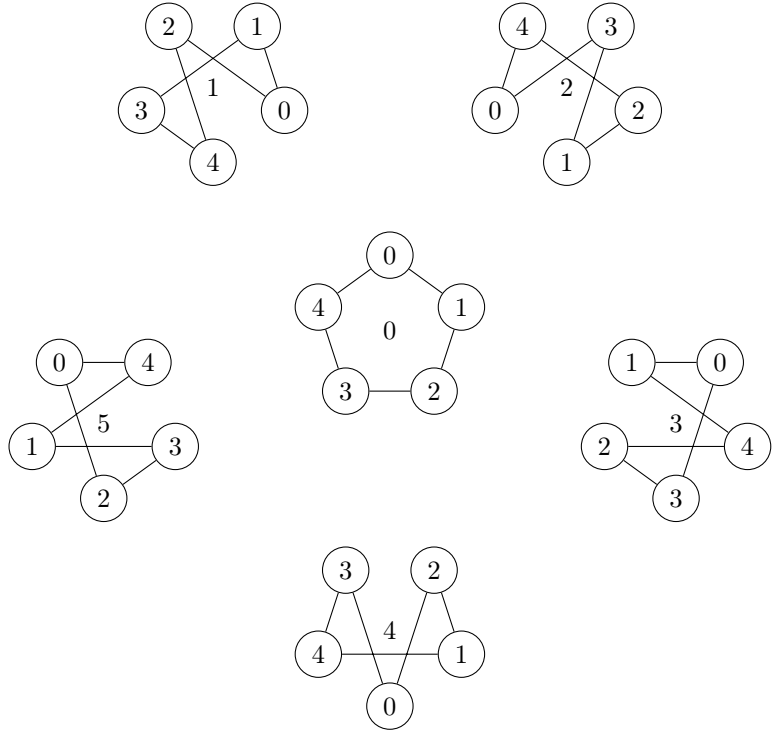}
\end{center}
\caption{Pictorial introduction of the block design \({\cal D}\) \label{figure:5.6}}
\end{figure}

\begin{figure}
\begin{center}
\input{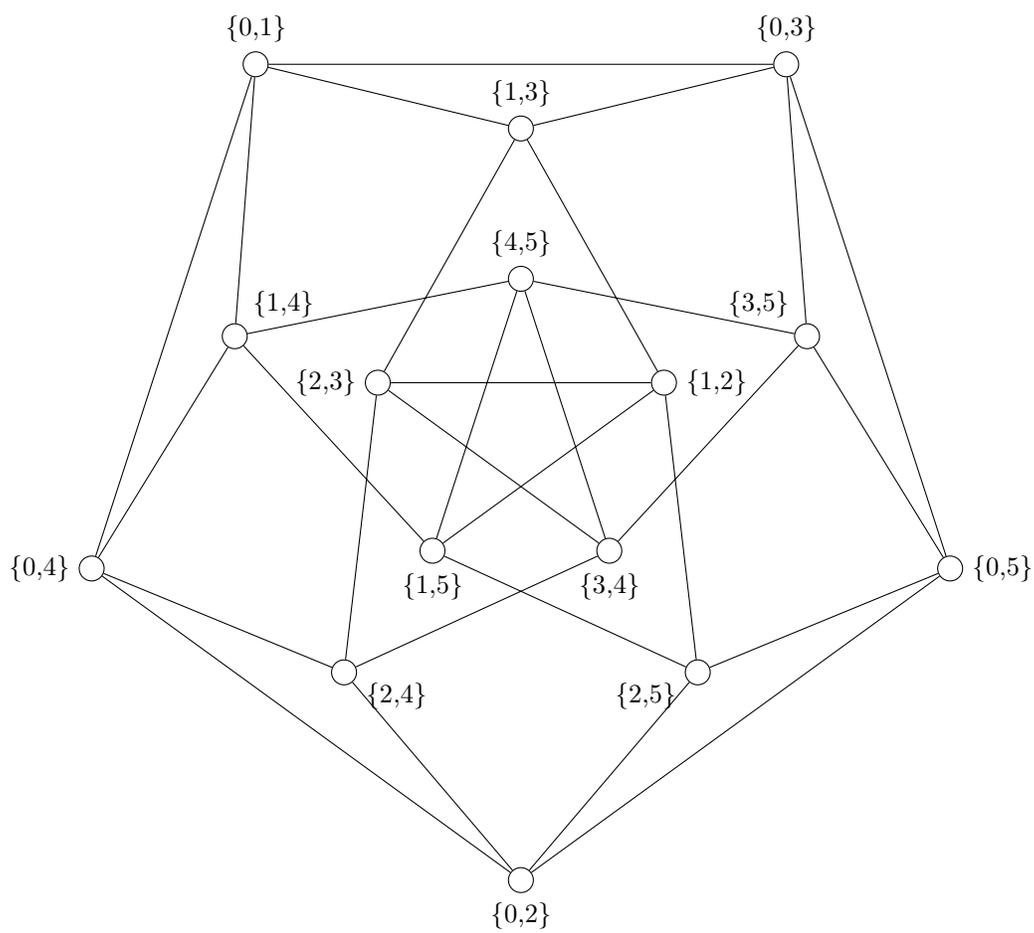}
\end{center}
\caption{Split of the edges of the graph \(\Delta_1\) into 10 triangles formed by the blocks of Design \({\cal D}\) \label{figure:5.7}}
\end{figure}

Having the design \({\cal D}\), let us define the graph \(P({\cal
  D})\). Its vertex set consists of all \emph{pairs}, that is 2-subsets of
the set \([0,5]\), regarded as the point set of \({\cal D}\). Two pairs
\(\alpha, \beta\) are adjacent in \(P({\cal D})\) if and only if
\(\alpha\cup\beta\) is a block in the design \({\cal D}\). The diagram
presented in Figure \ref{figure:5.7} confirms that \(P({\cal
  D})\cong\Delta\); similarly we define the graph \(P(\overline{\cal
  D})\). Because \({\cal D}\cong\overline{\cal D}\) we get that
\(P(\overline{\cal D})\) is also isomorphic to \(\Delta\). 

Finally we are ready to introduce the color graph \(NJ_{15}\)
(the notation will be clarified later on) as follows. Its vertex set
is \(\Omega\) of cardinality 15. Its basic graphs are: The reflexive
graph (consisting of all loops); the spread \(5\circ K_3\), visible in
Figure \ref{figure:5.7}, that is, the deleted one-factorization
\(F_1\) expressed in terms of pairs of \([0,5]\); the remaining copy
\(\Delta_0=\overline T_6\setminus 5\circ K_3\) of the graph \(\Delta\);
and two more copies of the graph \(\Delta\), namely \(\Delta_1=P({\cal
  D})\) and \(\Delta_2=P(\overline{\cal D})\).

\begin{proposition}
Let \(NJ_{15}\) be the color graph \((\Omega, \{\operatorname{Id}, S,
  \Delta_0, \Delta_1, \Delta_2\})\) of rank 5 with vertex set
\(\Omega\). Then
\begin{enumerate}
\item \label{org5f35e4d} \(\operatorname{Aut}(NJ_{15})\cong A_5\).
\item \label{orgee51622} \(NJ_{15}\) is the symmetrization of the AS \({\cal
     M}_{15}=V(A_5,\Omega)\), which is defined by the set of 2-orbits
of \((A_5,\Omega)\).
\end{enumerate}
\end{proposition}

\begin{proof}
Part \ref{org5f35e4d} follows from the previous proposition.

Note that clearly \(A_5\) has up to conjugacy a unique subgroup of
order 4, namely the elementary abelian group \(E_4\), which acts with
orbits of length 1 and 4. Let \((A_5,\Omega)\) be a transitive action
of \(A_5\) of degree 15. Let \(H=E_4\) be the stabilizer of a point
\(x\in\Omega\). We can identify \(H\) with the group \(\{e,(1,2)(3,4),
  (1,3)(2,4), (1,4)(2,3)\}\) in the action of \(A_5\) on \([0,4]\). Then it
is easy to check that in \((H,\Omega)\) we have 6 orbits of lengths
\(1^3,4^3\) with representatives (in brief notation) as follows:
\((12,34),(13,24),(14,23), (01,23),(01,24),(01,34)\). This proves that
the rank of \((A_5,\Omega)\) is 6. Because \(NJ_{15}\) admits
\((A_5,\Omega)\) its basic graphs are unions of basic relations in
\(V(A_5,\Omega)\). Comparing valencies of the regular graphs in \({\cal
  M}_{15}\) and \(NJ_{15}\) we come to the conclusion that the spread \(S\)
of valency 2 is split in \({\cal
  M}_{15}\) into directed graphs of valency 1, which are of the form
\(5\circ \vec C_3\). Here \(\vec C_3\) is a directed cycle of length 3.
\end{proof}

\begin{corollary}
The color graph \(NJ_{15}\) is the symmetrization of the
(non-symmetric) rank 6 AS \({\cal M}_{15}\). In other words, \(NJ_{15}\)
is a non-proper Jordan scheme.
\end{corollary}
The last corollary explains the essence of the notation \(NJ_{15}\).

\section{The non-symmetric rank 6 AS of order 15 revisited}
\label{sec:org3b77452}
\label{org105b323}
At this stage we are aware of the structure of the non-symmetric rank
6 AS \({\cal M}_{15}\) with the basic relations \(\operatorname{Id}, S_1,
  S_2, R_0, R_1, R_2\) of valencies 1,1,1,4,4,4, repsectively. Here the
graphs \((\Omega, R_i)\), \(i=0,1,2\) are all isomorphic to the DTG
\(\Delta\). These graphs were introduced evidently via the pictorial way
in Section \ref{orgaf4339e}.

It is also clear that the graphs \((\Omega, S_1)\) and  \((\Omega,
  S_2)\) are both isomorphic to \(5\circ\vec C_3\), while
\(S_2=S_1^T\). Recall that the symmetrization \(S=S_1\cup S_2\) defines
the spread \(5\circ K_3\). If we select an arbitrary copy of 5
connectivity components then, clearly, there are just two options for
its cyclic orientation.

\begin{proposition}
The cyclic orientation of one of the five components of the spread
\(S\) uniquely defines the cyclic orientation of the remaining four
components of \(S\).
\end{proposition}

\begin{proof}
Let \(\{_iS^c, i\in[0,4]\}\) be the set of components of \(S\). Because
\(A_5=\operatorname{Aut}({\cal M}_{15})\), the group \(A_5\) leaves invariant
each of the basic relations of \({\cal M}_{15}\), in particular the
spread \(S\) and its orientations \(S_1\) and \(S_2\). Clearly, \(A_5\) acts
naturally on the five components \(_iS^c\). Thus the stabilizer of,
say, the component \(_4S^c\) coincides with \(A_4\), and \(A_4\) has
orbits of length 1 and 4 in action on the components. The kernel of
the action \((A_4, {_4S^c})\) is \(E_4\) of order 4, thus the induced
action \((A_4, {_4S^c})\) coincides with the regular action of the cyclic
group \({\mathbb Z}_3\) of order 3. Now the kernel \(E_4\) acts
regularly on the remaining four components of the spread, while the
induced group \({\mathbb Z}_3\) acts semiregularly on the set \(\Omega\)
of cardinality 15 with 5 orbits of length 3. This semiregular
action defines an orientation of all 5 components of the spread.

To employ the latter conclusion practically we select an orientation
of \(_4S^c\) and exploit the fact that at the basic graph \((\Omega,
  \Delta_i)\) between any two components of \(S\) there appears a
one-factor. 
\end{proof}

Relying on Figure \ref{figure:delta-classical}, let us use the procedure
described above. Then we get the following description of the
relation \(S_1\):

\begin{align*}
S_1 = \{ & (0,10), (10,11), (11,0),\\
& (1,4),(4,9),(9,1),\\
& (2,13),(13,12),(12,2),\\
& (3,8),(8,7),(7,3),\\
& (5,14),(14,6),(6,5)\}.
\end{align*}
The relation is presented as a set of 15 pairs; it corresponds to
the partition of \(\Omega\) into fibers \(\{0,10,11\}\), \(\{1,4,9\}\),
\(\{2,13,12\}\), \(\{3,8,7\}\), \(\{5,14,6\}\). Here we use the map of the
elements of \(\Omega\) as in Table \ref{table:map-3}.

\begin{table}
\begin{center}
\begin{center}
\begin{tabular}{rl}
\hline
0 & 01,23\\
1 & 12,34\\
2 & 01,34\\
3 & 04,23\\
4 & 13,24\\
5 & 04,12\\
6 & 01,24\\
7 & 02,34\\
8 & 03,24\\
9 & 14,23\\
10 & 02,13\\
11 & 03,12\\
12 & 04,13\\
13 & 03,14\\
14 & 02,14\\
\hline
\end{tabular}
\end{center}
\end{center}
\caption{Map of 15 elements from \(\Omega\) \label{table:map-3}}
\end{table}
\begin{proposition}
\label{orgc664f85}
The AS \({\cal M}_{15}\) is non-commutative.
\end{proposition}

\begin{proof}
At this stage our proof is quite naive. Let us count the product
\(\Delta_0\cdot \Delta_1\), here \(\Delta_0\) and \(\Delta_1\) are
presented in Figures \ref{figure:deleted} and \ref{figure:5.7},
respectively. We arbitrarily select the starting point \(\{0,1\}\) of
\(\Omega\) and consider 16 paths of length 2 in the relational product 
\(\Delta_0\cdot \Delta_1\). Note that \(\{\{0,1\},\{2,5\},\{3,4\}\}\)
form one of the fibers of \(S\) in the current notation. Check that 4
paths end in the point \(\{3,4\}\) and none in \(\{2,5\}\). Similarly
considering \(\Delta_1\cdot\Delta_0\) we find 4 paths ending in
\(\{2,5\}\) and none in \(\{3,4\}\). 
\end{proof}

We postpone discussions of ideas for more intelligent proofs to
further sections. 
\begin{remark}
 Originally, the proof of Proposition \ref{orgc664f85} was
given around the year 2002, observing the tensor of structure
constants of \({\cal M}_{15}\), which was produced by COCO. 
\end{remark}
\section{Siamese combinatorial objects of order 15.}
\label{sec:orgfa7a5c1}
\label{org4e457e4}
The concepts which are briefly discussed in this section were
originally at the centre of the used methodology. It initiated
understanding that one of the two existing Siamese color graphs of
order 15 is a very promising candidate to form an example of a
proper Jordan scheme. Finally, this methodology, with all its quite
interesting details remains on the periphery of the current
paper. Nevertheless for completeness we present it below as a kind
of slight deviation. For more details the reader is referred to
\cite{KliRW05} and \cite{KliRW09}, as well as to the discussion at
the end of this paper.

\begin{define}
Let \(W=(\Omega, \{\operatorname{Id}_\Omega, S, R_1, \ldots,
  R_t\})\) be a color graph for which
\begin{enumerate}
\item \((\Omega, S)\) is an imprimitive disconnected SRG, i.e., a
partition of \(\Omega\) into cliques of equal size. It is usually
called a \emph{spread.}
\item For each \(i\), the graph \((\Omega,R_i)\) is an imprimitive DRG of
diameter 3 with antipodal system \(S\).
\item For each \(i\), the graph \((\Omega, R_i\cup S)\) is an SRG.
\end{enumerate}

Then W is called a \emph{Siamese} color graph of Siamese rank \(t\).
\end{define}

Usually we assume that all SRGs which form a Siamese color graph \(W\)
have the same parameter set, indicated by \((n,k,\lambda,\mu)\), where
\(n=|\Omega|\). 

\begin{define}
An association scheme \({\cal M}=(\Omega, \{
  \operatorname{Id}_\Omega, S_1, \dots, S_f, R_1, \dots, R_t \})\) is
called a \emph{Siamese AS} if \((\Omega, \{
  \operatorname{Id}_\Omega,\cup_{i=1}^f S_i,  R_1, \dots, R_t \})\) is a
Siamese color graph.
\end{define}

In other words we allow that in the process of the WL-stabilization
of a given Siamese color graph its spread \(S\) is split into the
union of suitable (regular) graphs \((\Omega, S_i)\), \(i=1,\dots,
  f\). In such a case we also say that the considered Siamese color
graph admits a Siamese AS.

\begin{define}
A Siamese color graph \(W\) is said to be \emph{geometric} if each of its
SRGs \((\Omega, R_i\cup S)\) is the point graph of a suitable
generalized quadrangle. It is called \emph{pseudo-geometric} if each SRG
is pseudo-geometric, in the general context of AGT, se, e.g.,
\cite{CamvL91}. 
\end{define}

\begin{proposition}
\label{org5527514}
Let \(W\) be a Siamese color graph of order \(n=(q^4-1)/(q-1)\) with
\(t=q+1\) basic relations \(R_i\), such that each SRG \(\Gamma_i=(\Omega,
  R_i\cup S)\) has parameters \((n, q(q+1), q-1, q+1)\) and is
geometric. For each point graph \(\Gamma_i\) construct the
corresponding GQ. Denote by \(B\) the union of all lines in the
resulting GQ. Then the incidence structure \({\frak S}=(\Omega, B)\) is a
Steiner design \(S(2,q+1, (q^4-1)/(q-1))\).
\end{proposition}
\begin{proof}
This is evident.
\end{proof}

Thus each geometric Siamese color graph \(W\) provides a Steiner
design \({\frak S}\) with a spread. Moreover, removing the spread, we
may arrange a partition of the remaining blocks in \({\frak S}\) into
\(t=q+1\) subsets, each of which is the set of lines of a GQ of order
q. We call it a Siamese partition of \({\frak S}\). 

\begin{proposition}
\label{orge3cf94f}
\begin{enumerate}
\item \label{orgf271c6f} Up to isomorphism there exist exactly two Siamese color graphs of
order 15.
\item Both color graphs are geometric.
\item The WL-closure of the first graph is the Schurian AS \({\cal
     M}_{15}\) of rank 6 introduced above.
\item The WL-closure of the second graph is the centralizer algebra of
rank 21 of the intransitive action of \(A_4\) with orbits of
lengths 3 and 12.
\end{enumerate}
\end{proposition}

\begin{proof}
Part \ref{orgf271c6f} was established with the aid of a computer, see
\cite{Rei03}. The other parts are considered in \cite{KliRW09}. 
\end{proof}

\begin{proposition}
\label{orga3f805d}
\begin{enumerate}
\item Up to isomorphism there exist exactly two Siamese Steiner triple
systems of order 15.
\item The first system is the classical structure
\(\operatorname{PG}(4,2)\) which is partitioned into a spread and
three isomorphic GQs of order 2 having a common spread.
\item The second system is a special partition of the STS(15) \(\#7\)
according to the classical catalogue \cite{MatPR83, MatPR83b}.
\item The automorphism groups of the existing Siamese partitions of
order 15 have orders 360  and 72, respectively (with known
structure of the groups and their actions).
\end{enumerate}
\end{proposition}

\begin{proof}
The proof appears in \cite{KliRW09}.
\end{proof}

\begin{remark}
\begin{enumerate}
\item There exist a couple of computer-free proofs of the two last
propositions. They will be presented in \cite{KliMRRW}.
\item One of the proofs in \cite{KliMRRW} is based on the elegant
computer-free description, by Alex Rosa, of all \(\operatorname{STS}(15)\) having
a substructure isomorphic to \(\operatorname{GQ}(2)\). This proof
does not require the knowledge of all 80
\(\operatorname{STS}(15)\), as in \cite{MatPR83}.
\item The advantage of the use of the language of Siamese partitions vs
Siamese color graphs is that the automorphism group of the
Siamese partition of a considered design coincides with the color
group of the corresponding color graph. In a certain sense, the
symmetries of the Siamese partition might be more visible for a
human.
\end{enumerate}
\end{remark}
\section{Color Cayley graph of order 12 over \(A_4\)}
\label{sec:org8d275df}
\label{org8e3100c}
Let us restart again from AS \({\cal M}_{15}\) and its group
\(A_5=\operatorname{Aut}({\cal M}_{15})\). As was mentioned, the group
\(A_5\) acts naturally on the set of five fibers which form a spread
of \({\cal M}_{15}\) or of its symmetrization \(NJ_{15}\). We are now
interested in the restriction of \(NJ{15}\) to four fibres of the
spread \(S\). Let us denote it by \(\operatorname{Con}_{12}\). Clearly,
\(\operatorname{Con}_{12}\) is a rank 5 color graph, its basic graphs
are regular of valencies 1,2,3,3,3, respectively.

We start with the consideration of the restriction of
\(\Delta=\Delta_0\). Let us assume that the fiber
\(_4S^c=\{(01,23),(02,13),(03,12)\}\), as in Figure \ref{figure:5.3},
is removed. Here the index 4 indicates that the element \(4\in[0,4]\)
does not appear in the labelling of the elements of the removed
fiber. The restriction \(TT_0\) is depicted in Figure
\ref{figure:8.1}. Here the four distinct kinds of designation for
the vertices indicate four different components of the spread \(S\)
with respect to the graph \(\Delta_0\).

\begin{figure}
\begin{center}
\input{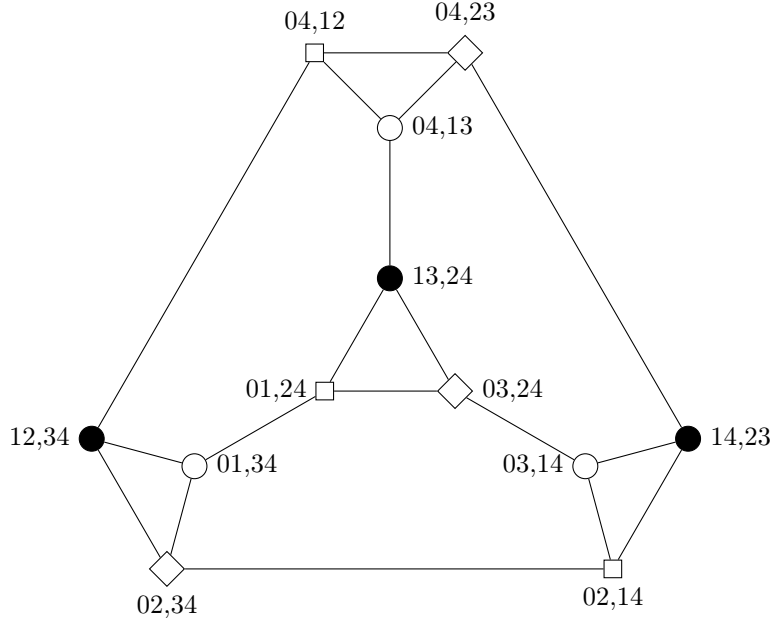}
\end{center}
\caption{Truncated tetrahedron \(TT_0\) \label{figure:8.1}}
\end{figure}

The reader is welcome to recognize in the diagram the \emph{truncated
tetrahedron,} this clarifies the used notation. Note the fibers of
\(\Delta_0\) remain to be the fibers of the graph \(TT_0\).

\begin{proposition}
\begin{enumerate}
\item The removal of one fiber from the graph \(\Delta\) leads to the
graph \(TT\) of valency 3 isomorphic to the truncated tetrahedron.
\item The graph \(TT\) is vertex transitive with the group
\(\operatorname{Aut}(TT)\cong S_4\).
\item The graph \(TT\) is of diameter 3 and is an antipodal cover of
\(K_4\).
\item The graph \(TT\) is not a DRG.
\item The graph \(TT\) admits a regular action of the group \(A_4\).
\end{enumerate}
\end{proposition}
\begin{proof}
Follows from visual observations.
\end{proof}

The last proposition implies that the color graph
\(\operatorname{Con}_{12}\) of order 12 is invariant with respect to
the regular group \(A_4\). We aim to present below a suitable 
description of \(\operatorname{Con}_{12}\) in terms, similar to the
methodology of Schur rings.

Each vertex of the graph \(TT_0\), as it is depicted in the Figure
\ref{figure:8.1}, is labelled by a symbol \(\{\{a,b\},\{c,4\}\}\),
where \(\{a,b,c\}\subseteq[0,3]\), \(|\{a,b,c\}|=3\). Let us bijectively
associate to this label (we use here again its full designation) the
ordered pair \((c,d)\), where \(\{d\}=[0,3]\setminus\{a,b,c\}\). This
new (intermediate) notation appears in Figure \ref{figure:8.2.a}.

The group \(A_4\) acts sharply 2-transitively on the set
\([0,3]\). Therefore there exists a natural bijection between ordered
pairs of distinct elements from \([0,3]\) and elements of \(A_4\). For
example, if \(g=\begin{pmatrix} 0&1&2&3\\x&y&z&u\end{pmatrix}\in
  A_4\), then we associate to \(g\) the pair \((x,y)\). In this fashion we
finally are able to regard the graph \(TT_0\) as a Cayley graph over
\(A_4\). The diagram, depicted in the Figure \ref{figure:8.2.b} presents
such a new labelling. We use here the classical notation for
permutations. 

\begin{figure}
\begin{center}
\input{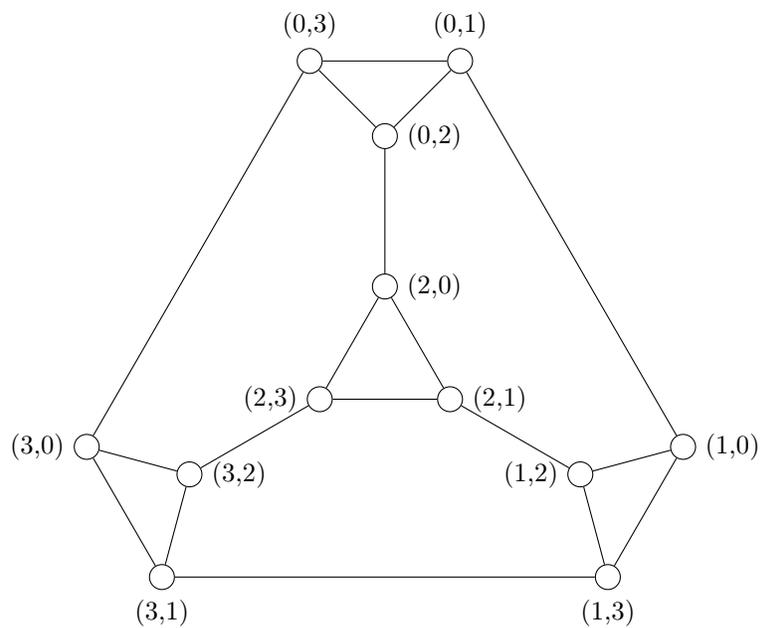}
\end{center}
\caption{Vertices as ordered pairs \label{figure:8.2.a}}
\end{figure}

\begin{figure}
\begin{center}
\input{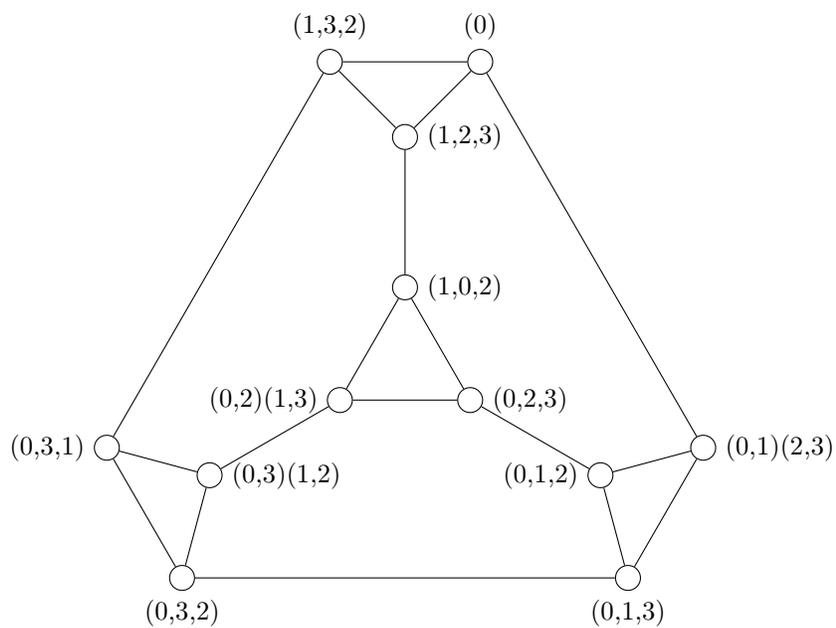}
\end{center}
\caption{Vertices as elements of \(A_4\) \label{figure:8.2.b}}
\end{figure}

The reader is welcome to observe from the diagram, which is depicted
in Figure \ref{figure:8.2.b} that the graph \(TT_0\) is a Cayley
graph with respect to the group \(A_4\), which is defined by the
connection set \(T_0=\{(0,1)(2,3), (1,2,3), (1,3,2)\}\). It is also
visible that the connection set \(\{(0,2,3), (0,3,2)\}\) defines a
disconnected graph \(4\circ K_3\) invariant under \(A_4\), which is a
spread of \(TT_0\). 

It turns out that the same procedure may be arranged for the graphs
\(\Delta_1\) and \(\Delta_2\), each time resulting in an isomorphic copy
of truncated tetrahedra \(TT_1\) and \(TT_2\), respectively, both are
invariant with respect to the same copy of the regular group \(A_4\). 

Let us define a rank 5 color Cayley graph
\[
   Y_{12}=\operatorname{Cay}(A_4, \{e\}, \{(0,2,3),(0,3,2)\}, T_0,
   T_1, T_2),
   \]
where
\begin{align*}
T_0&=\{(0,1)(2,3), (1,2,3),(1,3,2)\},\\
T_1&=\{(0,2)(1,3), (0,1,2), (0,2,1) \},\\
T_2&=\{(0,3)(1,2), (0,1,3), (0,3,1) \}.
\end{align*}

\begin{proposition}
\begin{enumerate}
\item The rank 5 color Cayley graph \(Y_{12}\) is isomorphic to the color
graph \(\operatorname{Con}_{12}\).
\item \(\operatorname{WL}(Y_{12})\) is the full rank 12 color Cayley graph over \(A_4\).
\item \(\operatorname{Aut}(Y_{12}) \cong A_4\).
\item \(\operatorname{CAut}(Y_{12})\cong (A_4\times {\mathbb
      Z}_3):{\mathbb Z}_2\) is a group of order 72,
which acts as \(S_3\) on the colors.
\item \(Y_{12}\) is not a Jordan scheme.
\item The Jordan stabilization of \(Y_{12}\) appears as the
symmetrization of \(\operatorname{WL}(Y_{12})\) and has rank 8.
\end{enumerate}
\end{proposition}

\begin{proof}
We start with the automorphism group \(\operatorname{Aut}(TT_0)\) of
order 24, which has a natural geometric action on the vertex set of
\(TT_0\). It defines an S-ring of rank 7 over \(A_4\), generated by
\(TT_0\). The valencies of this S-ring are 1,1,2,2,2,2,2, the basic
quantities are visible in Figure \ref{figure:8.2.b}. After that,
two analogous S-rings generated by \(TT_1\) and \(TT_2\) are
considered. The combination of the obtained information easily
leads to the necessary results. On this way there were used COCO,
hand computations, and {\sf GAP}. Finally, everything may be repeated by
routine reasonings and hand calculations in the group ring
\({\mathbb Z}[A_4]\).
\end{proof}
\section{Coherent configuration of order 15 and rank 21 with two fibers of size 3 and 12}
\label{sec:org92777de}
\label{orga11617b}
Starting from the group \(A_5=\aut({\cal M}_{15})\), where \({\cal
  M}_{15}\) is a non-commutative AS of rank 6, we reached its
symmetrization \(NJ_{15}\), regarded as a non-proper Jordan scheme of
rank 5, and having also group \(A_5=\aut(NJ_{15})\). Considering the
subgroup \(A_4\) in this action of degree 15, we investigated its
non-faithful rank 3 action on the orbit of length 3 and the regular
faithful action (of rank 12) on the orbit of length 12. Now we need
to consider simultaneously these two actions, that is, an intransitive
action of \(A_4\) with two orbits of lengths 3 and 12. In what
follows, the shorter orbit with be called \emph{island}, while the longer
orbit will be called \emph{continent} (this clarifies the notation
introduced in the previous section). We return (cf. Table
\ref{table:15aa.map}) to the intermediate notation for elements of
\([0,14]\) which was created by COCO, see Table \ref{table:9.1}.

\begin{table}
\begin{center}
$
\begin{array}{|r|r|}
\hline
  0&(0,1)\\
  1&(1,2)\\
  2&(0,2)\\
  3&(2,0)\\
  4&(2,3)\\
  5&(1,0)\\
  6&(0,3)\\
  7&(3,0)\\
  8&(3,1)\\
  9&(2,1)\\
 10&(1,3)\\
 11&(3,2)\\
 12&\{\{0,1\},\{2,3\}\}\\
 13&\{\{0,3\},\{1,2\}\}\\
 14&\{\{0,2\},\{1,3\}\}\\
 \hline
\end{array}
$
\end{center}
\caption{COCO labelling for elements in the action of \((A_4,\Omega)\) \label{table:9.1}}
\end{table}

\begin{proposition}
The rank of the intransitive action \((A_4,\Omega)\) is equal to 21.
\end{proposition}
\begin{proof}
Let \(\Omega=\Omega_1\cup \Omega_2\), \(\Omega_1=[0,11]\),
\(\Omega_2=[12,14]\). We are already aware that
\(\operatorname{rank}(A_4,\Omega_1)=12\), 
\(\operatorname{rank}(A_4,\Omega_2)=3\). 

The stabilizer of, say, \(b=\{\{0,1\},\{2,3\}\}\) (12 in COCO
notation), clearly coincides with the group \(E_4\) of order 4. This
group sends the pair \((0,1)\) to the pairs \((0,1)\), \((1,0)\), \((2,3)\)
and \((3,2)\). In other words, in COCO notation the 2-orbit of
\((12,0)\) of the group \((A_4,\Omega)\) contains the pairs \((12,i)\),
where \(i\in\{0,4,5,11\}\). This implies that there are three 2-orbits
from \(\Omega_1\) to \(\Omega_2\) and vice versa. Altogether we get
\(12+3+6=21\) 2-orbits.
\end{proof}

\begin{table}
\begin{center}
\begin{tabular}{|r|c|c||r|c|c|}
\hline
2-Orbit & Subdegree & Representative&
2-Orbit & Subdegree & Representative\\
\hline
0&1&(0,0)&12&1&(0,12)\\
1&1&(0,1)&  13&1&(0,13)\\

2&1&(0,2)&  14&1&(0,14)\\

3&1&(0,3)&  15&4&(12,0)\\
4&1&(0,4)&  16&4&(12,1)\\
5&1&(0,5)&  17&4&(12,2)\\
6&1&(0,6)&  18&1&(12,12)\\
7&1&(0,7)&  19&1&(12,13)\\
8&1&(0,8)&  20&1&(12,14)\\
9&1&(0,9)& &&\\
10&1&(0,10)&&&\\
11&1&(0,11)& &&\\ 
\hline
\end{tabular}
\end{center}
\caption{2-orbits of the group \((A_4,\Omega)\) \label{table:9.2}}
\end{table}

Table \ref{table:9.2} shows notation and representatives of all
21 2-orbits. Note that here we get the following pairs \((i,j)\) of
antisymmetric 2-orbits \((R_i,R_j)\): \((1,3)\), \((2,6)\), \((7,10)\),
\((8,9)\), \((12,15)\), \((13,17)\), \((14,16)\), \((19,20)\).

\begin{corollary}
\begin{enumerate}
\item The symmetrization \(\tilde Y_{15}\) of \(Y_{15}=V(A_4,\Omega)\) has
rank 13.
\item The WL-stabilization of this symemtrization coincides with the
rank 21 CC \(Y_{15}=V(A_4,\Omega)\).
\end{enumerate}
\end{corollary}

The proof is evident.\noproof

COCO returns that \(Y_{15}\) has 108 nontrivial homogeneous mergings
with ranks varying from 3 to 6. Twelve of the rank 6 mergings have
automorphism groups of order 60 of rank 6 with valencies
1,1,1,4,4,4. Clearly, they are all isomorphic to the AS \({\cal
  M}_{15}\). Their symmetrization provide 12 copies of the rank 5
Jordan scheme \(NJ_{15}\). With the aid of \gap we detect 12 copies of
\(PJ_{15}=J_{15}\) with the group \(A_4\), which appear as mergings of
unicolor graphs of \(\tilde Y_{15}\). We aim to pick up just one copy
of \(NJ_{15}\) and those copies of \(J_{15}\), which are relatives in
some sense of the selected copy of \(NJ_{15}\). Everything will be
described in terms of \(\tilde Y_{15}\).

Using the established notation, we present a pictorial description of
all basic relations of the color graph \(\tilde Y_{15}\). 

Thus Figures \ref{figure:9.1.a} through \ref{figure:9.1.k}, depict
all symmetric basic graphs (non-reflexive) which constitute the
color graph \(\tilde Y_{15}\). We have:
\begin{itemize}
\item four copies of \(4\circ K_3\) consisting of disjoint triangles;
\item three copies of \(6\circ K_2\), that is, partial 1-factors with 12 vertices;
\item three copies of \(3\circ St_{1,4}\), here \(St_{1,4}\) is a \emph{star} with a
center of valency 4 and four leaves of valency 1.
\item one graph of type \(K_3\).
\end{itemize}

The first two kinds of graphs have three isolated vertices whereas
the last one has 12 isolated vertices. Together with the two
reflexive graphs we obtain the requested basic graphs of \(\tilde
  Y_{15}\).

These 13 basic graphs form a kind of "kindergarten constructor",
which permits to describe all \emph{intermediate} unicolor graphs as
unions of suitable basic graphs of \(\tilde Y_{15}\). Each such an
intermediate graph is invariant with respect to \(A_4\cong
  \operatorname{Aut}(Y_{15})\). Some suitable unions of the obtained intermediate
graphs might be invariant with respect to the transitive action
\((A_5, [0,14])\) regarded as overgroup of \((A_4, [0,14])\). We also
may manipulate by the knowledge of the color groups of \(\tilde
  Y_{15}\) and \({\cal M}_{15}\). Each element of the color group of the
considered color graph, by definition, sends any intermediate graph
to its isomorphic copy. 

With the understanding of the rules of the game let us now introduce
a certain new color graph \(\operatorname{Pre}_{15}\), which will be a
suitable merging of \(\tilde Y_{15}\) and will be jargonically called
the \emph{pregraph} of order 15.

The main advantage of \(\operatorname{Pre}_{15}\) will be that it has
to contain (quite transparently) a copy of \(NJ_{15}\) and a few
copies of \(J_{15}\). 

In a certain sense the object \(\operatorname{Pre}_{15}\) will serve
as methodological bridge between all computer aided activities,
which led to the discovery of \(J_{15}\), and further theoretical
reasonings. The fact that each intermediate graph may be pretty
well obtained with the aid of the diagrams mentioned above is
expected to be helpful for the reader.

\begin{figure}
\begin{center}
\input{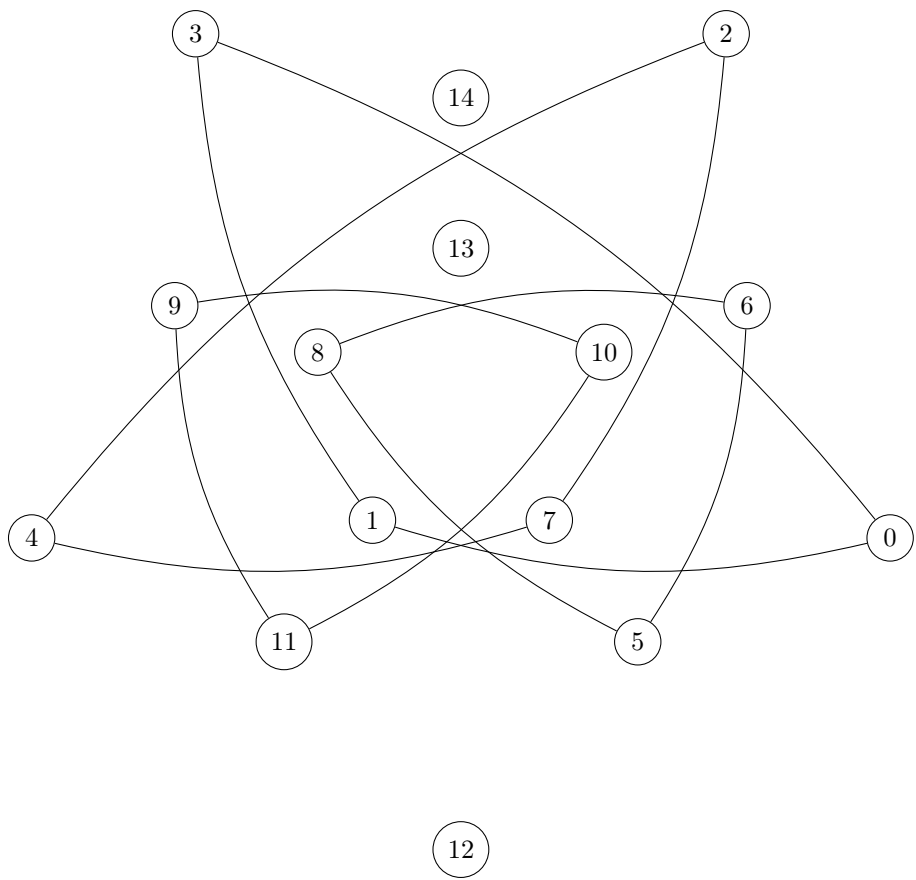}
\end{center}
\caption{Basic graph \(\tilde R_1\cong 4\circ K_3\) \label{figure:9.1.a}}
\end{figure}

\begin{figure}
\begin{center}
  \input{figure.9.1.b.tex}
\end{center}
\caption{Basic graph \(\tilde R_2\cong 4\circ K_3\)}
\end{figure}

\begin{figure}
\begin{center}
\input{figure.9.1.c.tex}
\end{center}
\caption{Basic graph \(\tilde R_4\cong 6\circ K_2\)}
\end{figure}

\begin{figure}
\begin{center}
  \input{figure.9.1.d.tex}
\end{center}
\caption{Basic graph \(\tilde R_5\cong 6\circ K_2\)}
\end{figure}

\begin{figure}
\begin{center}
  \input{figure.9.1.e.tex}
\end{center}
\caption{Basic graph \(\tilde R_7\cong 4\circ K_3\)}
\end{figure}

\begin{figure}
\begin{center}
\input{figure.9.1.f.tex}
\end{center}
\caption{Basic graph \(\tilde R_8\cong 4\circ K_3\)}
\end{figure}
\begin{figure}
\begin{center}
\input{figure.9.1.g.tex}
\end{center}
\caption{Basic graph \(\tilde R_{11}\cong 6\circ K_2\)}
\end{figure}

\begin{figure}
\begin{center}
\input{figure.9.1.h.tex}
\end{center}
\caption{Basic graph \(\tilde R_{12}\cong 3\circ St_{1,4}\)}
\end{figure}

\begin{figure}
\begin{center}
\input{figure.9.1.i.tex}
\end{center}
\caption{Basic graph \(\tilde R_{13}\cong 3\circ St_{1,4}\)}
\end{figure}

\begin{figure}
\begin{center}
\input{figure.9.1.j.tex}
\end{center}
\caption{Basic graph \(\tilde R_{14}\cong 3\circ St_{1,4}\) \label{figure:9.1.j}}
\end{figure}

\begin{figure}
\begin{center}
\input{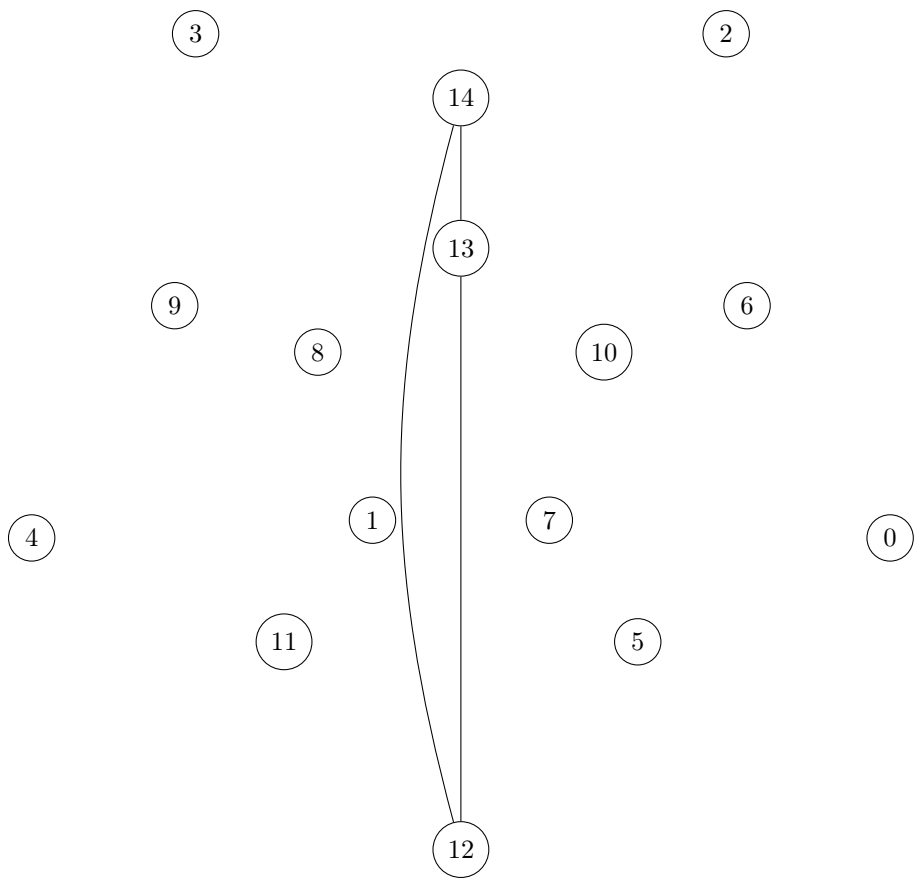}
\end{center}
\caption{Basic graph \(\tilde R_{19}\cong K_3\) on the island \label{figure:9.1.k}}
\end{figure}

\section{Rank 10 pregraph of order 15: Model "Island and Continent"}
\label{sec:org3fe023e}
\label{org94d2fe9}
We are using notation for \(\Omega=[0,14]\) and the relations \(\tilde
  R_i\) on \(\Omega\) as introduced in the previous section, relying on
the results obtained with the aid of COCO.

Let us start with the description of all unicolor graphs forming the
"pregraph" \(Pre_{15}\) announced above. Recall that \(\Omega=[0,14]\)
is partitioned as \(\Omega=\Omega_1\cup \Omega_2\), where
\(\Omega_1=[0,11]\), \(\Omega_2=[12,14]\). The pregraph \(Pre_{15}\) will
consist of the following parts:
\begin{itemize}
\item the \emph{"continent"} \(Con_{12}\), that is, the full color subgraph of
\(Pre_{15}\), induced by \(\Omega_1\);
\item the \emph{"island"} \(Isl_3\), that is, the full color subgraph of
\(Pre_{15}\) induced by \(\Omega_2\);
\item the \emph{bridges}, that is the three unicolor graphs which appear as
ingredients of \(Pre_{15}\) in its subgraph \(K_{3,12}\), that is the
full bipartite color graph between the continent and the island.
\end{itemize}

Let us first discuss some general issues of a possible independent
interest.

Let \(H\) be a group, \((H,H)\) its right regular
representation. Here  \(g\in H\) acts by right multiplication:
 \(h^{g} = hg\) for \(h\in H\). Denote
by \(Thi(H) = V(H,H)\) the centralizer algebra of \((
  H,H)\), the \emph{thin} AS in P.-H. Zieschang's terminology, see, e.g.,
\cite{Zie05}, which is defined by the group \(H\).

Let \(Hol(H)=\caut(Thi(H))\), that is the color group of the thin
\correction{scheme} \(Thi(H)\).

\begin{proposition}
\begin{enumerate}
\item The group \(Hol(H)\) is isomorphic to a semidirect product
\correction{\(H\rtimes\aut(H)\)}, where \(\aut(H)\) is the automorphism group of \(H\).
\item There exists a natural bijection between orbits of the group
\((\aut(H), H)\) and the 2-orbits of the group \((Hol(H),H)\).
\item \(|Hol(H)| = |H|\cdot |\aut(H)|\).
\end{enumerate}
\end{proposition}

\begin{proof}
We admit that this definition of the \emph{holomorph} \(Hol(H)\) of a group
\(H\) is given in somewhat non-traditional terms in comparison with
the classical definition. Nevertheless the results follow
immediately from the definition of the color group of an AS and its
application to the provided definitions vs classical ones.
\end{proof}

\begin{proposition}
\begin{enumerate}
\item \correction{\(Hol(A_4)\cong A_4\rtimes S_4\)} is a group of order 288.
\item \(\aut(A_4)\) acts transitively on the elements of order 1, 2 and
3, while the subgroup \(Inn(A_4)\cong A_4\) has two orbits of
length 4 on the elements of order 3.
\item \(Hol(A_4)\) has a subgroup of index 2, the group of order 144,
which is isomorphic to \(A_4\times A_4\).
\end{enumerate}
\end{proposition}

\begin{proof}
The proof is a straightforward exercise in elementary group theory.
\end{proof}

Let us now return to the consideration of the rank 21 CC
\(Y_{15}=V(A_4, \Omega_{15})\).

\begin{proposition}
\begin{enumerate}
\item \label{orga3049dc}The color group \(\caut(Y_{15})\) is isomorphic to the
split extension \(A_4:(S_3\times S_4)^+\). Here the group
\((S_3\times S_4)^+\) is the subgroup of index 2 formed by all even
permutations in the direct sum of \(S_3\) and \(S_4\).
\item \((A_4,\Omega_{15}) = \left< g_1, g_2\right>\), here
\begin{align*}
g_1 &= (0,3,1)(2,9,5)(4,10,6)(7,11,8)(12,14,13)\\
g_2 &= (0,6,2)(1,8,4)(3,5,7)(9,10,11)(12,13,14),
\end{align*}
\item \(\caut(Y_{15}) = \left< g_1,g_2, g_3, g_4, g_5\right>\), where
\begin{align*}
g_3 &= (0,5)(1,9)(2,3)(4,11)(6,7)(8,10)(12,13,14)\\
g_4 &= (0,3,7)(1,5,10)(2,9,11)(4,8,6)(12,13,14)\\
g_5 &= (0,6,5,9)(1,11,7,4)(2,8)(13,14)
\end{align*}
\item The group \(\caut(Y_{15})\) has the following orbits on the 21
colors of \(Y_{15}\) (we use the same labelling of 2-orbits that
was presented above):
\begin{itemize}
\item 0
\item 1,   2,   7,   6,   9,   8,   3,  10
\item 4,   5,  11
\item 12,  13,  14
\item 15,  17,  16
\item 18
\item 19,  20
\end{itemize}
\end{enumerate}
\end{proposition}

\begin{proof}
Parts b-d, which are of a technical nature, were obtained with the
aid of a computer.

To prove \ref{orga3049dc}, first let us embed \(\caut(Y_{15})\) into the direct sum of
the groups \(\caut(Y_{15,3})\) and \(\caut(Y_{15,12})\). Clearly,
\(\caut(Y_{15,3})=S_3\) is a group of order 6, while \(\caut(Y_{15,12})
  = Hol(A_4)\) is a group of order 288. Because \(E_4\) is a
characteristic subgroup of \(A_4\), while the action of \(A_4\) on
\(Y_{15,3}\) appears as the homomorphic image of the group \(A_4\), the
group \(\caut(Y_{15,12})\) has \(\caut(Y_{15,3})\) as its homomorphic
image. This implies that \(\caut(Y_{15})\), indeed, appears as a
subdirect product of \(\caut(Y_{15,3})\) and \(\caut(Y_{15,12})\).

Note that the quotient group of \(A_4\) (in its action on \(Y_{15,3}\))
with respect to the group \(E_4\) is  isomorphic to \({\mathbb
  Z}_3\). In other words, each permutation from \(({\mathbb
  Z}_3,Y_{15,3})\) is induced by a suitable element of the regular
action of \(A_4\) on \(Y_{15,12}\). On the other hand it is clear that
any transposition from \((S_{3}, Y_{15,3})\) regarded as
\(\caut(Y_{15,3})\), cannot  be induced by a suitable element of \(A_4\),
because \(A_4\) does not contain a normal subgroup of order 2. This
implies that the kernel \(K\) of the action of the group
\(\caut(Y_{15})\) on \(\Omega_{15,3}\) is a group of index 2 in \(Hol\). 

At this stage we are forced to determine the subgroup \(K\) of index 2
in \(Hol(A_4)\) which acts as the kernel of the considered subdirect
product in its action on \(\Omega_{15,12}\). It turns out that there
is just one candidate for this role, namely \(A_4: Inn(A_4) \cong
  A_4\times A_4\). To justify the latter claim, note that there are
exactly two normal subgroups isomorphic to \(A_4\) in \(Hol(A_4)\);
these are the initial regular group \(A_4\) and its counterpart, the
left regular representation. Because \(A_4\) does not contain any
subgroup of order 6, it is clear that each of these invariant
subgroups of \(Hol(A_4)\) is a subgroup of the kernel \(K\). Thus we
have only one choice: \(K = A_4\times A_4\).

Finally, we can determine the abstract structure of the group
\(\caut(Y_{15})\) of order 864. Indeed, \(\caut(Y_{15}) =
  \aut(Y_{15}):Q_{72} = A_4:Q_{72}\). The group \(Q_{72}\) acts on the set
of three bridges between the island and the continent as \(S_3\), and
acts as \(S_4\) on the set of four imprimitivity systems of type
\(4\circ K_3\) in \(A_4\). Similar reasonings show that we get for
\(Q_{72}\) a description in the form of a subdirect product of \(S_3\)
and \(S_4\) with common quotient \({\mathbb Z}_2\). It turns out that
the group \correction{\((S_3\times S_4)^+\)} is the only suitable candidate for \(Q_{72}\).
\end{proof}

\begin{remark}
The provided justification of a) allows to reach a computer-free
proof of the other parts in the last proposition. However again we
prefer to keep the selected proportion of the style of reasoning for
the entire text.
\end{remark}

Let us also consider the symmetrization \(\tilde Y_{15}\) of
\(Y_{15}\). This is a rank 13 color graph with the same group
\(\aut(\tilde Y) \cong A_4\). 

\begin{proposition}
\begin{enumerate}
\item The group \(\caut(\tilde Y_{15})\) is isomorphic to the group
\(\caut(Y_{15})\), which was described above.
\item The groups \(\caut(\tilde Y_{15})\) and \(\caut(Y_{15})\) literally
coincide, regarded as permutation groups of degree 15, acting on
the set \(\Omega=\Omega_{15}\).
\item The quotient group \(Q_{72} = \caut(\tilde Y_{15})/\aut(\tilde
     Y_{15})\) acts on the set of 13 symmetric colors of the graph
\(\tilde Y_{15}\) intransitively with the following orbits:
\begin{itemize}
\item two orbits on loops of lengths 3 and 12;
\item one orbit on the edges of the graph \(K_3\) on the set \(\Omega_{15,3}\);
\item one orbit consisting of three bridges \(Br_i\), aka
\(3\circ St_{1,4}\),  each bridge
between three vertices of \(\Omega_{15,3}\) and twelve vertices
of \(\Omega_{15,12}\), having valency 4 on the island and valency
1 on the continent;
\item one orbit including four graphs of kind \(4\circ K_3\) on the
set \(\Omega_{15,12}\);
\item one orbit consisting of three graphs of kind \(6\circ K_2\) on
the set \(\Omega_{15,12}\).
\end{itemize}
\end{enumerate}
\end{proposition}
\begin{proof}
For us it was more convenient to rely on the results of
computations. For this purpose a special ad hoc simple program was
written, see discussion at the end of the text.

Theoretical justification is based on the elaboration of the
arguments presented above. Namely, two steps were fulfilled:
\begin{itemize}
\item to prove that each permutation from \((\caut(\tilde
    Y_{15}),\Omega_{15})\) appears from the group
\((\caut(Y_{15}),\Omega_{15})\);
\item to prove that the action of \(\caut(Y_{15})\) on the thirteen colors
of \(\tilde Y_{15}\) does not have a non-trivial kernel; in other
words, there is no non-identical permutation in \(\caut(Y_{15})\) which
transposes some pairs of antisymmetric colors of \(Y_{15}\) while
fixing all symmetric colors.
\end{itemize}
\end{proof}

Now we are ready to consider the procedure of construction of a copy
of the pregraph \(Pre_{15}\) and evaluation of its symmetry. Recall
that the correct language to proceed with the graphs on
\(\Omega_{15,12}\) is to interpret them as Cayley graphs over
\(A_4\). The graphs \(6\circ K_2\) correspond to connection sets
consisting of single involutions in \(A_4\). The graphs \(4\circ K_3\)
correspond to connection sets \(\{x,x^{-1}\}\), where \(x\) is an
element of order 3 in \(A_4\).

Relying on very naive reasonings one can consider the following
sequence of steps on the way toward the desired pregraph:
\begin{itemize}
\item In step zero of the construction we select the two reflexive
relations on the island and the continent. There is exactly one
option.
\item On the first step we select one graph of the form \(4\circ K_3\) to
fulfill the role of the main part of the spread (on the
continent). There are four such options.
\item On the second step select a correspondence between the remaining
relations \(4\circ K_3\) with the relations \(6\circ K_2\). There are
\(3!=6\) such options.
\item Finally observe that we are getting on the continent a spread of
valency 2 and three copies of TT of valency three.
\end{itemize}

Clearly, according to the "principle of multiplication" there are
\(4\cdot 3! = 24\) formal possibilities to construct the desired
pregraph. 

It turns out that not all of them provide isomorphic color graphs
and not all are leading to the desired conditions for the pregraph.

\begin{proposition}
\begin{enumerate}
\item Under the action of the group \(\aut(Y_{15})=\aut(\tilde Y_{15})\)
the 24 possible color graphs of rank 10 split into three orbits
of lengths 4, 8, and 12.
\item Only the color graphs in the orbit of length 4 provide a
possibility to get as their merging the desired \(J_{15}\).
\end{enumerate}
\end{proposition}

\begin{proof}
A purely theoretical proof given in terms of the analysis of injective
mappings from the \([0,2]\) into the set \([0,3]\) in principle may be
provided based on the model to be described in \cite{KliMR}, a refined
version of the current text in preparation.

Here we rely on the computations with the aid of {\sf GAP}. Namely, three
partitions of the colors of \(Y_{15}\), in our initial labelling, are
given below. 

\begin{itemize}
\item A representative of the orbit of length 4 is [ [ [ 0 ], [ 1, 3 ], [ 2, 6, 11 ], [ 4, 8, 9 ], [ 5, 7, 10 ], [ 12, 15 ], [ 13, 17 ], [ 14, 16 ], 
[ 18 ], [ 19, 20 ] ].
\item A representative of the orbit of length 8 is
[ [ 0 ], [ 1, 3 ], [ 2, 4, 6 ], [ 5, 8, 9 ], [ 7, 10, 11 ], [ 12, 15 ], [ 13, 17 ], [ 14, 16 ], 
  [ 18 ], [ 19, 20 ] ].
\item A representative of the orbit of length 12 is
[ [ 0 ], [ 1, 3 ], [ 2, 4, 6 ], [ 5, 7, 10 ], [ 8, 9, 11 ], [ 12, 15 ], [ 13, 17 ], [ 14, 16 ], 
  [ 18 ], [ 19, 20 ] ].
\end{itemize}
\end{proof}

Thus finally we select the above representative in the orbit of
length 4, construct the corresponding rank 10 graph of order 15 and
call it the \emph{pregraph} of order 15, denoted by \(Pre_{15}\). 

\begin{proposition}
\begin{enumerate}
\item \(\aut(Pre_{15}) \cong A_4\).
\item \correction{\(\caut(Pre_{15}) \cong (S_3\times (A_4\rtimes S_3))^+\).} Here \(A_4\) is the
regular action of \(A_4\); \correction{\(A_4\rtimes S_3\)} is the subgroup of index 4 in
\(Hol(A_4)\) which appears from the action via conjugation of a
subgroup \(S_3\) in \(S_4=\aut(A_4)\) on \(A_4\). The group \(H^+\) of a
permutation group \((H,Col)\), as was before in this text, consists
of all even permutations in \(H\), this time acting on the set
\(Col\) of colors.
\end{enumerate}
\end{proposition}

\begin{proof}
\gap returns the following list of generators of the group
\((\caut(Pre_{15}), \Omega_{15})\):
\begin{align*}
&(12,14,13),\\
&(1,3)(2,7)(5,11)(6,10)(8,9)(13,14),\\
&(1,7,6)(2,3,10)(4,11,5),\\
&(0,6,2)(1,8,4)(3,5,7)(9,10,11)(12,13,14).
\end{align*}
The quotient group \(Q_{18}\) in its action on the set of colors of
\(\tilde Y_{15}\) has the generators \((5,7,6)\), \((2,3)(6,7)\), \((1,3,2)\). 

Extra proceeding of the generated groups on \(\Omega_{15}\) and colors
of \(\tilde Y_{15}\) implies b). The proof of a) follows implicitly
from the previous sections. Clearly it was also justified with the
aid of a computer.
\end{proof}

Summing up, let us describe the relations of \(Pre_{15}\) evidently in
terms of \(Y_{15}\):
\begin{itemize}
\item two reflexive relations \(Id_{12}=R_0\) and \(Id_3=R_{18}\), that is,
the identities on the continent \(\Omega_1\) and island \(\Omega_2\);
\item spread \(S_{12}\) on the continent, \(S_{12}=\tilde R_8=R_8\cup
    R_9\cong 4\circ K_3\);
\item a truncated tetrahedron \(TT_0\) on the continent, here
\(TT_0=R_2\cup R_6\cup R_5 = \tilde R_2\cup R_5\);
\item two truncated tetrahedra on the continent, namely \(TT_1=R_1\cup
    R_3\cup R_4 = \tilde R_1\cup R_4\) and \(TT_2=R_7\cup R_{10}\cup
    R_{11} = \tilde R_7\cup R_{11}\);
\item a spread on the island; \(S_3=R_{19}\cup R_{20} = \tilde R_{19}\);
\item three bridges between the island and the continent, namely
\(Br_0=R_{13}\cup R_{17} = \tilde R_{13}\),
\(Br_1=R_{12}\cup R_{15} = \tilde R_{12}\),
\(Br_2=R_{14}\cup R_{16} = \tilde R_{14}\).
\end{itemize}

Altogether, 10 unicolor graphs are presented, forming the complete
color graph \(Pre_{15}\) of rank 10.

\begin{proposition}
\label{org5e01ae2}
\begin{enumerate}
\item \label{org1b8d1ac} The introduced pregraph \(Pre_{15}\) appears as the merging of some
colors of \(\tilde Y_{15}\) and has rank 10.
\item \label{orge7eadc5} \(\wl(Pre_{15}) = Y_{15} = V(A_4,\Omega)\).
\item \label{orgd29c58d} \(\aut(Pre_{15})\cong A_4\), where \((A_4,\Omega)\) is the
intransitive action of degree 15 with orbits of lengths 12 and 3.
\item \label{org5857fa3} \(\caut(Pre_{15})\) is isomorphic to the stabilizer of \(Pre_{15}\)
in the full color group \(\caut(Y_{15})\), where \(Pre_{15}\) is
regarded as a union of colors in \(\tilde Y_{15}\).
\end{enumerate}
\end{proposition}

\begin{proof}
Parts \ref{org1b8d1ac}, \ref{orge7eadc5}, \ref{orgd29c58d} appear as simple repetition of some
claims presented before. Justification of \ref{org5857fa3} again, was
implicitly discussed before.
\end{proof}

\begin{proposition}
\begin{enumerate}
\item The merging of colors of \(Pre_{15}\) presented below and denoted
by \({\frak X}_{15}\) provides a copy of \(NJ_{15}\), that is the
symmetrization of \({\cal M}_{15}\).

Here
\(S=S_{12}\cup S_3\),
\(\Delta_0=TT_0\cup Br_0\),
\(\Delta_1 =TT_1\cup Br_1\),
\(\Delta_2=TT_2\cup Br_2\),
\(Id_{15} = Id_3 \cup Id_{12}\).

\item \label{orgb2fb622} \(\wl({\frak X}_{15}) \cong {\cal M}_{15}\).
\end{enumerate}
\end{proposition}

\begin{proof}
We claim that \(\aut({\frak X}_{15}) \cong A_5 = \left<g_1,
  g_2,g_3\right>\). Here 
\begin{align*}
g_1 &= (0,1,3)(2,5,9)(4,6,10)(7,8,11)(12,13,14)\\
g_2 &= (0,2,6)(1,4,8)(3,7,5)(9,11,10)(12,14,13)
\end{align*}
are our initial generators of \((A_4,\Omega_{15})\), while \(g_3\)
together with \(g_1,g_2\) generates the group \(A_5\), here 
\[
  g_3 = (1,3)(2,5)(4,13)(6,12)(7,11)(10,14).
  \]
While the results were attained with the aid of {\sf GAP}, easy hand
inspection may be done by the reader. Again, \ref{orgb2fb622} follows from the
previous considerations.
\end{proof}

For the reader's convenience, below we provide the diagram of the
graph \(\Delta_1\) (see Figure \ref{figure:delta.15.1}) on the same canvas of the diagrams which were
presented in Section 9. Note that the graph \(\Delta_0\) in the
current description coincides with the classical diagram of
\(\Delta\).

\begin{figure}
\begin{center}
\input{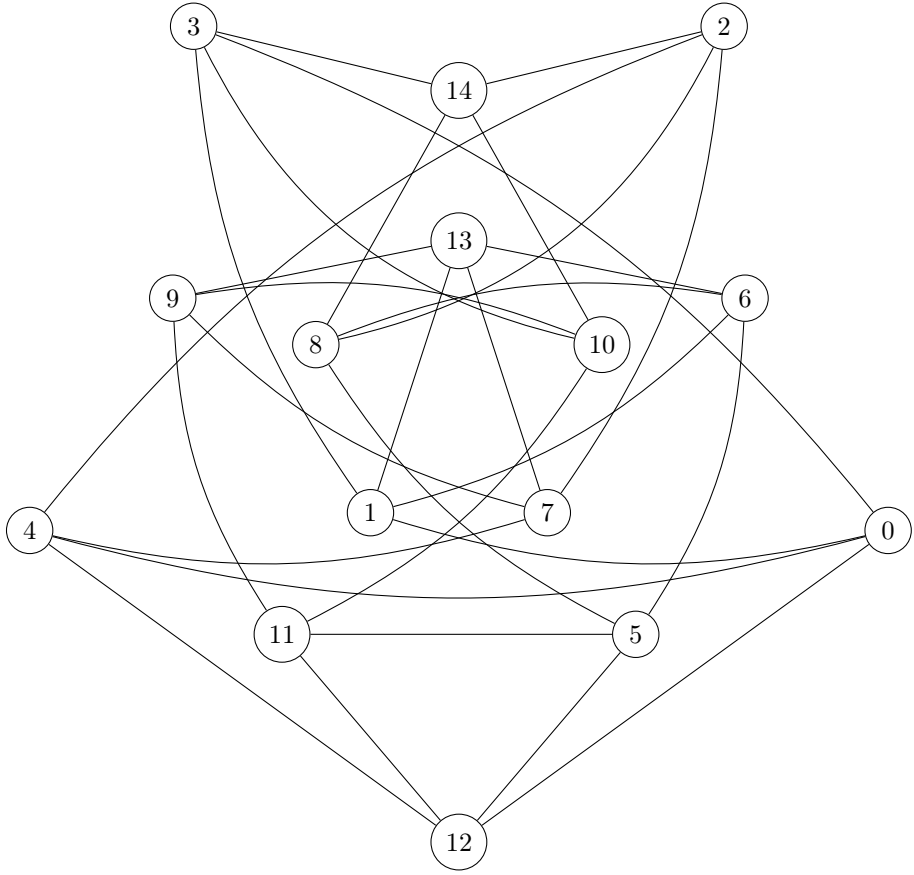}
\end{center}
\caption{Distance transitive graph \(\Delta_1\) \label{figure:delta.15.1}}
\end{figure}

\begin{figure}
\begin{center}
\input{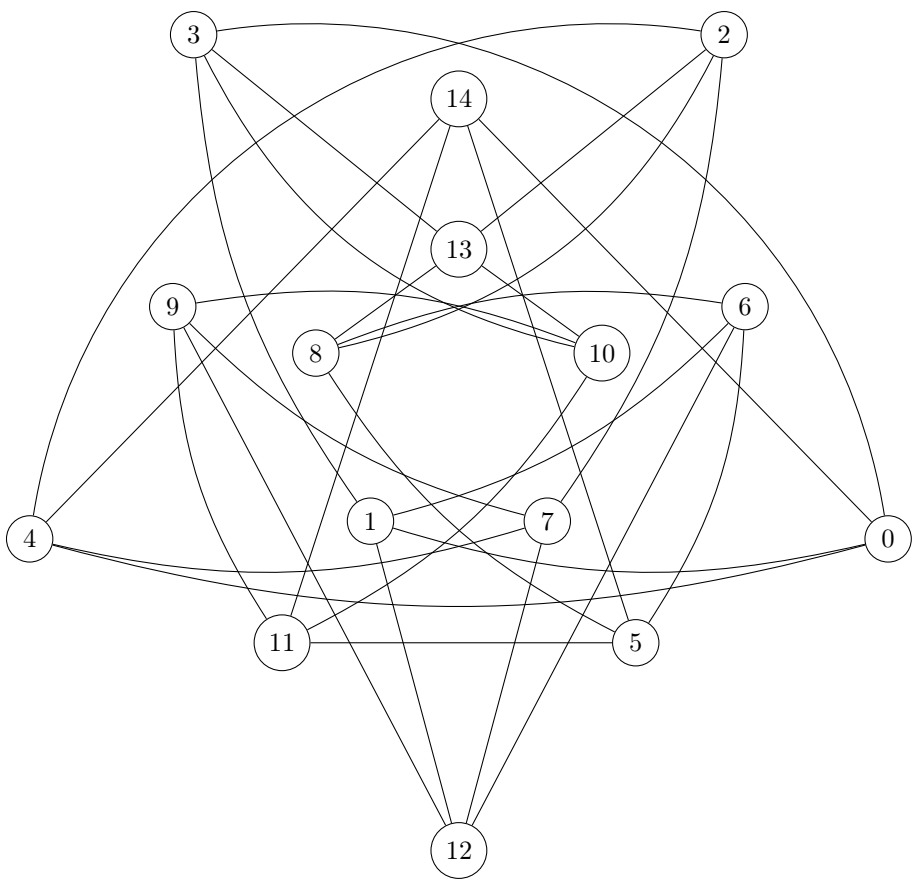}
\end{center}
\caption{Basic graph \(\Delta_1^1\) as merging of colors in \(Pre_{15}\) \label{figure:delta.15.1.1}}
\end{figure}

\begin{proposition}
\begin{enumerate}
\item \label{orgc4c1f8c}Any union of the truncated \correction{tetrahedron} with a bridge in the
pregraph \(Pre_{15}\) provides an isomorphic copy of the DTG \(\Delta\).
\item \label{org1fdb29d}There are 9 copies of the graph \(\Delta\), which appear as unions of
colors of \(Pre_{15}\).
\item \label{orgabbd6c9}The group \(\caut(Pre_{15})\) acts transitively on the set of 9
copies of \(\Delta\), which may be glued from the colors of \(Pre_{15}\).
\end{enumerate}
\end{proposition}

\begin{proof}
As it was \correction{mentioned} before, the result originally was obtained with the aid of
{\sf GAP}. Nevertheless we invite the reader to witness the proof as
follows.

Consider the permutation \(g_4=(1,10)(2,6)(3,7)(4,11)(8,9)(12,14)\) acting on the set
\(\Omega_{15}\). Check that \(\tilde g_4\),  the induced action of \(g_4\) on the
colors of \(Pre_{15}\),  preserves \(TT_0\) as well as
\(Br_0\). This implies that \(\left<A_4,g_4\right>\) is a group of order
24, which is isomorphic to \(S_4\).

Thus, the observed group serves as the stabilizer of one copy of
\(\Delta\), which is glued from colors of \(Pre_{15}\). Because its
index in \(\caut(Pre_{15})\) is equal to 9, we justify \ref{orgabbd6c9}). This,
evidently, implies \ref{orgc4c1f8c}) and \ref{orgb2fb622}).
\end{proof}

At this point all details in a required "constructor", that is
unicolor graphs of \(Pre_{15}\) are obtained. Also some properties of
certain details are established. We go toward understanding of all
Jordan rank 5 schemes, which can be glued from the presented
constructor set.

\section{Switching between non-proper and proper Jordan schemes of order 15}
\label{sec:org4292925}
\label{orgaf6cddd}
Using the unicolor graphs in the selected copy of \(Pre_{15}\), let us
consider all possibilities to get, as  mergings of \(Pre_{15}\), color
graphs of rank 5 consisting of \(Id_{15}\), the spread \(S=S_{15}\), and
three copies of a regular graph of valency 4. For a moment no extra
requirements are formulated for the obtained merging.

Clearly, we are forced to glue \(Id_{12}\) with \(Id_3\), as well as
\(S_3\) with \(S_{12}\). Thus the only remaining degree of freedom is to
make an assignement between three bridges and three graphs of
valency 3 on the continent.

\begin{proposition}
\begin{enumerate}
\item \label{org4e1acf7} There exist exactly six possibilities to glue from the unicolor
graphs of \(Pre_{15}\) a rank 5 color graph consisting of five regular
graphs of valencies \(1,2, 4^3\), one of them the identity \(Id_{15}\),
another one \(S_{15}=5\circ K_3\).
\item \label{org0a9ec3d} At least one of the above possibilities provides a copy of a
non-proper Jordan scheme \(NJ_{15}\) of rank 5.
\item \label{org4611dac} In fact there are at least three copies of \(NJ_{15}\), which are
obtained as mergings of \(Pre_{15}\).
\end{enumerate}
\end{proposition}

\begin{proof}
Part \ref{org4e1acf7}) is evident, because there are exactly \(3!=6\) bijections
between the bridges and the graphs of valency 3 on the continent.

One way to prove \ref{org0a9ec3d}) is to rely on Proposition
\ref{org5e01ae2}. Another way would be to rely on the use of a
computer, showing that a suitable merging satisfies our
requirements. 

To prove \ref{org4611dac}), let us remember the history of the described copy of
\(NJ_{15}\) in terms of \(Pre_{15}\). Then we consider the action of
\(\caut(Pre_{15})\) on the colors of \(Pre_{15}\) and conclude that this
\(NJ_{15}\) is in the orbit of length 3. Because the action of
\(\caut(Pre_{15})\) on the colors is induced by the action of \(\caut(Pre_{15})\) on
the set \(\Omega_{15}\), all three color graphs in the obtained orbit
are combinatorially isomorphic.
\end{proof}

It remains to understand what is going on with the three other rank 5
color graphs which are glued from \(Pre_{15}\). For this purpose let
us introduce on \(Pre_{15}\) the combinatorial operation of \emph{bridge
switching}, denoted by \(Sw_i\), \(i=0,1,2\). The switching \(Sw_i\)
preserves the bridge \(Br_i\) and transposes the two other
bridges. Note that, by definition, each switching \(Sw_i\) fixes all
other unicolor graphs. We also do not claim that the operation
\(Sw_i\) is induced by a suitable permutation acting on 15
points. Clearly, all switchings generate the symmetric group \(S_3\)
acting faithfully on the set of bridges.

Taking into account that each element of this group \(S_3\) acts on
the set of colors of \(Pre_{15}\), we are allowed to consider the group
\((\tilde S_3,Pre_{15})\). Here the notation \(\tilde S_3\) stresses
that the group acts on the colors, however this time not on the
vertices of \(Pre_{15}\).

\begin{proposition}
\begin{enumerate}
\item \label{org6f46495} Each transposition \(Sw_i\) from \((\tilde S_3, Pre_{15})\) moves the
color graph \(NJ_{15}\) to another rank 5 color graph \(J_{15,i}\),
which is a merging of \(Pre_{15}\).
\item \label{orga2e80b7} Each graph \(J_{15,i}\), \(i\in\{0,1,2\}\), contains three isomorphic
copies of the regular graph \(\Delta\) of order 15 and valency 4.
\item \label{org664be18} Each element  of the cyclic subgroup \((\tilde {\mathbb Z}_3,
     Pre_{15})\) moves the graph  \(NJ_{15}\) to another isomorphic copy
of this graph.
\item \label{orgc6b9949} The group \((\tilde S_3,Pre_{15})\) has two orbits of lengths 3
each on color graphs of rank 5, consisting of the images of
\(NJ_{15}\) and graphs in the set
\(\{J_{15,0},J_{15,1},J_{15,2}\}\).
\item \label{org2d825cf} All color graphs in the second orbit of length 3 are
combinatorially isomorphic.
\end{enumerate}
\end{proposition}
\begin{proof}
Parts \ref{org6f46495},  \ref{orga2e80b7},  \ref{orgc6b9949},  \ref{org2d825cf} are simply reformulations of previous result in
combination with the definition of switching and Part \ref{org664be18}.

To prove  \ref{org664be18},  consider the structure of \(\caut(Pre_{15})\) and conclude
that the cyclic group of order 3, which acts only on the island, is
a subgroup of group \(\caut(Pre_{15})\) in its action on
\(\Omega_{15}\). 
\end{proof}

\begin{corollary}
The cyclic group of rotations on the island cyclically moves the
three copies of \(NJ_{15}\), which are glued from the "details" in
\(Pre_{15}\). 
\end{corollary}
The proof appears as a combination of arguments in the two previous
propositions. \noproof

Let us now consider the color graph \(J_{15,0} = \{Id_{15}, S_{15},
  \Delta_0, \Delta'_1, \Delta'_2\}\), here the edge set of \(\Delta'_1\)
consists of the relations \(R_1,R_3,R_4,R_{14}, R_{16}\), the edge set of \(\Delta'_2\) of the
relations \(R_7,R_{10},R_{11}, R_{12}, R_{15}\). At the same time the copy \(\Delta_0\) is formed by
\(R_2, R_5, R_6,R_{13}, R_{17}\). 

For all three copies the relations \(R_0\cup R_{18}\) form the graph
\(Id_{15}\), while the relations \(R_8,R_9,R_{19},R_{20}\) provide a
common spread \(S=S_{15}\) for all the graphs \(\Delta_0, \Delta'_1,
  \Delta'_2\). 

\begin{proposition}
The introduced color graph \(J_{15,0}\) is a proper Jordan scheme of
order 15 and rank 5.
\end{proposition}

\begin{proof}
At this point the proof can be obtained with the aid of a computer
or via routine hand computations.
\end{proof}
We refer to the next section for a more interesting proof.

\begin{proposition}
All color graphs \(J_{15,0},J_{15,1},J_{15,2}\) provide isomorphic
copies of proper Jordan scheme of order 15 and rank 5.
\end{proposition}
\begin{proof}
This follows from Proposition 11.4 together with the same argument
which were used to justify Proposition 11.2.
\end{proof}

Finally, following the style accepted in this text, we present the
diagram of the graph \(\Delta'_1\), see Figure
\ref{figure:delta.15.1.1}. 
The graph \(\Delta_1'\) together with its mate \(\Delta'_2\) and the
initial graph \(\Delta_0\) provide the trio or requested DTG's,
forming a copy \(J_{15,0}\) of the proper
Jordan scheme of order 15.

In principle, the diagram of \(\Delta'_2\) may be also provided at
this place. It is however omitted, stressing the fact that at the
current stage functions of the redundant set of the depicted
diagrams are practically exhausted.

A new, more formal style of the presentation is approached, starting
from the next section.

Recall that on the current canvas the graph \(\Delta_0\) coincides
with the canonical copy of \(\Delta\) depicted before.

\section{\(J_{15}\) is a proper Jordan scheme: Formal proofs}
\label{sec:orgbd18582}
\label{org4553fba} \label{org5fceb33}
In the previous section we introduced the color graph
\(J_{15}=J_{15,0}\), which was obtained from the copy \(NJ_{15}\) of
non-proper Jordan scheme of order 15 with the aid of the switching
\(Sw_0\). Recall that this operation of switching leads to a new color
graph \(J_{15,0}\). It has three basic graphs isomorphic to the DTG
\(\Delta\), namely \(\Delta_0, \Delta'_1, \Delta'_2\), where
\(\Delta'_1=Sw_0(\Delta_1)\), 
\(\Delta'_2=Sw_0(\Delta_2)\).
In addition, the graph \(J_{15,0}\) contains the spread \(S\), which
coincides with the spread of \(NJ_{15}\) and the relation
\(\id_{15}\). As was discussed earlier, we are already aware that
\(J_{15,0}\) is also a Jordan scheme. Here we aim to justify this fact
on a purely theoretical level. The advantage of the proof below is
that the idea will also work for more general situation.

Let us call the DTG \(\Delta_0\) \emph{direct}, and call \(\Delta'_1,
  \Delta'_2\) \emph{skew} unicolor graphs in \(J_{15,0}\). Note that this
definition is conditional, reflecting the history of the appearance
of \(J_{15,0}\) (this time from \(NJ_{15}\) with the aid of the operation
\(Sw_0\)). In principle, the same color graph \(J_{15,0}\) my be
obtained from another initial copy of \(NJ_{15}\), using another
operation of switching.

\begin{lemma}
\label{orgb745be0}
Let \({\frak S}=(\Omega, \{\id_n, S, R_0, R_1, R_2\})\) be a symmetric
regular coloring of the complete graph with the vertex set \(\Omega\),
\(|\Omega|=n\). Denote \(R=R_0\cup R_1\cup R_2\). Assume that for any
\(i\in\{0,1,2\}\) the merging \({\frak S}_i=(\Omega, \{\id_n,S,R_i,
  R\setminus R_i\})\) is an association scheme. Then \({\frak S}\) is a
Jordan scheme.
\end{lemma}
\begin{proof}
Denote the adjacency matrices of the unicolor graphs as follows:
\(I_n=A(\id_n)\), \(B=A(S)\), \(A_i=A(R_i)\), \(i=0,1,2\), \(A=A(R)\). Denote
by \({\cal A}\) the linear span of the introduced matrices. By
assumption, \({\cal A}_i = \left<I_n, B, A_i, A-A_i\right>\) is the
BM-algebra of an AS \({\frak S}_i\). By Proposition
\ref{org163c63e} it
is sufficient to show that for any two basic matrices  \correction{\(A_i,A_j\in {\cal A}\) the square $(A_i+A_j)^2$ belongs to \({\cal A}\).}

\correction{If $i=j$, then $(A_i+A_j)^2\in {\cal A}_i\subseteq {\cal A}$. In the case of $i\neq j$ we distinguish two cases.}

\begin{itemize}
\item \correction{One of the two matrices, say $A_j$, is \(B\).  Then \((B+A_i)^2\in{\cal A}_i\le {\cal A}\).}
\item \correction{Both \(A_i,A_j\) are distinct from $B$. In this case \((A_i+A_j)\in {\cal A}_k\),  where \(k\ne\{i,j\}\). Therefore  \((A_i+A_j)^2\in {\cal     A}_k \subseteq {\cal A}\).}
\end{itemize}

\end{proof}

\begin{proposition}
The introduced color graph \(J_{15} = J_{15,0}\) is a Jordan scheme.
\end{proposition}

\begin{proof}
According to its construction, the unicolor graphs of \(J_{15}\) are
\(\id_{15}\), a spread, and three copies of the imprimitive antipodal
DTG \(\Delta\). Each of these copies generates an AS of order 15 and
rank 4 with basic graphs \(\id_{15}\), \(S\), a suitable DTG and its
"anti-spread" 
complement. Thus we can apply Lemma \ref{orgb745be0} and get the requested
result. 
\end{proof}

\begin{proposition}
The Jordan scheme \(J_{15}=J_{15,0}\) is proper.
\end{proposition}

\begin{proof}
At this moment we provide just a brief outline. Remembering the
obtained structure of \(J_{15}\) via the model "island and continent",
we may present all basic matrices of \(J_{15}\) in block form with
submatrices of size \(x\times y\), \(x,y\in\{3,12\}\). Manipulation with
these matrices  allows to prove that \(\wl({\cal A})\) is a CC
with two fibers of sizes 3 and 12.
\end{proof}

\begin{remark}
An interested reader may fulfill the suggested calculations on first
reading. We warn however that the requested result will be proved in
a more general form later on.
\end{remark}

Going toward this announced general situation, let us first consider one
more proper Jordan scheme \(J_{24}\) of rank 5 and order 24.

\section{The Klein graph of order 24 and valency 7}
\label{sec:org61fc84e}
\label{org5e0938c}
The example of the proper Jordan scheme \(J_{24}\), which is now
approached, played quite a significant role in order to understand
how the object \(J_{15}\) appears as the initial member of the first
infinite class of proper Jordan schemes of rank 5.

Here again we face as basic graphs three isomorphic copies of an
antipodal distance regular cover of a complete graph. In addition,
an extra advantage of the appearing DRG is that it is also
well-known in the framework of classical map theory.

Thus at the beginning we refer to the genus 3 orientable surface
which is usually represented as the Klein quartic. On this surface
there appears the Klein map with 24 heptagonal faces and 56
vertices. Its Schläfli symbol is \(\{7,3\}_8\). The point graph of
this map, the \emph{cubic Klein graph} with 56 vertices, has number
\(Fo_{56}B\) according to the famous Foster census \cite{Bow88}. The
description of this graph goes back to the classical paper
\cite{Kle879}. 

The \emph{dual Klein graph} \(\operatorname{Kle}_{24}\) has the above 24
heptagons as vertices, two heptagons are adjacent if they have a
common edge. Both Klein graphs have isomorphic automorphism groups
of order 336, while the symmetry group of both maps is twice smaller
and is isomorphic to the second smallest non-abelian simple group (of order
168), having incarnations \(\operatorname{PSL}(3,2) \cong
  \operatorname{PSL}(2,7)\).

A number of beautiful properties of the Klein maps and related
mathematical structures are considered in a remarkable collection of
papers \cite{Lev99}. It contains essays written by leading experts
in Geometry, Number Theory, Geometric Group Theory, Invariant
Theory, etc. This masterpiece appears on the edge between science
and art, see e.g., the paper by Fields medalist W.P. Thurston,
considering properties of the sculpture at MSRI, created by Helmut
Ferguson. The collection also contains translations into English of
the original paper \cite{Kle879} by F. Klein.

For us it was quite convenient to construct the graph
\(\operatorname{Kle}_{24}\), using the standard COCO
methodology. Initial objects on this way were the Fano plane
\(\operatorname{PG}(2,2)\), its incidence graph, well-known under the
name \emph{Heawood graph} denoted by \({\cal H}_{14}\), as well as
\(\aut(\operatorname{PG}(2,2)) = \operatorname{PSL}(3,2)\), regarded
as a permutation group of degree 7 and order 168. We rely on the
classical representation of the Fano plane with the aid of the
Singer cycle \((0,1,2,3,4,5,6)\) and difference set \(\{1,2,4\}\),
leading to the famous diagram below. Knowledge of  group-theoretical
information, related to
\(\operatorname{PSL}(3,2)=\aut(\operatorname{PG}(2,2))\), such as in
\cite{Pfe97} and \cite{Vis07}, and especially in \cite{JamL93}, is
also helpful.

\begin{figure}
\begin{center}
\input{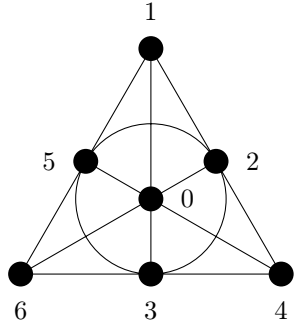}
\end{center}
\caption{Diagram of the Fano plane \(PG(2,2)\) \label{figure:13.1}}
\end{figure}

\begin{proposition}
Let \((G,[0,6])\) be the automorphsim group of the classical model of
the Fano plane, as above. Then
\begin{enumerate}
\item \(Z(G,[0,6]) = \frac{1}{168}(x_1^7 + 21 x_1^3x_2^2 + 56 x_1x_3^2 +
     42 x_1x_2x_4 + 48 x_7)\);
\item \(G = \left< (0,1,2,3,4,5,6),(1,2)(3,6)\right>\);
\item For elements of order 1,2,3 and 4 the relation of similarity in
\(S_7\) coincides with the relation of conjugacy in \(G\);
\item Permutations of order 7 split into two conjugacy classes of size
24 each;
\item An element of order 7 and its inverse lie in different conjugacy
classes.
\end{enumerate}
\end{proposition}

\begin{proof}
In principle, the proof might be a helpful straightforward exercise
on the intersection between permutation group theory, character
theory and finite geometries. Nevertheless, the reader is welcome to
consult the sources mentioned above, which provide all necessary
information. Here we refer to permutations  \(g_1,g_2,g_3,g_4\) as
representatives of the conjugacy classes of elements of orders
2, 3, 4
and 7, respectively.
\begin{align*}
g_1 &= (2,4)(5,6),\\
g_2 &= (1,2,4)(3,6,5),\\
g_3 &= (1,2,3,6)(4,5),\\
g_4 &= (0,1,2,3,4,5,6).
\end{align*}
\end{proof}

\begin{figure}
\begin{center}
\input{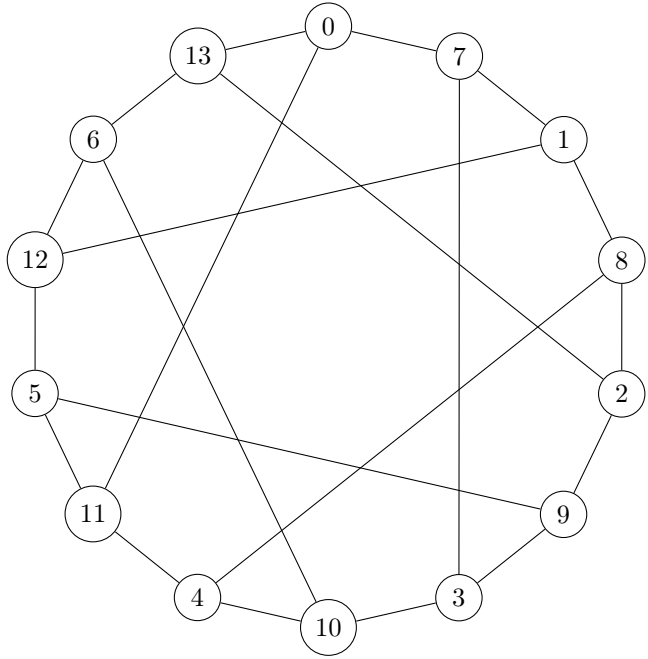}
\end{center}
\caption{The Heawod graph \({\cal H}_{14}\), the incidence graph of \(PG(2,2)\) \label{figure:13.2}}
\end{figure}

\begin{figure}
\begin{center}
\input{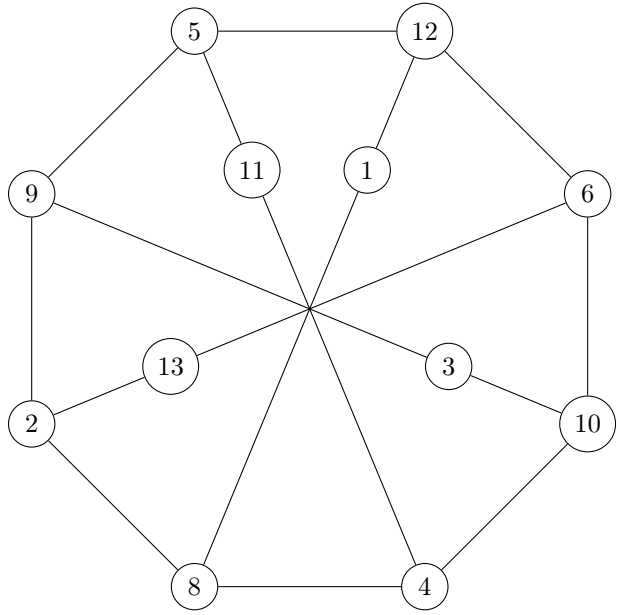}
\end{center}
\caption{Induced subgraph by \(V({\cal H}_{14})\setminus\{0,7\}\) \label{figure:13.3}}
\end{figure}

\begin{figure}
\begin{center}
\input{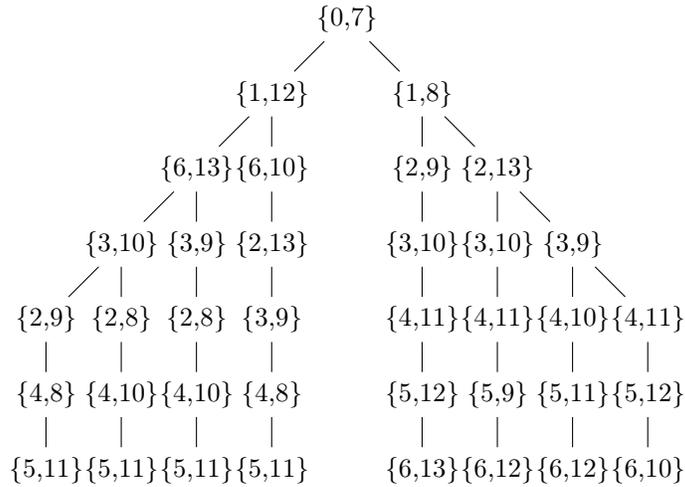}
\end{center}
\caption{Backtracking search of matchings containing the edge \(\{0,7\}\) \label{figure:13.4}}
\end{figure}

Now we turn to the consideration of the Heawood graph \({\cal
  H}_{14}\), which is depicted in Figure \ref{figure:13.2}. Note that
the line \(\{0,1,3\}\) is denoted by 7; images of this line with
respect to the permutation \(g_4\), acting on points, are
denoted by elements from \([7,13]\) consecutively, such that the
permutation \(\tilde g_4 = (0,1,2,3,4,5,6)(7,8,9,10,11,12,13)\) is an
automorphism of \({\cal H}_{14}\) of order 7.

Note also that an involution
\(t=(0,7)(1,13)(2,12)(3,11)(4,10)(5,9)(6,8)\) clearly belongs to
\(K=\aut({\cal H}_{14})\). Moreover the group \(D_7=\left<\tilde g_4,
  t\right>\) is the stabilizer of the Hamilton cycle in \({\cal
  H}_{14}\), visible on its diagram, presented in Figure
\ref{figure:13.2}. This observation allows to consider one of the
remarkable properties of the Heawood graph, see, e.g., the home page
of Andries Brouwer \cite{Bro}.

\begin{proposition}
\label{orgaf3386b}
\begin{enumerate}
\item \label{org110a9d2} \(K=\aut({\cal H}_{14}) \cong \operatorname{PGL}(2,7)\) is a group
of order 336.
\item \label{org277b063} The Heawood graph is a bipartite DTG with the parameters
(3,2,2;1,1,3).
\item \label{orga8e1c15} Each Hamilton cycle in \({\cal H}_{14}\) uniquely splits into
two perfect matchings.
\item \label{orgd85551c} The graph \({\cal H}_{14}\) contains exactly 24 Hamilton cycles, on
which the group \(K\) acts transitively. The stabilizer of one such
cycle is a regular action of the dihedral group \(D_7\) of order 14.
\item \label{orgf8b5442} \({\cal H}_{14}\) contains exactly 24 perfect matchings. Each
matching is the complement of a corresponding Hamilton cycle.
\item \label{org8a0c911} The 24 matchings of \({\cal H}_{14}\) split into 8 sets of three
pairwise edge-disjoint
matchings. The union of any two  matchings in each  set
forms Hamilton cycle.
\item \label{orgeee9bbd} The action of \(K\) on the set of eight sets of matchings is similar
to the action of \(\operatorname{PGL}(2,7)\) on the projective line
\(\operatorname{PG}(1,7)\).
\end{enumerate}
\end{proposition}
\begin{proof}
For \ref{org110a9d2}, \ref{org277b063}, see \cite{BroCN89} as well as \cite{GodR01}, p.97. To
prove \ref{orga8e1c15} note that \({\cal H}_{14}\) is a bipartite graph, which is
the incidence graph of a symmetric design.

 It is clear from the visual arguments mentioned above that \({\cal
  H}_{14}\) conntains at least one orbit of Hamilton cycles of length
24, with a representative given in Figure \ref{figure:13.2}. Thus
there is also at least one orbit \({\cal O}\) of 24
matchings. Clearly, each edge of \({\cal H}_{14}\) appears in 8
matchings exactly from \({\cal O}\). Let us fix an edge of \({\cal
  H}_{14}\), say \(\{0,7\}\) and construct the subgraph of \({\cal
  H}_{14}\) induced by \(V({\cal H}_{14})\setminus\{0,7\}\), see Figure
\ref{figure:13.3}. A bit annoying, though simple brute force
inspection, see Figure \ref{figure:13.4}, shows that there are
exactly 8 extensions of \(\{0,7\}\) to a perfect matching. This
proves \ref{orgd85551c} and \ref{orgf8b5442}.

For the proof of \ref{org8a0c911} consider the Hamilton cycle visible in Figure
\ref{figure:13.2}. Split it into two matchings \(M_1\) and \(M_2\) and
join to them the matching \(M_0\), also visible in the picture. Check
that \(M_0\cup M_1\) and \(M_0\cup M_2\) are also Hamilton cycles.

Finally, we get eight sets of disjoint matchings, three matchings in each
set. The group \(K\) acts on the set of these sets faithfully, as well
as the group \(PSL(2,7)\) of order 168. Thus we get one more
incarnation of the classical isomorphism \(PSL(2,7)\cong
  PGL(3,2)\cong PSL(3,2)\), cf., e.g., \cite{Hir98}, p.144.
\end{proof}

\begin{corollary}
The group \(K\) is isomorphic to \(PGL(2,7)\). 
\end{corollary}

\begin{proof}
This fact was already discussed implicitly. Now it follows from the
well-known catalogue of primitive permutation groups of small
degree, e.g., \cite{DixM96}.
\end{proof}

In the next section we will investigate the imprimitive action
\((PSL(3,2),\Omega)\), where \(\Omega\) is the unique orbit of perfect
matchings of \({\cal H}_{14}\), which, as we now know, has
length 24. Clearly, the stabilizer of an element in this action is
isomorphic to \({\mathbb Z}_7\) and acts semiregularly on the set
\([0,13]\), regarded as the vertex set of \({\cal H}_{14}\). Recall that \(M_0\) is the matching
visible in Figure \ref{figure:13.2}. 

\begin{proposition}
Let \(\Omega=\Omega_{24}\) be the set of perfect matchings of the
Heawood graph \({\cal H}_{14}\). Let \({\mathbb Z}_7 = \left<\tilde
  g_4\right>\), where \(\tilde g_4=(0,1,2,3,4,5,6)(7,8,9,10,11,12,13)\). Then
\begin{itemize}
\item \({\mathbb Z}_7\) has on \(\Omega\) six orbits of lengths
1,1,1,7,7,7. The representatives of these orbits are matchings
\(M_0,M_1,M_2,M_3,M_4,M_5\), respectively. Here \(M_1\) is the
matching containing the edge \(\{0,7\}\),  \(M_2\) contains
\(\{0,13\}\), while their union is the Hamiltonian cycle visible in
Figure \ref{figure:13.2}
\item The matchings \(M_3,M_4,M_5\) are as follows:
\begin{align*}
M_3&=\{ \{ 0, 7 \}, \{ 1, 8 \}, \{ 2, 13 \}, \{ 3, 9 \}, \{ 4, 10 \}, \{ 5, 11 \}, \{ 6, 12 \} \}\\
M_4&=\{ \{ 0, 7 \}, \{ 1, 12 \}, \{ 2, 8 \}, \{ 3, 10 \}, \{ 4, 11 \}, \{ 5, 9 \}, \{ 6, 13 \} \}\\
M_5&=\{ \{ 0, 7 \}, \{ 1, 12 \}, \{ 2, 13 \}, \{ 3, 9 \}, \{ 4, 8 \}, \{ 5, 11 \}, \{ 6, 10 \} \}
\end{align*}
\item The intersection of the edge set of the considered matchings with 
\[
    M_0=\{\{0,11\},\{1,12\},\{2,13\},\{3,7\},\{4,8\},\{5,9\},\{6,10\}\}
    \]
has sizes 7,0,0,1,2,4, respectively.
\end{itemize}
\end{proposition}

\begin{proof}
Check that three more matchings, besides those from the set
\(\{M_0,M_1,M_2\}\), indeed are matchings of \({\cal H}_{14}\). Clearly,
none of them is preserved by \({\mathbb Z}_7\). Because their
intersections with \(M_0\) have different size, they belong to
distinct orbits of \(({\mathbb Z}_7,\Omega)\). As \(3\cdot 1+3\cdot
  7=24\), we covered all orbits of \(({\mathbb Z}_7,\Omega_{24})\).
\end{proof}

We still did not provide our combinatorial definition of the Klein
graph \(\operatorname{Kle}_{24}\).

\begin{proposition}
Let \(\Omega=\Omega_{24}\) be the set of 24 matchings of \({\cal
  H}_{14}\). Let \((PSL(3,2),\Omega)\) be the transitive action of
degree 24. Then:
\begin{enumerate}
\item The group \((PSL(3,2),\Omega)\) has rank 6 with subdegrees \(1^3,7^3\).
\item This action is imprimitive with the imprimitivity system
consisting of eight blocks of size 3.
\item The 2-closure is twice larger, has order 336 and is isomorphic to
the group \(PGL(2,7)\).
\item The color group \(CAut({\cal M}_{24})=V(PSL(3,2),\Omega_{24})\) is
isomorphic to the group \(PGL(2,7)\times{\mathbb Z}_3\) of order 1008.
\item The algebraic group \(AAut({\cal M}_{24})\cong S_3\) and has order 6.
\end{enumerate}
\end{proposition}

\begin{proof}
The results were obtained with the aid of a computer, though there
is a good potential to get them independently by nice reasonings.
\end{proof}

\begin{proposition}
\begin{enumerate}
\item \label{orge5874c7} The three basic graphs of valency 7 in the above rank
6 AS \({\cal M}_{24}\) are isomorphic.
\item \label{orgf7bbad0} All these graphs, denoted by \(Kle_{24}\), are distance regular
covers of \(K_8\) with the intersection array (7,4,1;1,2,7).
\item \label{org359f1bf} The graph \(Kle_{24}\) is not distance-transitive.
\end{enumerate}
\end{proposition}
\begin{proof}
Part \ref{orge5874c7} follows from the transitivity of \(CAut({\cal M}_{24})\) on the basic
graphs of valency 7. A nice consideration of \(Kle_{24}\) as a DRG
appears in \cite{Jur95}. Part \ref{org359f1bf} also follows from the structure of
\(CAut(\operatorname{Kle}_{24})\). 
\end{proof}

We postpone further discussion of all structures presented here to
the concluding sections.

\begin{remark}
A. Jurišić proved in \cite{Jur95} the uniqueness of a
distance regular antipodal cover of \(K_8\). This graph is locally a
heptagon. Thus this result provides a purely combinatorial definition
of the DRG \(\operatorname{Kle}_{24}\) which does not depend on any activities with
the Klein quartic or, alternatively, with the Heawood
graph. Nevertheless we included the current section into the text,
considering it as a kind of a reasonably interesting deviation. 
\end{remark}

\section{The proper Jordan scheme \(J_{24}\) of order 24}
\label{sec:orgbef731b}
\label{org508eb9e}
The Klein graph \(\operatorname{Kle}_{24}\) is a crucial ingredient in the
description of the structures to be presented in the current
section. Before it was described implicitely. For this purpose the
graph \({\cal H}_{14}\) was considered. Although the Heawood graph is
a remarkable object of an independent interest it was playing just
an auxiliary role toward \(\operatorname{Kle}_{24}\).

Now the Klein graph will be constructed directly and, moreover, it
will appear as a basic graph of two Jordan schemes \(NJ_{24}\) and
\(J_{24}\) of order 24. Three parallel models will be presented. Each
time the obtained computer aided data will be supplied by a sketch
of reasonings supporting it. Full details of the related proofs will
be left to the reader.

One of the models will be called the local model. It will appear in
the same methodology "island and continent". The continental part of
\(\operatorname{Kle}_{24}\) turns out to be a vertex transitive graph \(Jur_{21}\) of
order 21 and valency 6. The diagram of \(Jur_{21}\), to be presented,
is very similar in spirit to the one given by A. Jurišić in
\cite{Jur95}, p. 32. This justifies the selected notation. The graph
\(Jur_{21}\) will be an analogue of the truncated tetrahedron
\(TT\). The two graphs together will serve to prepare the reader to
the acquaintance with a more general paradigm.

Let us start with the group \(G=PGL(2,7)\), in order to describe Model
A. \(G\) is a 3-transitive permutation group of order 336, acting on
the set \([0,7]\):
\begin{align*}
G &= \left<a_7, b_7, d_8\right>\\
a_7 &= (0,1,2,3,4,5,6)\\
b_7 &= (1,3,2,6,4,5)\\
d_8 &= (0,7)(1,6)(2,3)(4,5).
\end{align*}
(Notation here is shifted to \([0,7]\) according to COCO style, in
comparison with the action on \([1,8]\) used by Ch. Sims.)

Using COCO, we construct the induced action \((\tilde G,\Omega_A)\) on
the set \(\Omega_A=\Omega_{24,A}\) of the images of the canonical
cycle \(C_0\), depicted in Figure \ref{figure:14.1}. 

\begin{proposition}
\begin{enumerate}
\item \label{orgb29b5fc} The AS \({\cal M}_{24} = V(\tilde G, \Omega_A)\) is a rank 6
non-commutative scheme with valencies \(1^3,7^3\)
\item \label{org771a7b3} The representatives of the orbits of the stabilizer of the cycle
\(C_0\) are depicted in Figure \ref{figure:14.1}.
\item \label{org5d6ef7f} The symmetrization \(NJ_{24}\) of \(V(\tilde G,\Omega_A)\) is a
non-proper rank 5 Jordan scheme \(NJ_{24}\).
\item \label{org9e9b89e} The group \((\tilde G,\Omega_A)\) is 2-closed.
\end{enumerate}
\end{proposition}
\begin{proof}
The representatives were obtained with the aid of COCO. Clearly,
they belong to different orbits. Simple counting confirms that all
2-orbits are obtained. This proves \ref{orgb29b5fc}, \ref{org771a7b3}. Non-commutativity, which
follows from COCO protocols, may be confirmed by hand. The proof
of \ref{org5d6ef7f} is evident.

The computer-aided proof of \ref{org9e9b89e} will be accessed again below.
\end{proof}

\begin{figure}
\begin{center}
\input{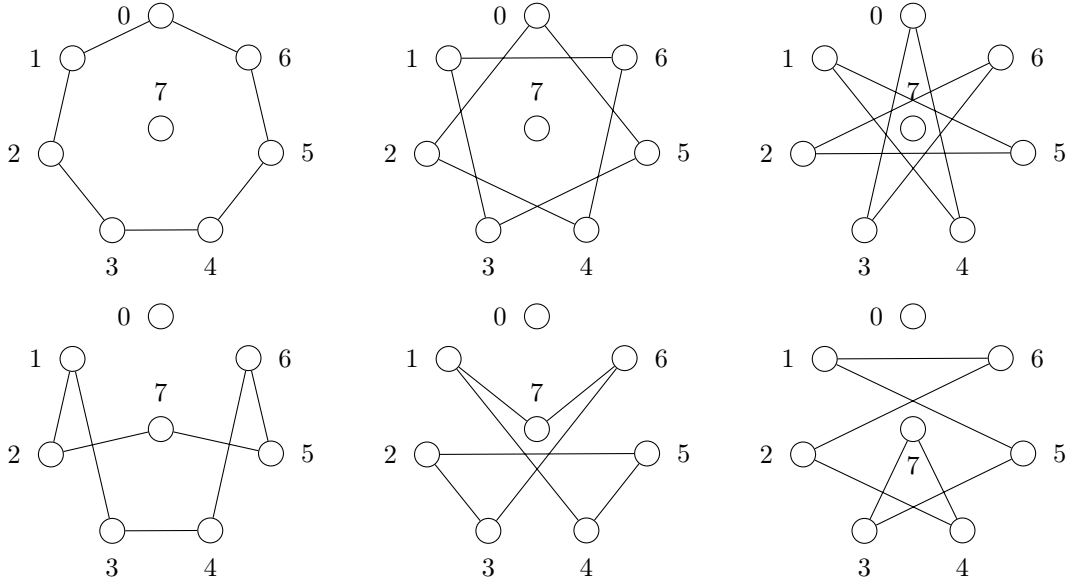}
\end{center}
\caption{Representatives of orbits of the stabilizer of cycle \(C_0\) in \(PGL(2,7)\) \label{figure:14.1}}
\end{figure}

Let us now approach the stabilizer \(G_7\) of the group \(G\). Clearly,
\(G=\left<a_7, b_7\right> = AGL(1,7)\) acts sharply 2-transitively on
the set \([0,6]\). Now consider
\(\Omega_B=\Omega_{B,21}\cup\Omega_{B,3}\). Here, \(\Omega_{B,21}\) is
the edge set of the complete graph \(K_7\) with the vertices \([0,6]\),
\(\Omega_{B,3}=\{C_0, C_1, C_3\}\) as in the first row of Figure
\ref{figure:14.1}. The next result was again obtained with the aid
of COCO.

\begin{proposition}
\begin{enumerate}
\item \label{orgdcab36c} The Schurian CC \(Y_{24}=Y_{B,24}=V(G_7,\Omega_B)\) has two fibers
of length 21 and 3 and rank 21.
\item \label{orgdeebd4b} \(\aut(Y_{24})=G_7\) is the same group of order 42.
\item \label{org1163d52} The CC \(Y_{24}\) has 27 AS mergings, among them three
non-commutative rank 6 mergings with the valencies \(1^3,7^3\) and
the group of order 336, which are all isomorphic.
\item \label{org11d1074} There are nine rank 4 AS mergings which provide non-Schurian
metrical schemes with valencies \(1,7,14,2\).
\item \label{org20d50d6} All mergings in \ref{org11d1074} are combinatorially isomorphic.
\item \label{org487f12d} The symmetrization \(\tilde Y_{24}\) is a rank 13 color graph.
\end{enumerate}
\end{proposition}

\begin{proof}
As was mentioned, the results are again a combination of diverse
kinds of activities. Some of them imitate the previous job for
\(Y_{15}\). The proof for \ref{org20d50d6} follows from \cite{Jur95}. 
\end{proof}

Let us now focus our attention on the continent \(\Omega_{B,21}\)
consisting of edges of \(K_7\). Starting from a suitable copy of
\(\operatorname{Kle}_{24}\), as above, we may restrict it to \(\Omega_{B,21}\) and
consider the induced subgraph \(\operatorname{Kle}_{21}\) of \(\operatorname{Kle}_{24}\).

\begin{proposition}
\begin{enumerate}
\item \label{org987ef0f}The graph \(\operatorname{Kle}_{21}\) is a vertex-transitive graph of valency 6
with the automorphism group of order 42.
\item \label{orgd2802e2}The \(WL(\operatorname{Kle}_{21})\) is a Schurian rank 12 AS with valencies
\(1^3,2^9\); it coincides with the centralizer algebra of
\((AGL(1,7),\Omega_{B,21})\).
\item \label{org4634071}The graph \(\operatorname{Kle}_{21}\) appears as a union of three 2-orbits of the
group \(AGL(1,7)\), one of them symmetric disconnected, which
defines the graph \(3\circ K_7\).
\item \label{org6bfa581}Two other 2-orbits, forming \(\operatorname{Kle}_{21}\), are antisymmetric and
connected, thus their symmetrization provides a graph of valency 4.
\item \label{org363680e}The graph \(\operatorname{Kle}_{21}\) is isomorphic to the graph \(Jur_{21}\), as it
appears in \cite{Jur95}.
\end{enumerate}
\end{proposition}
\begin{proof}
Parts \ref{org987ef0f}, \ref{orgd2802e2} were obtained with the aid of COCO. A computer-free
proof is left as an exercise to the reader.

For the proof of the remaining parts consider the diagram
below. Figure \ref{figure:14.2.a} presents the selected 2-orbit \(R_1\),
which defines the graph \(3\circ C_7\). Vertices \(\{a,b\}\) are denoted
\(ab\). To explain it, let us start
from the canonical cycle \(C_0\) on \([0,6]\). The vertices in the
middle cycle are edges in \(C_0\), two vertices are adjacent if they
form a consecutive path of length 2 in \(C_0\). In a similar manner
the remaining cycles in this figure can be explained.

\begin{figure}
\begin{center}
\input{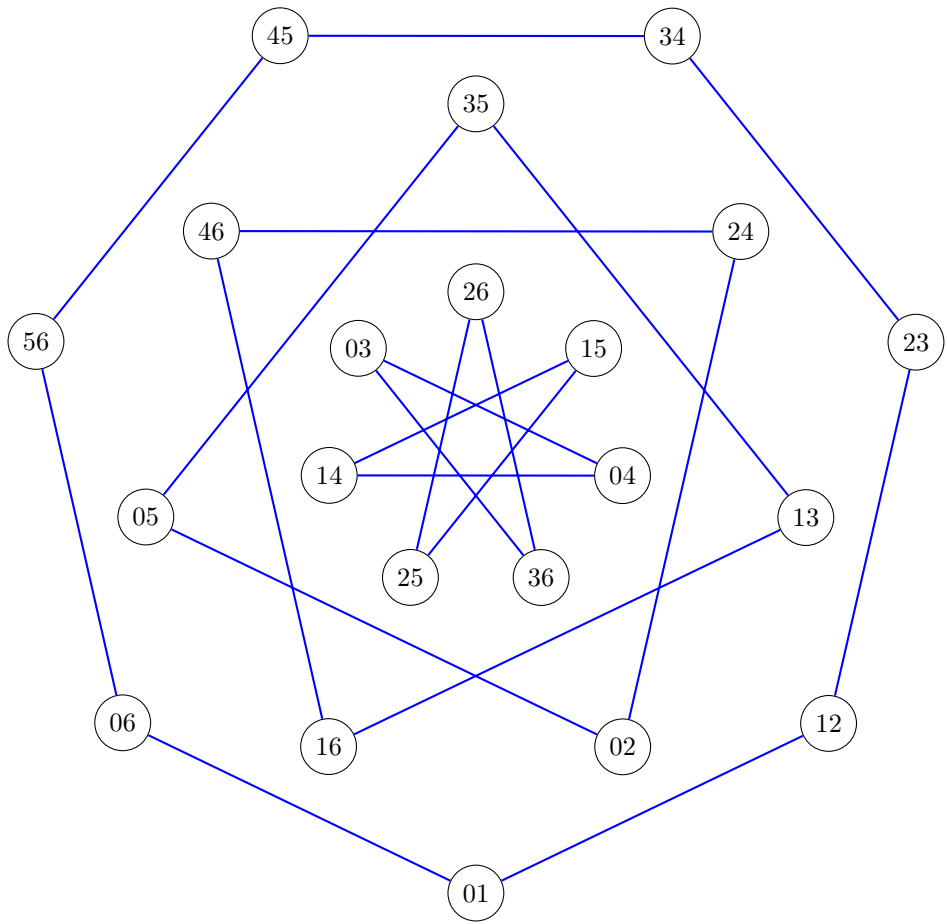}
\end{center}
\caption{2-orbit \(R_1\) of \((G_7,\Omega_{B,21})\) \label{figure:14.2.a}}
\end{figure}

\begin{figure}
\begin{center}
\input{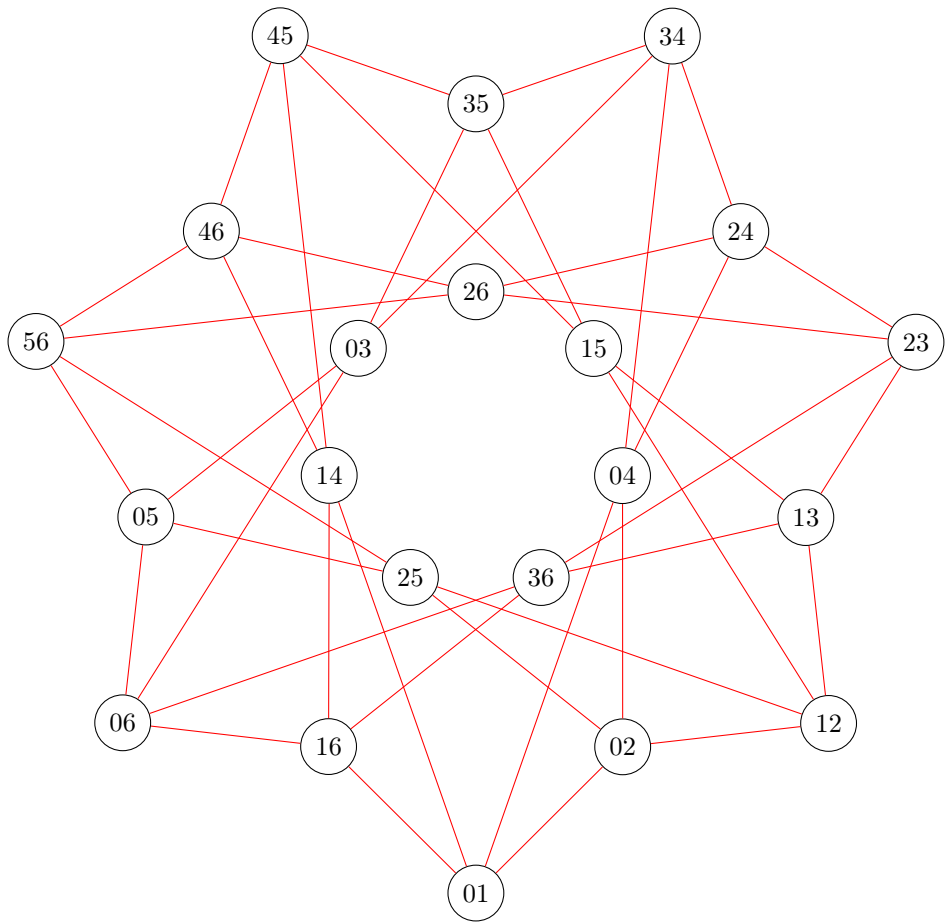}
\end{center}
\caption{Symmetrized 2-orbit \(\tilde R_2\) of \((G_7,\Omega_{B,21})\) \label{figure:14.2.b}}
\end{figure}

\begin{figure}
\begin{center}
\input{klein3.tex}
\end{center}
\caption{The graph \(\operatorname{Kle}_{21}\)}
\end{figure}

Figure \ref{figure:14.2.b} keeps the same notation as in the
previous diagram. The more sophisticated relation \(\tilde R_2\) that is depicted here
may be described in a reasonably human-friendly wording as well as
more formally.

In order to describe the relations \(R_1\) and \(\tilde R_2\)
theoretically, we define an absolute value on \(F={\mathbb Z}_7\): For
\(x\in F\) define
\[
  \left|x\right| = \{x,-x\}.
  \]
We can extend multiplication in \(F\) to absolute values, by
\(\alpha\cdot |x| = |\alpha x|\).

We take as vertices unordered pairs in \(F\). Two pairs
\(\{a,b\},\{c,d\}\) are in relation \(R_1\) if they have one element in
common, and they span the same distance:
\[
  (\{a,b\},\{c,d\})\in R_1 \iff |\{a,b\}\cap\{c,d\}|=1 \wedge |a-b|=|c-d|.
  \]

If two pairs are in relation \(\tilde R_2\) they also need to have one
element in common. Here, \(\{a,b\}\) is in relation to \(\{b,c\}\) if
either \(|a-c|=|a-b|\), or \(|a-c|=|b-c|\). We get four neighbors for
each vertex.

Note that both relations are invariant under addition and
multiplication of constants, i.e., under \(AGL(1,7)\).

Observing the two figures implies a proof of \ref{org4634071} and \ref{org6bfa581}. To prove \ref{org363680e},
the reader is welcome to overlay both figures and to realize that
the result literally coincides with Figure 3.4 b) in \cite{Jur95}.
\end{proof}

\begin{proposition}
The imprimitive system of the AS \({\cal M}_{21}=V(AGL(1,7),\Omega_{B,21})\)
is presented in Table \ref{table:14.3}. We use the same notation as
in the diagrams above.
\end{proposition}
\begin{proof}
The group \(AGL(1,7)\) has one conjugacy class of subgroups of
order 3. It consists of seven groups. The members of this conjugacy
class bijectively correspond to the blocks of the system presented
in the list.
\end{proof}
\begin{table}
 \begin{align*}
&  (\{0,1\},  \{2,6\},  \{3,5\})\\
&  (\{0,2\},  \{4,5\},  \{3,6\})\\
&  (\{0,3\},  \{4,6\},  \{1,2\})\\
&  (\{0,4\},  \{1,4\},  \{5,6\})\\
&  (\{0,5\},  \{2,4\},  \{1,4\})\\
&  (\{0,6\},  \{1,5\},  \{2,4\})\\
&  (\{1,6\},  \{3,4\},  \{2,5\})\\
 \end{align*}
\caption{Permutations of order 3 on \(E(K_7)\) defining the imprimitivity system of \({\cal M}_{21}\) \label{table:14.3}}
\end{table}
\begin{remark}
In Table \ref{table:14.3} on purpose we present permutations of order 3 on the set \(E_{K_7}\).
\end{remark}
In a manner similar to the description of the graph \(\operatorname{Kle}_{21}\),
eight more isomorphic copies invariant under the group
\((AGL(1,7),[0,6])\) may be constructed. Taking two suitable copies
(details omitted on purpose) we can consider the induced action of
this group on the set \(\Omega_{B,3}=\{C_0,C_1,C_3\}\) of the three
7-cycles defined above. This allows to describe three bridges
between the island and continent very naturally: a bridge exists
between a cycle and an edge if the cycle contains
the edge. Clearly we get three stars \(St_{1,7}=K_{1,7}\). Again the natural
union of bridges with three selected copies of \(\operatorname{Kle}_{
  21}\) provides a
non-commutative AS of rank 6. Switching with respect to one bridge
provides a proper Jordan scheme \(J_{24}\).

\begin{proposition}
\begin{enumerate}
\item \label{org77a5f9c}The Models A and B of the Klein graph \(\operatorname{Kle}_{24}\) are isomorphic.
\item \label{org3d6d610}Switching with respect to a bridge in the local model B provides
a proper Jordan scheme \(J_{24}\).
\item \label{org2ab7826}All Jordan schemes of rank 5 invariant under
\((AGL(1,7),\Omega_{B,24})\) are isomorphic to \(NJ_{24}\) or \(J_{24}\).
\item \label{orgd3f8160}\(\aut(J_{24}) = AGL(1,7)\).
\end{enumerate}
\end{proposition}

\begin{proof}
The justification was obtained with the aid of a computer. In
principle there is a possibility for a computer-free proof already
on this stage. Nevertheless we warn the reader again that a more
general proof will be given later on.
\end{proof}

Let us now, following P. Gunnells \cite{Gun05}, approach one more way
toward \(J_{24}\), with the aid of another global Model C.

So let \(F_7={\mathbb Z}_7\) be the finite field of integers
modulo 7. Consider the set \(F_{7}^2\setminus\{\vec 0\}\) of non-zero
vectors \((a,b)\), \(a,b\in F_7\). Two vectors \(\vec x\), \(\vec y\) are
called equivalent if \(\vec y=\alpha\cdot\vec x\), \(\alpha=\pm 1\) in
\(F_7\). Let \(\Omega_{C,24}\) be the set of equivalence classes
with respect to this equivalence. Under the action of the
multiplicative group \(F_7^*\) the set \(\Omega_{C,24}\) splits into eight
equivalence classes.  Elements of each such class are called
\emph{associates}. Let \(\vec x=(x_1,x_2)\), \(\vec y=(y_1,y_2)\). Let us
claim that the vertices \(\vec x\) and \(\vec y\) are adjacent in the
graph \(Gun(7,\alpha)\) if
\(\det\begin{pmatrix}x_1&x_2\\y_1&y_2\end{pmatrix} = \pm \alpha\), for
\(\alpha\in\{1,2,4\}\). Following Gunnells, the reader is welcome to
fix a vertex \(v\in\Omega_{C,24}\) and to observe that the set \(N(v)\)
of its neighbors in \(Gun(7,\alpha)\) induces a heptagon. The \(v\cup N(v)\)
is called a \emph{cap} in \cite{Gun05}; each cap is nothing but a wheel
with center \(v\) and seven spokes to \(N(v)\). For a fixed equivalence
class of three associates the graph \(Gun(7,\alpha)\) can be
decomposed to three induced wheels. (Note that, clearly, the
consideration in \cite{Gun05} is given for an arbitrary prime \(p\),
here we just used \(p=7\).)

\begin{proposition}
\begin{enumerate}
\item \label{org4963464}The graphs \(Gun(7,\alpha)\) for \(\alpha\in\{1,2,4\}\) are isomorphic
to the Klein graph \(\operatorname{Kle}_{24}\).
\item \label{orgf0430a7}The color graph \(Gun(7)\), which is formed by \(Gun(7,\alpha)\),
\(\alpha\in\{1,2,4\}\) and the graph \(8\circ K_3\) defined by
associates, forms a Jordan scheme of rank 5 which is isomorphic
to \(NJ_{24}\).
\item \label{org2a1d21c}Selecting one equivalence class of associates and its
corresponding system of three caps we define the island
\(Isl_{24,3}\) and the continent \(Con_{24,21}\).
\item \label{org3655180}Manipulating with \(NJ_{24}\) in \ref{orgf0430a7} with respect to the island and
continent in \ref{org2a1d21c} we obtain the proper Jordan scheme \(J_{24}\).
\end{enumerate}
\end{proposition}

\begin{proof}
Clearly Model C serves as a methodological link between Models A and
B, providing a very natural way to a computer-free proof. The proof
itself is again postponed to further sections.
\end{proof}

We once more refer to the closing sections for additional discussion
of the structures considered currently.

\section{General formulation of the model "Island and Continent" for rank 6 AS}
\label{sec:org18fbb2d}
\label{org68efaa1}

Relying on the earned experience of constructing proper rank 5
Jordan schemes \(J_{15}\) and \(J_{24}\), let us now approach a more
general model.

Note that due to purely logistical reasons related to the way of the
creation of the current version of the preprint, the notation for the
relations used from now on differs from the previous one. Namely, thin
relations are now denoted by $R_i$, while thick relations by
$\Delta_i$. At the end of the text the initial notation might be
exploited as well again.

Let \(n=3(m+1)\) and \(3|(m-1)\). Assume that 
\[{\frak X}_n=(\Omega, \{R_0, R_1,
  R_2, \Delta_0, \Delta_1, \Delta_2\})\] is a non-commutative rank 6 AS
with three thin relations (that is, of valency 1) \(R_0, R_1, R_2\)
and three thick relations \(\Delta_0, \Delta_1, \Delta_2\) of valency
\(m\), here \(R_0=Id_\Omega\) is the identity relation on the vertex set
\(\Omega=\Omega_{n}\) of cardinality \(n\). Assume that the AS \({\frak X}_n\)
has the following multiplication table:
\begin{subequations}
\label{eq:multiplication}
\begin{align}
R_iR_j &= R_{i+j},\\
R_i \Delta_j &= \Delta_{i+j},\\
\qquad \qquad \Delta_j R_i &= \Delta_{j-i},\\
\Delta_i \Delta_j &= m R_{i-j} + \frac{m-1}{3} (\Delta_0 + \Delta_1 + \Delta_2);
\end{align}
\end{subequations}
here and below all arithmetic on indices is done modulo three.

\begin{proposition}
\begin{enumerate}
\item \label{org510a22a}The basic graphs \((\Gamma_i=\Delta_i,\Omega)\), \(i=0,1,2\) are
antipodal distance regular covers of the complete graph \(K_{m+1}\).
\item \label{org10afc5b}The cyclic group \({\mathbb Z}_3=\left<a\right>\), \(a=(\Delta_0,
     \Delta_1, \Delta_2)\), is a subgroup of the algebraic group \(\aaut({\frak X}_n)\).
\end{enumerate}
\end{proposition}

\begin{proof}
For each \(i\in\{0,1,2\}\) the multiplication rules imply \ref{org510a22a}. The
proof of \ref{org10afc5b} is also a straightforward implication.
\end{proof}

\begin{corollary}
The symmetrization \(NJ_n\) of the AS \({\frak X}_n\) is a non-proper rank 5
Jordan scheme of order \(n\).
\end{corollary}

\begin{proof}
This is an immediate consequence of the basic facts about the Jordan
schemes, which were formulated in Section \ref{org4a0fd08}.
\end{proof}

Recall that the relations \(R_0,R_1,R_2\), regarded as square matrices
of order \(n\), form a multiplicative group of order 3 with respect to
the multiplication of matrices, which is isomorphic to \({\mathbb
  Z}_3\).

\begin{lemma}
\label{org1cc5c32}
The three thick graphs \((\Delta_i, \Omega_n)\) in \({\frak X}_n\),
\(i\in{\mathbb Z_3}\), are combinatorially isomorphic.
\end{lemma}

\begin{proof}
According to Equations \ref{eq:multiplication}, \(R_i^{-1}\Delta_j R_i = R_i^{-1} \Delta_{j-i} =
  R_{-i} \Delta_{j-i} = \Delta_{j-2i}\). Now fix \(j=0\) and consider
\(i\in\{1,2\}\). 
\end{proof}

\begin{remark}
Note that each adjacency matrix of a thin relation is a
permutational matrix, thus its inverse is well-defined and is again
a permutational matrix.
\end{remark}

We will now show how to change the scheme \({\frak X}_n\) by switching the
colors in order to get new color graphs.

Clearly, the relation \(R_0\cup R_1 \cup R_2\) on the set \(\Omega\) is
an equivalence relation \(E\) with \(m+1\) equivalence classes \(F_i\), \(0\le i \le
  m\) of size 3. Each class \(F_i\), regarded as a subset, is called a
\emph{fiber}. As was done in previous examples for \(n=15\) and \(n=24\), let
us distinguish the last fiber \(F_m\), to be called \emph{island}
\(Isl_{n,3}\), and the union of all previous fibers, to be called
\emph{continent} \(Con_{n,3m}\).

Such a distinguishing again allows to split the original five
relations in \(NJ_n\) into ten relations on the set \(\Omega_n\). These
ten relations are:
\begin{itemize}
\item \(Id_{3m}\) on the continent and \(Id_3\) on the island;
\item the spread \(Spr_{3m}\) on the continent, which is formed by the
restriction of \(R_1\cup R_2\) on the set \(Con_{n,3m}\);
\item the spread \(Spr_3\) on the island, which is obtained from \(R_1\cup
    R_2\) in the same manner;
\item three bridges \(Br_i\) from the elements of \(Isl_{n,3}\) to the
vertices of \(Con_{n,3m}\), each bridge is a bipartite graph with
valencies 1 and \(m\);
\item three restrictions of the relations \(\Delta_i\) on the continent,
considered as regular graphs of valency \(m-1\) and order \(3m\), and
denoted by \(T_0,T_1, T_2\). Altogether we get 10 unicolor graphs.
\end{itemize}

It is necessary to clarify that, according to the properties of an
antipodal cover of the the complete graph \(K_{m+1}\), for each graph
\(\Delta_i\), \(i\in \{0,1,2\}\), the following property is
satisfied. As soon as a fiber \(F_m=\{x_0,x_1,x_2\}\) is selected, the
sets of the neighbours of each vertex \(x\in F_m\) form on the
continent a subset of size \(m\), and these three subsets form a
partition of the vertex set of the continent. This implies that for
each graph \(\Delta_i\) the subgraph between the island and the
continent, denoted by \(Br_i\), is a bipartite graph with valencies
\(m\) on the island and valencies 1 on the continent. The union of
these three edge-disjoint bridges provides the complete bipartite
graph \(K_{3,3m}\).

Let us denote by \(Pre_n\) the rank 10 color graph formed by the
described unicolor graphs. Note that in the current fashion the
obtained \emph{pregraph} \(Pre_n\) of order \(n\) does not appear via some
group-theoretical considerations. It is defined in terms of the
initial non-commutative rank 6 AS \({\frak X}_n\), no matter if \({\frak X}_n\) is
schurian or not.

Let us introduce the operations of switching \(Sw_i\), where each switching is
applied to the basic graphs \(\Delta_i\) in a given symmetrization
\(\tilde {\frak X}_n\): we remember that \(\tilde {\frak X}_n\) appears as a merging of
colors in \(Pre_n\). Indeed, \(\Delta_i=T_i\cup Br_i\) for \(i=0,1,2\). 

So the result \(Sw_0(\Delta_0)\) is a triple \(\{\Delta_0,
  \Delta'_{0,1}, \Delta'_{0,2}\}\), where the edge set of
\(\Delta'_{0,1}\) is \(T_1\cup Br_2\), while the edge set of
\(\Delta'_{0,2}\) is \(T_2\cup Br_1\), In other words, the result of
\(Sw_0\) in its action on \({\frak X}_n\) is a rank 5 color graph \(J_{n,0}\) with
the colors \(Id_n\), \(Spr_n=Spr_{n,3}\cup Spr_{n,3m}\), \(\Delta_0\),
\(\Delta'_{0,1}\), \(\Delta'_{0,2}\), which appears as a merging of
\(Pre_n\). Similarly we define two other switchings \(Sw_1\) and \(Sw_2\)
with respect to graphs \(\Delta_1\) and \(\Delta_2\), respectively. Each
such operation will be called \emph{switching of bridges}. It is worthy
to note that although the initial object is a copy of \({\frak X}_n\), in
fact, the switching is obtained in the framework of \(Pre_n\), while
the definition of the pregraph depends on the selection of a fiber,
which plays the role of the island.

Assume that the AS \({\frak X}_n\) is a rank 6 non-commutative AS with
multiplication table  \ref{eq:multiplication}. Let \(Pre_n\) be a rank 10 color graph
obtained from \({\frak X}_n\) by selection of an island. Then, using
switching, we get nine graphs
\(\Delta_0\),
\(\Delta_1\),
\(\Delta_2\),
\(\Delta'_{0,1}\),
\(\Delta'_{0,2}\),
\(\Delta'_{1,0}\),
\(\Delta'_{1,2}\),
\(\Delta'_{2,0}\),
\(\Delta'_{2,1}\).
(Here in the notation \(\Delta'_{i,j}\) the index \(i\) shows which
thick graph \(\Delta_i\) remains unchanged by the switching.)

\begin{proposition}
The graphs \(T_0, T_1, T_2\) in \(Pre_n\) are combinatorially isomorphic.
\end{proposition}
\begin{proof}
It follows from the proof of Lemma \ref{org1cc5c32} that the
semiregular permutation \(h_1\) of order 3, which moves cyclically
each fiber in \(E\), permutes the basic graphs
\(\Delta_0,\Delta_1,\Delta_2\) of \({\frak X}_n\) cyclically. (Here \(h_1\) defines a
directed graph of valency 1 with the adjacency matrix \(R_1\).) Take
into account that \(\Delta_i=Br_i\cup T_i\), \(i\in\{0,1,2\}\). Also,
all bridges \(Br_i\) are clearly combinatorially isomorphic. Note also
that in the action of \(\left<h_1\right>\) the bridges and the three
graphs \(T_i\) form orbits, each of length 3. This, together with
Lemma \ref{org1cc5c32}, implies the result.
\end{proof}

Let \(F_m=\{x_0,x_1,x_2\}\), \(h_2=(x_0,x_1,x_2)\in S_n\); in other
words, \(h_2\) is a permutation on \(\Omega_n\) which acts as a 3-cycle
on the island and fixes the continent pointwise.

\begin{proposition}
The permutation \(h_2\) is a color automorphism of \(Pre_n\).
\end{proposition}
\begin{proof}
All basic graphs on the continent are not touched by \(h_2\). The
permutation \(h_2\) preserves \(Id_3\) and \(Spr_3\) on the
island. Finally it permutes the three bridges cyclically. 
\end{proof}

\begin{proposition}
\label{orgefbb13d}
\begin{enumerate}
\item \label{orgc7eb2f6}The permutations \(h_1,h_2\) considered above generate an
elementary abelian group \(E_9\) of order 9, which acts on
\(\Omega_n\).
\item \label{orgd797d9e}The group \(E_9\) is a subgroup of the color group \(\caut({\frak X}_n)\).
\item \label{org71a6675}The group \(E_9\) acts regularly on the set of graphs consisting of
\(\Delta_i\) and \(\Delta'_{i,j}\), \(i,j\in \{0,1,2\}\), \(i\neq j\).
\end{enumerate}
\end{proposition}
\begin{proof}
Part \ref{orgc7eb2f6} is evident. To prove \ref{orgd797d9e}, reformulate Lemma
\ref{org1cc5c32} as the fact that
\(h_1\in\caut(Pre_n)\). Then \ref{org71a6675} follows from \ref{orgc7eb2f6} and \ref{orgd797d9e} and previous
definitions. 
\end{proof}

\begin{corollary}
All nine graphs \(\Delta_i\), \(\Delta'_{i,j}\) considered above are
combinatorially isomorphic. 
\end{corollary}
\begin{proof}
This clearly follows from Proposition \ref{orgefbb13d}.
\end{proof}

\section{Conditions for the existence of a proper rank 5 Jordan scheme}
\label{sec:org30433d4}
\label{orgd4c6594}
Let \({\frak X}_n\) be a non-commutative rank 6 AS with the multiplication
table  \ref{eq:multiplication} as in the previous section. Let \(Pre_n\) be the rank 10
pregraph obtained from \({\frak X}_n\) after selection of a fiber \(F_m\) as the
island. Let \(J_{n,0}\) be the color graph of rank 5 with basic graphs
\(Id_n\), \(Spr_n\), \(\Delta_0\), \(\Delta'_{0,1}\), \(\Delta'_{0,2}\).

\begin{proposition}
The color graph \(J_{n,0}\) forms a rank 5 Jordan scheme.
\end{proposition}
\begin{proof}
Merging any two thick relations in \(J_{n,0}\) leads to a metric rank
4 AS generated by the remaining thick relation, regarded as an
antipodal DRG. Now the result follows from Lemma \ref{orgb745be0}.
\end{proof}

\begin{corollary}
\label{org3f316d8}
The color graphs \(J_{n,1}\), \(J_{n,2}\) are Jordan schemes isomorphic
to \(J_{n,0}\).
\end{corollary}
\begin{proof}
Use suitable permutations from \(E_9\le\caut(Pre)\).
\end{proof}

We use the following notation: We divide our point set \(\Omega\) into
\(\Omega= \Omega_1\cup \Omega_2\), where \(\Omega_1=Isl_{n,3}\), and
\(\Omega_2=Con_{n,3m}\). For any binary relation \(Q\) on \(\Omega\), and
\(i,j\in\{1,2\}\) we let \(Q_{ij}\) be the restriction \(Q_{ij} = Q\cap
  (\Omega_i\times \Omega_j)\). 
\correction{Applying this to the relations \(R_i\) and
\(\Delta_i\), \(i=0,1,2\), we get that\footnote{\correction{As usual, we do not distinguish between relations and their
adjacency matrices.}}}
\begin{align*}
R_i &= \mtrx{(R_i)_{11}}{0}{0}{(R_i)_{22}},\\
\Delta_i &= \mtrx{0}{(\Delta_i)_{12}}{(\Delta_i)_{21}}{(\Delta_i)_{22}}
\end{align*}

\begin{proposition}
\label{010419c}
For any \(i,j\in{\mathbb Z}_3\) it holds that
\begin{align*}
(\Delta_i)_{12} (\Delta_j)_{21} & =  m (R_{i-j})_{11};\\
(\Delta_i)_{12} (\Delta_j)_{22} & =  \frac{m-1}{3} J_{12};\\
(\Delta_i)_{22} (\Delta_i)_{21} & =   \frac{m-1}{3} J_{21};\\
(\Delta_i)_{21} (\Delta_j)_{12} & =   L\cdot (R_{i-j})_{22};\\
(\Delta_i)_{22} (\Delta_j)_{22} & =   \frac{m-1}{3}(J_{22} - E_{22}) + m (R_{i-j})_{22} - L \cdot (R_{i-j})_{22}
\end{align*}
here \(E=R_0\cup R_1\cup R_2\) and \(L=(\Delta_0)_{21}(\Delta_0)_{12}\). 
\end{proposition}

\begin{proof} It follows from Equations~\eqref{eq:multiplication} that
$\Delta_i \Delta_j = m R_{i-j} + \frac{m-1}{3} (J - E).$ Writing this equality in a block-matrix form yields us
\begin{align*}
(\Delta_i)_{12} (\Delta_j)_{21} & =  m (R_{i-j})_{11};\\
(\Delta_i)_{12} (\Delta_j)_{22} & =  \frac{m-1}{3} J_{12};\\
(\Delta_i)_{22} (\Delta_i)_{21} & =   \frac{m-1}{3} J_{21};\\
(\Delta_i)_{21}(\Delta_j)_{12} + (\Delta_i)_{22} (\Delta_j)_{22} & =  \frac{m-1}{3}(J_{22} - E_{22}) + m (R_{i-j})_{22}.
\end{align*}
This proves the first three rows of our statement.

Since the fifth row is a direct consequence of the fourth one, it
remains only to prove the fourth row.

It follows from \eqref{eq:multiplication}
that
\[(R_k)_{22} (\Delta_0)_{21} = (\Delta_k)_{21} = (\Delta_0)_{21}
(R_{-k})_{11}\]
and
\[(\Delta_0)_{12}(R_k)_{22}  = (\Delta_{-k})_{12} = (R_{-k})_{11}(\Delta_0)_{12} \].
Therefore,  
\begin{align*}
  (\Delta_i)_{21} (\Delta_j)_{12}
  &= (\Delta_0)_{21} (R_{-i})_{11}(\Delta_0)_{12} (R_{-j})_{22} \\
  &= (\Delta_0)_{21}(\Delta_0)_{12}  (R_{i})_{22}(R_{-j})_{22} \\
  &= L \cdot (R_{i-j})_{22}.
\end{align*}
\end{proof}
To describe the structure of the matrix \(L\) we notice that the
relation \((\Delta_0)_{21} = \Delta_0\cap (\Omega_2\times \Omega_1)\)
is a surjective function from \(\Omega_2\) onto
\(\Omega_1\). Therefore
\(L=(\Delta_0)_{21}(\Delta_0)_{12}=(\Delta_0)_{21}(\Delta_0)_{12}^\top\)
is an adjacency matrix of the kernel of
the function \((\Delta_0)_{21}\). More precisely, two points \(x',x\in
\Omega_2\) are equivalent iff \(\Delta_0(x')\cap \Omega_1 =
\Delta_0(x)\cap \Omega_1\). This relation has \(3\) equivalence
classes of size \(m\). Thus the matrix \(L\) is permutation equivalent
to \(I_3\otimes J_m\).

Since \(|\Delta_0(x)\cap E(x)|=1\) for all \(x\in \Omega\), the
relations \(L\) and \(E_{22}\) are orthogonal, i.e. any pair of their
classes are intersecting by one element.

\begin{theorem}
\label{org54afef5}
The rank 5 Jordan scheme \(J_{n,0}\) is a proper Jordan scheme
\end{theorem}

\begin{proof}

Let us compute the product \(\Delta_0\cdot \Delta'_{0,1}\) in block form:
\begin{align*}
\Delta_0\cdot \Delta'_{0,1} &=
\mtrx{O}{(\Delta_0)_{12}}{(\Delta_0)_{21}}{(\Delta_{0})_{22}} 
\mtrx{O}{(\Delta_2)_{12}}{(\Delta_2)_{21}}{(\Delta_{1})_{22}}\\
&= 
\mtrx{(\Delta_0)_{12}(\Delta_2)_{21}}{(\Delta_0)_{12}(\Delta_{1})_{22}}{(\Delta_{0})_{22}(\Delta_2)_{21}}{(\Delta_0)_{21}(\Delta_2)_{12} + (\Delta_{0})_{22}(\Delta_{1})_{22}}
\end{align*}
To compute the latter matrix we use formulae of Proposition \ref{010419c}:
\begin{align*}
(\Delta_0)_{12}(\Delta_2)_{21} &=  m (R_{1})_{11};\\
(\Delta_0)_{12}(\Delta_{1})_{22} &=  \frac{m-1}{3} J_{12};\\
(\Delta_{0})_{22}(\Delta_2)_{21} &=  \frac{m-1}{3} J_{21} 
\end{align*}
and
\[
   (\Delta_0)_{21}(\Delta_2)_{12} + (\Delta_{0})_{22}(\Delta_{1})_{22} = 
   L\cdot (R_{1})_{22} +
   \frac{m-1}{3} (J_{22} - E_{22}) + m (R_{2})_{22} - L\cdot (R_{2})_{22} 
   \]
Thus the coherent closure \(A\) of \(J_{n,0}\) contains the following matrix:
\begin{equation}\label{050419a}
\Delta_0\cdot \Delta'_{0,1} = 
\mtrx{m (R_{1})_{11}}{\frac{m-1}{3} J_{12}}{\frac{m-1}{3} J_{21}}
{L\cdot (R_{1})_{22} +
\frac{m-1}{3} (J_{22} - E_{22}) + m (R_{2})_{22} - L\cdot (R_{2})_{22}}
\end{equation}
Taking into account that \(E_{22}\circ L = I_{22}\) we obtain that 
\[
   (\Delta_0\cdot \Delta'_{0,1})\circ E =
   \mtrx{m (R_{1})_{11}}{O}{O}
   {(R_{1})_{22} + (m-1) (R_{2})_{22}}\in{\mathcal A}.
   \]
By the Schur-Wielandt principle we conclude that the coherent closure \({\mathcal A}\) contains the following matrices
\[
   \mtrx{(R_{1})_{11}}{O}{O}{O},
   \mtrx{O}{O}{O}{(R_{1})_{22}},
   \mtrx{O}{O}{O}{(R_{2})_{22}}.
   \]
Therefore \({\mathcal A}\) is inhomogeneous, as claimed.
\end{proof}

\begin{corollary}
\label{org2853378}
\(J_{n,0}\),    \(J_{n,1}\),    \(J_{n,2}\) are combinatorially
isomorphic proper rank 5 Jordan schemes.
\end{corollary}

\begin{proof}
Consider together Corollary \ref{org3f316d8} and Theorem
\ref{org54afef5}. 
\end{proof}

At this moment we are familiar with two examples of proper Jordan
schemes, namely \(J_{15}\) and \(J_{24}\) (considered up to
isomorphism). The existence of these structures was established in
the first huge part of the current text. On this way computer
aided results, elements of human reasonings, as well as outlines of
some auxiliary claims were extensively exploited.

Now the existence of both of these structures follows immediately from
the claims presented in Section \ref{orgd4c6594}.

We aim to embed \(J_{15}\) and \(J_{24}\) to an infinite class of
proper rank 5 Jordan schemes, as its first members.

\section{An infinite series of non-commutative rank 6 AS's and related proper Jordan schemes}
\label{sec:org70e2dcb}
\label{org3941a9a}

In this section an imprimitive rank 6 action of the projective group
\(PSL(2,q)\) for a prime power \(q\) will be considered. The rank 6 AS
corresponding to this action satisfies all the conditions required
above and thus leads to the existence of proper Jordan schemes of
order \(n=3(q+1)\) for \(q\equiv 1 \mod 3\). 

The exploited ideas go back to R. Mathon, A. Neumaier,
H. Kharaghani, P. Gunnells and other researchers, see the discussion
at the end of the current text.

The style of the presentation below goes back to \cite{KliR12},
\cite{Rei14}, \cite{Rei19a}, although it does not coincide
literally with any of the previous incarnations of the exploited
ideas. 

Typically in this section we provide just simplified outlines of the
necessary reasonings, rather than full proofs. For all details the
reader is referred to the papers cited below.

Thus, let \(p\) be a prime number, \(q=p^\alpha\), \(q\equiv 1\mod 3\),
more precicely, \(q=3l+1\). Let \(F=F_q\) be the finite field of order
\(q\). Denote by \(F^*\) the multiplicative group of \(F\). Clearly, this
is a cyclic group of order \(q-1\). Hence, it contains a cyclic
subgroup \({\mathbb Z}_l\) of order \(l=\frac{q-1}{3}\) acting on \(F\) by
multiplication. 

Now consider a vector space \(P\) of dimension 2 over \(F\). An element
\(\vec p\in P\) is a pair \(\vec p=(p_1,p_2)\), where \(p_1,p_2\in F\),
\(\vec 0=(0,0)\) is the zero vector in \(P\).

We consider the set \(P\setminus\{\vec 0\}\) and define on it a binary
relation \(\sim\): For two vectors \(\vec p', \vec p''\) from \(P\) we
write \(\vec p'\sim \vec p''\) if \(\vec p''=\lambda \vec p'\) for some
\(\lambda\in{\mathbb Z}_l\). 

Clearly, the relation \(\sim\) is an equivalence on the set
\(P\setminus\{\vec 0\}\). we call the equivalence classes of this
relation \emph{one-third quasi-projective points}, or \emph{ot-points} for
short. 

Clearly, the cardinality \(n\) of the set \(\Omega=\Omega_n\) of all
ot-points is equal to \(n=\frac{q^2-1}l = \frac{q-1}l (q+1) =
  3(q+1)\). 

\begin{proposition}
\label{orge37853d}
Let \(A\in SL(2,q)\), \(\vec p', \vec p'' \in P\setminus\{\vec
  0\}\). Then \(\vec p'\sim \vec p''\)  implies that \(A\vec p'\sim A\vec
  p''\).
\end{proposition}
\begin{proof}
If \(\vec p'\sim \vec p''\), then \(\vec p''=\lambda \vec p'\) for some
\(\lambda\in{\mathbb Z}_l\). But then \(A \vec p''= A (\lambda \vec p')
  = \lambda A \vec p'\), and thus \(A\vec p'\sim A\vec p''\).
\end{proof}

\begin{corollary}
The projective group \correction{\(PSL(2,q)\)} acts faithfully on the set \(\Omega_n\).
\end{corollary}
\begin{proof}
Due to Proposition \ref{orge37853d}, the action of \(SL(2,q)\) on the
set \(\Omega_n\) is well-defined. The kernel of this action coincides
with the kernel of the action of \(SL(2,q)\) on the "honest"
projective points, that is, classes of non-zero collinear vectors. 
\end{proof}

\begin{proposition}
\label{org83fe6ed}
\begin{enumerate}
\item \label{orgcc345c4}The action \((PSL(2,q), \Omega_n)\) is imprimitive with the
imprimitivity system of type \((q+1)\circ K_3\).
\item \label{org7c71ebe}The action \((PSL(2,q), \Omega_n)\) has rank 6 with valencies
\(1^3,q^3\).
\end{enumerate}
\end{proposition}
\begin{proof}
Clearly, three collinear ot-points correspond to one projective
point. Thus they form an imprimitivity block of size 3 in the action of
\(PSL(2,q)\). This proves \ref{orgcc345c4}). To prove \ref{org7c71ebe}), distinguish three
cosets of the group \({\mathbb Z}_l\) in \(F^*\). Using these cosets
define three binary relations \(R_0=Id_n\), \(R_1, R_2\) on \(\Omega\),
each of valency 1. The relations \(R_1,R_2\) both are isomorphic to
\((q+1)\circ \vec C_3\), where \(\vec C_3\) is a directed cycle of
length 3.

To define three remaining relations \(\Delta_0, \Delta_1, \Delta_2\),
each of valency \(q\), require that the value of the determinant
\(\det\begin{pmatrix} p'_1 & p'_2 \\ p''_1 & p''_2 \end{pmatrix}\)
 belong to the subgroup \({\mathbb Z}_l\) or one of its
two cosets in \(F^*\). Show that each time for a selected ot-point
\(\vec p'\) there are exactly \(q\) options to get non-equivalent
vectors \(\vec p''\), representing different ot-points, such that
\((\vec p', \vec p'')\in\Delta_i\).
\end{proof}
Denote by \({\frak X}_n=V(PSL(2,q),\Omega_n)\) the rank 6 centralizer algebra
of the considered imprimitive action of \(PSL(2,q)\).

\begin{proposition}
\label{org5d73a22}
Each graph \(\Gamma_i=(\Delta_i,\Omega_n)\), \(i\in{\mathbb Z}_3\),
defines an antipodal distance regular cover of the complete graph
\(K_{q+1}\) with \(q+1\) vertices.
\end{proposition}

\begin{proof}
This is a particular case of the result by R. Mathon, see
\cite{Mat75}, \cite{Mat87}, which was formulated in terms of
imprimitive AS's.
\end{proof}

\begin{proposition}
\label{org403533b}
The introduced AS \({\frak X}_n\) is non-commutative with multiplication table
satisfying the conditions  \ref{eq:multiplication}, presented in the previous sections,
provided that \(q=m\).
\end{proposition}
\begin{proof}
Part of the necessary equations are already satisfied. For the
remaining ones act independently or consult Sections 6,7 in
\cite{Rei19a}. 
\end{proof}

Let us now proceed with the AS \({\frak X}_n\), select an arbitrary fiber in
the spread \(Spr_n=(R_1\cup R_2, \Omega_n)\) and form again a rank 10
pregraph \(Pre_n\), as we did before. Working with it, construct
isomorphic color graphs \(J_{n,0}\), \(J_{n,1}\), \(J_{n,2}\). Let us
select one of them, say \(J_{n,0}\), for further consideration.

\begin{theorem}
\label{orga416f9e}
Let \(q\) be a prime power, \(q\equiv 1 \mod 3\). Consider the pregraph
\(Pre_n\), \(n=3(q+1)\), which is obtained from \({\frak X}_n\). Then its rank 5
merging \(J_{n,0}\) is a proper Jordan scheme of order \(n\).
\end{theorem}

\begin{proof}
Combine Proposition \ref{org403533b} with Theorem \ref{org54afef5}.
\end{proof}
\begin{corollary}
\label{orgf5082e9}
There exist infinitely many proper Jordan schemes.
\end{corollary}
\begin{proof}
Use Dirichlet's Theorem to get infinitely many primes \(p\) with
\(p\equiv 1 \mod 3\). Manipulation with \(\alpha\) provides more choices
for suitable values of \(q=p^\alpha\). 
\end{proof}

The reader is now welcome to identify the obtained proper schemes
\(J_{15}\) and \(J_{24}\) using \(q=4\) and \(q=7\), respectively, with the
previous two structures, whose consideration formed the essential
part of this paper.

Using values of \(q\) equal to 13, 16, 19, 25, 31, we are getting
proper Jordan schemes for orders \(n=42,51,60,78,96\), respectively
(here \(n<100\)).

In fact, the results in \cite{Rei19a} provide a wider spectrum of
attractive options for hunting for proper Jordan schemes, see again
comments in further sections.

\section{Toward more general classes of proper Jordan schemes}
\label{sec:org72f85d2}
\label{org30a0d76}
Here, we discuss an outline of ideas which imply more general
constructions of Jordan schemes. The full justification of the
formulated results will be presented in \cite{KliMR}.

Let \(\Omega=\Omega_n\) be a set of cardinality \(n\). Assume that
\({\cal M}_n=(R,\Omega)\) be a rank \(2l\) AS of order \(n=l(m+1)\) with the set
\(R\) of relations, which is split to \(l\) thin relations \(R_0, \dots,
  R_{l-1}\), and \(l\) thick relations \(\Delta_0, \dots,
  \Delta_{l-1}\). Each thin relation is of valency 1, while each thick
relation is  symmetric of valency \(m\).

We assume that \({\cal M}_n\) has a unique equivalence relation \(E=\bigcup_i
  R_i\) (the thin radical  in the sense of
\cite{Zie05}). We also assume that each of the basic graphs
\((\Delta_i, \Omega)\), \(i\in {\mathbb Z}_l\), is an \(l\)-fold antipodal
distance regular cover of the complete graph \(K_{m+1}\). Here we
require that \(l|(m-1)\).

Summing up all the assumptions, including those which were expressed
evidently above, we require that
\begin{align*}
R_iR_j &= R_{i+j},\\
R_i \Delta_j &= \Delta_{i+j},\\
(18.1) \qquad \qquad \Delta_j R_i &= \Delta_{j-i},\\
\Delta_i \Delta_j &= m R_{i-j} + \correction{\frac{m-1}{\ell}} (\Delta_0 + \dots + \Delta_{l-1}),
\end{align*}
here all arithmetic in the indices is done modulo \(l\). The symmetrization of the
above scheme \({\cal M}_n\), denoted by \(\tilde {\cal M}_n\), yields a non-proper
Jordan scheme \(NJ_n\) of rank \(r=\left\lfloor\frac{3l+2}
  2\right\rfloor\). 

The equivalence relation \(E\) has \(m+1\) fibers of size \(l\). Let \(F_i\)
be an arbitrary fiber. Let us call it the island \(Isl_{n,l}\) and set
\(\Omega_n\setminus F_i=Con_{n,n-l}\) the continent of size \(n-l\).

In the same manner as for the case \(l=3\), the selection of the
island implies the appearance of a pregraph \(Pre_n\). It has two
identity relations on the island and the continent,
\(\left\lceil\frac{l+1}{2}\right\rceil\) antireflexive symmetric
relations on the island; for \(l\) even one of them has valency 1, all
other have valency 2. Also, \(Pre_n\) has \(l\) bridges between the
island and the continent, and \(l+1\) antireflexive relations on the
continent: the remainder of the graph \(E\setminus R_0\) of kind
\(m\circ K_l\), and \(l\) restrictions \(T_i\) of the basic graphs
\(\Delta_i\) on the continent.

At this stage it is possible to define switching of the initial
scheme with respect to the selected graph \(\Delta_0\). Namely, the
thick graph \(\Delta_0\) remains unchanged. For all \(i\in{\mathbb
  Z}_l\) the bridges \(Br_i\) and \(Br_{-i}\) are transposed. In other
words, the thick graph \(\Delta_i\) with edge set \(Br_i\cup T_i\) is
substituted by a thick graph \(\Delta'_{0,i} = Br_{-i}\cup T_i\). Here, all arithmetic is done modulo \(l\). The
color graph obtained by this generalized procedure of switching is
denoted by \(J_{n,0}\). In a similar manner the color graphs \(J_{n,i}\)
are defined, \(i\in{\mathbb Z}_l\). All these graphs have rank
\(r=\left\lfloor\frac{3l+2} 2\right\rfloor\).

Below we formulate a few results. \correction{ Some of them will be proven in
\cite{KliMR}.}

\begin{proposition}
\label{orge71cdb7}
\begin{enumerate}
\item Each of the basic graphs \(\Delta'_{0,i}\) is an antipodal \(l\)-fold
distance regular cover of the complete graph \(K_{m+1}\).
\item Each color graph \(J_{n,i}\), \(i\in{\mathbb Z}_l\), is a Jordan
scheme of rank \(r\).
\end{enumerate}
\end{proposition}

\begin{proposition}
\label{org73ea9cf}
All Jordan schemes \(J_{n,i}\) are proper and pairwise isomorphic.
\end{proposition}

Based on the formulated results, in the next section we will
announce one more infinite class of proper Jordan schemes.

\section{More examples of proper Jordan schemes}
\label{sec:org7318abd}
\label{orgb36d93b}
At the earliest stages of this project, immediately after the
discovery of the proper Jordan scheme \(J_{15}\), a couple of similar
objects of order 40 were constructed.

Recall that the initial way to get \(J_{15}\) was simply to examine
two existing Siamese color graphs of order 15 and to figure out
which of them lead to Jordan schemes. As was discussed, both provide
such schemes, one non-proper, one proper.

The known Siamese color graphs have order \((q+1)\cdot(q^2+1)\), where
\(q\) is a prime power. For \(q=3\) we approach the possibility to work
with Jordan schemes of order \(n=40\).

Exactly this procedure was fulfilled. According to \cite{KliRW09}
we are aware of 475 Siamese designs on 40 points. For each such
design a corresponding Siamese color graph was obtained. Namely, for
each substructure \({\frak S}_i\) of a Siamese design \({\frak S}\) we
associate a color graph \(\Gamma_i\) such that points \(x,y\) form an
edge in \(\Gamma_i\) if and only if they belong to a block in \({\frak
  S}_i\). Let \(\Gamma=\Gamma({\frak S})\) be the resulting color graph. 

After that, Jordan stabilization was performed for each resulting
object. Only three of them survived the stabilization.

The first one \(NJ_{40}\) with the automorphism group \(\aut(NJ_{40})
  \cong A_6\times {\mathbb Z}_2 \cong PSL(2,9)\times {\mathbb Z}_2\) of
order 720 is the symmetrization of what we call the classical
Siamese scheme of order 40. More exactly it is the symmetrization of
the centralizer algebra of an imprimitive action of \(PSL(2,9)\) of
degree 40. The obtained scheme \(NJ_{40}\) is of rank 7. Note that its
spread \(10\circ K_4\) splits into one 1-factor and one 2-factor
consisting of 10 quadrangles. Thus, the resulting valencies are
\(1,1,2,9^4\), so \(NJ_{40}\), indeed,  has rank 7.

We also found proper Jordan schemes \(J_{40,1}\) and \(J_{40,2}\), both
having the automorphism group \((E_9:{\mathbb Z}_4)\times {\mathbb Z}_2\) of
order 72. This group has orbits of lengths 4 and 36 on the points.

It turns out that \(J_{40,1}\) can be obtained from the classical
scheme \(NJ_{40}\) using the switching in the model described as
"island and continent". The second scheme does not appear this way.

We again refer to the end of the text for further discussion of the
situation on 40 points.

The results in the previous Section \ref{org30a0d76} provide another rich
source of proper Jordan schemes. Below in Table 19.1 we
provide a summary of computations leading to confirmation of the
existence of proper Jordan schemes of order \(n\). Additional
information about the automorphism groups of the obtained proper
Jordan schemes provides evident background for a quite intelligent
guess regarding the structure of the automorphism groups of such
structures. This issue will be touched in the forthcoming
publication. 

\begin{table}
\begin{center}
\begin{tabular}{rrrrl}
\hline
q & l & n & r & Aut(J\(_{\text{n}}\))\\
\hline
8 & 7 & 63 & 11 & AGL(1,8)\\
11 & 5 & 60 & 8 & AGL(1,11)\\
16 & 5 & 85 & 8 & AGL(1,17)\\
19 & 9 & 180 & 14 & AGL(1,19)\\
23 & 11 & 264 & 17 & AGL(1,23)\\
27 & 13 & 364 & 20 & AGL(1,27)\\
29 & 7 & 210 & 11 & AGL(1,29)\\
31 & 5 & 160 & 8 & AGL(1,31)\\
37 & 9 & 342 & 14 & AGL(1,37)\\
41 & 5 & 210 & 8 & AGL(1,41)\\
43 & 7 & 308 & 11 & AGL(1,43)\\
61 & 5 & 310 & 8 & AGL(1,61)\\
64 & 7 & 455 & 11 & E\(_{\text{64}}\):Z\(_{\text{9}}\):Z\(_{\text{2}}\):Z\(_{\text{2}}\):Z\(_{\text{7}}\)\\
67 & 11 & 748 & 17 & AGL(1,67)\\
\hline
\end{tabular}
\end{center}
\caption{Proper Jordan schemes coming from the action of \(PGL(2,q)\). \label{table:proper-schemes}}
\end{table}

\section{Toward a prolific construction of Jordan schemes of rank 5 and order \(\binom{3^d+1}{2}\), \(d\in{\mathbb N}\)}
\label{sec:orgb2227ef}
\label{org682bc3a} \label{org3f2a5c7}
Below we outline a prolific construction of rank 5 Jordan schemes
of order \({3^d+1}\choose 2\), \(d\ge 1\). All introduced Jordan schemes
are algebraically isomorphic and are defined on the same set
\(\Omega=\Omega_n\) of cardinality \(n=\binom{3^d+1}{2}\),
\(d\in{\mathbb N}\). Each scheme is formed by four antireflexive basic
relations \(S, R_0, R_1, R_2\).

The graph \((\Omega,S)\), called again the \emph{spread}, is isomorphic to
\(\frac{3^d+1}{2}\circ K_{3^d}\). The basic relations \(R_i\), \(i\in
  \{0,1,2\}\), define SRG's with the same parameter set
\(n=\binom{3^d+1}{2}\), \(k=3^{d-1} \cdot \frac{3^d-1}2\), \(\lambda =
  \mu = 3^{d-1}\cdot \frac{3^{d-1}-1}{2}\). The idea of the
construction of such a family of SRG's goes back to W.D.Wallis
\cite{Wal71}, \emph{design graphs} in his terms. The construction used
here is a particular case of Construction 1, considered by
D.G. Fon-Der-Flaass in \cite{FDF02}. The graphs, which were defined
there are nowadays commonly called WFDF SRG's, or simply \emph{WFDF
graphs}.

In fact there are some degrees of freedom in the building of all
WFDF graphs. Manipulation by these degrees provides a possibility to
prove that the number of non-isomorphic SRG's with the given parameter set
grows hyperexponentially with the number of vertices. 

Thus for each \(d\in{\mathbb N}\) the starting structure is the affine
design \(AD(3^d)\), which is defined on the vector space \({\mathbb
  Z}_3^d\) of dimension \(d\) over the field \({\mathbb Z}_3\). The elements
of \({\mathbb Z}_3^d\) are the points of the design. The blocks are
all \emph{hyperplanes} in \({\mathbb Z}_3^d\). Clearly, each hyperplane
consists of solutions of a suitable linear equation with \(d\)
variables over \({\mathbb Z}_3\); there are exactly
\(r=\frac{3^d-1}{2}\) different hyperplanes containing the zero
vector.

The set \(\Omega={\mathbb Z}_3^d \times [0,r]\) of cardinality
\(n=\binom{3^d+1}{2}\) consists of \(r+1\) copies \(\Omega_i={\mathbb
  Z}_3^d\times\{i\}\) of the space \({\mathbb Z}_3^d\), which are labeled
by the elements \(i\in[0,r]\). For each hyperplane \(H_j\) we select one
of two possible linear equations which define it, regarding the
equation as an epimorphism \(\pi_i:{\mathbb Z}_3^d\to {\mathbb Z}_3\).

The defined sets \(\Omega_i\) serve as the fibers of the relation
\(S=(r+1)\circ K_{3^d}\). The three other relations are defined with
the aid of the Construction 1 from \cite{FDF02}, adjusted for the
current case. Using this approach, first we consider an SRG
\(\Gamma=\Gamma_0=(\Omega, R_0)\) with the parameter set given
above. Applying the classical result from \cite{HaeT96} about the
spread formed by Hoffman cocliques of the SRG we claim that
\((\Omega, \{\id, S, R_0, \overline{S\cup R_0}\})\) is a rank 4 AS.

The trickiest and in fact innovative part of our approach is related
to the necessity to split the relation \(\overline{S\cup R_0}\) to a
disjoint union \(R_1\cup R_2\) of two more relations, each of which
defines an SRG on the set \(\Omega\) with the same parameter set as
for \(R_0\). Moreover, the same relation \(S\) appears to define a spread
which consists of the Hoffman cocliques for both SRG's
\(\Gamma_1=(\Omega, R_1)\)
and
\(\Gamma_2=(\Omega, R_2)\). For this purpose the original techniques
by D. Fon-Der-Flaas is exploited in a more tricky manner.

Finally with the aid of Lemma \ref{orgb745be0} we claim that the color
graph  \[(\Omega, \{\id, S, R_0, R_1, R_2,\})\] is a rank 5 Jordan
scheme.

In principle, some of the defined rank 5 color graphs could form a non-proper
Jordan scheme. Good news is that, due to the prolific nature of the
presented construction it is possible to prove that at least  part of the
results lead to proper Jordan schemes.

The remaining details of the construction, proofs of its properties,
as well as clarification of the asymptotical features of the
suggested approach, will appear in the paper \cite{KliMR} in
preparation. 

Below we restrict ourselves to the  consideration of the degenerate case
\(d=1\) and the first non-trivial case \(d=2\), leading to schemes of
order 45. 

\section{Degenerate case of order 6: Shah's scheme \(NJ_6\)}
\label{sec:org39ab405}
\label{orgcbc9bab}
The smallest case of the situation discussed in the previous section
corresponds to \(d=1\). Clearly we get here a rank 5 Jordan scheme of
order \(\binom{4}{2} = 6\). Let us discuss it from a few points of
view.

\begin{model}
\label{org5a81b06}
Applying the general prolific construction to \(d=1\) we get a rank 5
scheme with valencies \(1,1,1,1,2\). Here, relations
\(R_0\) - \(R_2\) define SRG's with parameters \((6,1,0,0)\). Clearly, such
an SRG is a disconnected graph \(3\circ K_2\). The relation \(S\)
defines a spread \(2\circ K_3\). It is obvious that the corresponding
color graph is defined uniquely up to isomorphism. It is depicted in
Figure \ref{figure:21.1}. Let us denote it by \(NJ_6\). The notation
will be justified below; at this point the diagram introduces just a
nice small structure, nothing else is assumed.
\end{model}

\begin{figure}
\begin{center}
\input{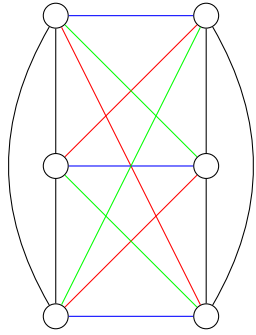}
\end{center}
\caption{Shah's scheme \label{figure:21.1}}
\end{figure}

\begin{model}
Let us consider the symmetric group \(S_3\), acting on the set
\(\{0,1,2\}\). Assume that the elements of \(S_3\) are as follows:
\(h_0=\id_3\), \(h_1=(0,1)\), \(h_2=(0,2)\), \(h_3=(1,2)\), \(h_4=(0,1,2)\),
\(h_5=(0,2,1)\). Let us consider the regular action \((\tilde S_3,
  S_3)\) of \(S_3\) on the set of all elements of \(S_3\), regarded as new
points. Then the centralizer algebra \(W=V(S_3,S_3)\) defines a thin
non-commutative AS of rank and order 6. In fact this is the smallest
non-commutative AS. Its symmetrisation provides a non-proper Jordan
scheme \(NJ_6\). The two directed relations corresponding to elements
of order 3 in \(S_3\) are glued in the symmetrization. Thus \(NJ_6\) has
rank 5.
\end{model}

\begin{model}
We consider the group \(PSL(2,2)=SL(2,2)=GL(2,2)\) of order 6, which
is isomorphic to \(S_3\). The points of \(PG(2,2)\) bijectively
correspond to the nonzero vectors from \({\mathbb Z}_2^2\), namely
\(\vec v_1=(0,1)\),
\(\vec v_2=(1,0)\),
\(\vec v_3=(1,1)\).
There are also three hyperplanes, which are defined by the equations 
\(x_1=1\),
\(x_2=1\),
\(x_1+x_2=1\). Note that here we consider only affine hyperplanes
which do not contain the zero vector. Finally, let us consider
\emph{ordered hyperplanes}; e.g., to the equation \(x_1=1\) we associate
the pairs \((\vec v_2,\vec v_3)\), \((\vec v_3,\vec v_2)\). Let us call
such ordered pairs degenerate quasiprojective points. Then the group
\(PSL(2,2)\) acts regularly on the set of such objects, and the
symmetrization of the appearing centralizer algebra is
\(NJ_6\). Though this is an artificial analogy with genuine
quasiprojective points, it has some sense provided we wish to appeal
to the imprimitive action of the group \(PSL(2,2,)\). (The essence of
the analogy is that, while quasiprojective points may be interpreted
as subsets of vectors, here we deal with ordered sets of vectors.)
\end{model}

\begin{model}
\label{org6680041}
We follow according to \cite{Sha59}. It is remarkable that this
paper was published in the same volume 30 of the Annals of
Mathematical Statistics, where the classical text \cite{BosM59} by
Bose and Mesner
appeared. 

While \cite{BosM59} is devoted to establishing background concepts
of the now classical theory of (commutative) AS's, the author
B.V. Shah of \cite{Sha59} goes further. Namely, he immediately
considers a natural generalization of AS's to what is now called
Jordan schemes. The wording itself is, of course, not used; the
motivation goes back to B. Harshbarger, see \cite{Har49}. The
reader is easily recognizing in Example 4.1 of \cite{Sha59}
a rank 5 color graph  (in terms of the
current paper), such that the corresponding
vector space of symmetric matrices of order 6 is closed with respect
to the introduced operation \(\frac 12 (AC + CA)\) for two matrices
\(A\) and \(C\). Clearly, this is exactly the Jordan product of
matrices. Note that the basic graphs, depicted in Figure
\ref{figure:21.1}, up to the used labelling of the vertices, are literally
coinciding with the basic "associates of treatments", which are
presented in Table 4.2, accompanying the Example 4.1 of \cite{Sha59}.
\end{model}

\begin{proposition}
Models \ref{org5a81b06}-\ref{org6680041}, described above, represent the unique non-proper rank 5
Jordan scheme \(NJ_6\) of order 6.
\end{proposition}

The proof is a straightforward, almost trivial exercise, which
hopefully will be enjoyed by the reader.\noproof

\begin{remark}
We can apply to any of the above models of \(NJ_6\) the approach
"island and continent". For this purpose in the spread \(S=2\circ
  K_3\), any of two triangles should be claimed as the island, while
the second one as the continent. Clearly, any of the three possible
switchings leaves the scheme \(NJ_6\) on the place (transposing two of
its thin relations). Thus switchings provide suitable elements of
\(\caut(NJ_6)/\aut(NJ_6) \cong S_3\). 
\end{remark}

\section{Computer aided enumeration of some Jordan schemes of order 45}
\label{sec:org47dd719}
\label{orgf5cb3c0}
Here we approach the first non-degenerate case of the prolific
construction outlined in Section \ref{org682bc3a}. Recall that for
\(d=2\) we obtain the valencies \(1,8,12, 12, 12\). Each of the
primitive basic graphs \(\Gamma_i\), \(i\in\{0,1,2\}\), is a
pseudogeometric SRG with the parameters \((45,12,3,3)\). A classical
geometric graph with these parameters, that is the point graph of a
\(GQ(4,2)\), has a rank 3 group \(P\Gamma U(4,2)\) of order \(2^7\cdot
  3^4\cdot 5 = 51840\). In fact, this is a remarkable group, which has
five different primitive actions, all of rank 3, and of degrees 27,
36, 40, 40, and 45.

The structure \(GQ(4,2)\) is unique up to isomorphism. The easiest way
to observe this is to consider the point graph of the dual structure
\(GQ(2,4)\), that is an SRG with the parameters \((27,10,1,5)\). The
uniqueness of this graph was proved by J.J. Seidel \cite{Sei68}.

The question of the enumeration of necessary SRG's of order 45 is
much more sophisticated, see below.

First, let us briefly touch on the prolific construction for rank 5
Jordan schemes of order 45. The task of a more accurate outline for
the current case \(d=2\) seems to be simpler, compared to the general
case. 

We start with the vector space \({\mathbb Z}_3^2\) in dimension \(d=2\)
over the finite field \({\mathbb Z}_3\). It contains
\(r=\frac{3^2-1}{2}=4\) different (vector) hyperplanes. Thus the point
set \(\Omega\) of the coming construction consists of \(r+1=5\) copies
\(\Omega_i\) of the vector space \({\mathbb Z}_3^2\), \(i\in[0,4]\). If
\(\vec v \in {\mathbb Z}_3^2\), we denote by \(\vec v_i = (\vec v, i)\)
its copy in the set \(\Omega_i\).

The spread \(S=5\circ K_9\) here consists of
a disjoint union of 5 copies of a complete graph with nine vertices
in \(\Omega_i\).

Following the WFDF construction we identify the set \([0,4]\) with
\({\mathbb Z}_5\) and define for \(x,y\in {\mathbb Z}_5\) the binary
operation \(x\diamond y = x-y\), here subtraction is done in \({\mathbb
  Z}_5\). 

We also choose for each vector hyperplane \(H_i\) in \({\mathbb Z}_3^2\)
a linear epimorphism \(\pi_i:{\mathbb Z}_3^2 \to {\mathbb Z}_3\), such
that its kernel coincides with the corresponding hyperplane
\(H_i\). Thus let us choose for \((x,y)\in{\mathbb Z}_3^2\):
\begin{align*}
\pi_1(x,y)&=x\\ \pi_2(x,y)&=y\\ \pi_3(x,y)&=x+y\\ \pi_4(x,y)&=x-y.
\end{align*}
Now it is necessary to choose bijections \(\sigma_{ij}\in
  Sym({\mathbb Z}_3)\), \(i\neq j\in{\mathbb Z}_5\). They have to
satisfy the condition \(\sigma_{ji}=\sigma_{ij}^{-1}\). Clearly there
are \((3!)^{\binom 52}=6^{10}\) such choices.

To define the relation \(R_0\), we choose \(\sigma_{ij}=\id_{{\mathbb
  Z}_3}\) for all \(i,j\in{\mathbb Z}_5\). Then the description of the
relation is as follows: 
\[
  R_0 = \left\{
  ((\vec u, i), (\vec v, j)) \mid 
  i\neq j \wedge \pi_{i-j}(\vec u) = \pi_{j-i}(\vec v)
  \right\}.
  \]
The relation \(R_0\) is symmetric by definition.

On the next stage we have to define relations \(R_1\) and \(R_2\), both
disjoint with \(R_0\) and with each other. Analysis of the general
construction for arbitrary dimension \(d\) shows that this requirement
implies that each selected permutation \(\sigma_{ij}\) from
\(Sym({\mathbb Z}_3)\) should be a 3-cycle. 

Let us pick up the simplest choice \(\theta_{ij} = (0,1,2)\) for all
\(0\le i < j \le 4\), that is \(\theta_{ij}(x) = x+1\), \(x\in{\mathbb
  Z}_3\). 

Then we are using the general definition of the relations \(R_1\) and
\(R_2\) and are relying on the compact notation \(\vec  u_i = (\vec u,
  i)\), \(\vec v_j = (\vec v,j)\).

The resulting descriptions are as follows:
\begin{align*}
R_1 &= \left\{
((\vec u, i), (\vec v, j)) \mid 
i\neq j \wedge \pi_{i-j}(\vec u) = \pi_{j-i}(\vec v) + \varepsilon_{ij}
\right\}\\
R_2 &= \left\{
((\vec u, i), (\vec v, j)) \mid 
i\neq j \wedge \pi_{i-j}(\vec u) = \pi_{j-i}(\vec v)-\varepsilon_{ij}
\right\}.
\end{align*}

Here, \(\varepsilon_{ij}=1\) if \(i<j\), and \(\varepsilon_{ij}=-1\) if
\(i>j\). 

\begin{remark}
The presented construction provides a description of one concrete
example of a rank 5 Jordan scheme of order 45 which includes three
WFDF graphs. In order to analyze all possibilities for the
appearance of the prolific case one has to manipulate by all options
to select necessary functions on each step of the process. 
\end{remark}

\begin{proposition}
The described construction provides a concrete example of a rank 5
Jordan scheme of order 45.
\end{proposition}
\begin{proof}
On the first stage one has to check that basic graphs
\(\Gamma_i=(\Omega, R_i)\), \(i\in\{0,1,2\}\) are indeed SRGs with the
parameters \((45,12,3,3)\). For this purpose one may simply trust to
\cite{FDF02} or fulfill from scratch all necessary inspections.

It is clear that \(S\) appears as a common spread of the basic SRGs
\(\Gamma_i\). It satisfies conditions, formulated in
\cite{HaeT96}. Thus each color graph \((\Omega, \{S, R_i,
  \overline{R_i\cup S}\})\) is a rank 4 AS.

Finally we apply again the sufficient condition for the existence of
a rank 5 Jordan scheme which was exploited before a few times.
\end{proof}

\begin{corollary}
Exploiting all degrees of freedom for the selection of necessary
functions on each step we approach a big amount of rank 5 Jordan
schemes of order 45.
\end{corollary}

We can regard the schemes which appear due to the last Corollary as the
results of the prolific construction for \(d=2\). The evidence for
such a formulation follows for example from the prolific property of
the WFDF graphs. See also the discussion at the end of the text.

As soon as this prolific construction is implemented, in principle,
a reasonable pattern of activity would be to try to construct all
appearing structures (or pick them pseudo-randomly) and to classify
the obtained Jordan schemes up to isomorphism. Of course, a crucial
question would also be to figure out which structures are proper
Jordan schemes and which are not.

Nevertheless, due to diverse reasons of mainly logistical nature,
another more general approach was selected for the current case \(n=45\).

In the case \(d=2\), leading to SRGs of order 45, one may naturally
think of the structure of subgraphs of a considered SRG induced by
selected fibers of size 9. Each one fiber induces an empty graph of
order 9. It is easy to figure out that any two fibers of an
arbitrary WFDF graph induce a subgraph of order 18 and valency 3
which consists of a disjoint union of three complete bipartite
graphs \(K_{3,3}\). Let us define this property of a pair
\((\Gamma,S)\), where \(\Gamma\) is an SRG of valency 12 and order 45,
S is a spread in \(\overline \Gamma\) (a Hoffman coloring in other
words) as property \((\alpha)\). 

At this stage we consider the full catalogue of SRGs of order 45 and
valency 12, available on the home page \cite{Spe} of Ted
Spence. This part of the catalog corresponds to the results
reported in \cite{CooDS06}. For each of the 78 SRGs \(\Gamma\) with
the parameters \((45, 12, 3, 3)\) we enumerate all spreads (up to
isomorphism) in \(\overline \Gamma\). We obtain 768 such pairs. Good
news is that this number coincides with the one reported in
J. Degraer's PhD thesis \cite{Deg07}.

It turns out that only three pairs in this list satisfy condition
\((\alpha)\). Namely, we get the following pairs. (The numbering of
SRGs is according to Spence's catalogue.)

\begin{itemize}
\item SRG \# 12 has a unique spread with property \((\alpha)\). Thus this
spread is invariant with respect to the group
\(\aut(SRG_{12})\cong {\mathbb Z}_3\times S_3\) of order 18.
\item SRG \# 15 with the group \({\mathbb Z}_{10}\) and SRG \# 61 with the
group \(E_9\) also admit such good spreads.
\end{itemize}

At this stage we concentrated on the enumeration of all rank 5
Jordan schemes of order 45 which contain the SRG \# 12 with a good
spread.

For the solution of such kinds of problems a special program was
written.
\begin{itemize}
\item The algorithm starts with a prescribed pair \((\Gamma, S)\) and the
group \(H=\aut(\Gamma, S)\).
\item Let \(G=\aut(S)\), clearly \(H\le G\).
\item Find all permutations \(g\in G\) which map \(\Gamma\) to a spanning
subgraph \(\Gamma'\) of \(\overline{\Gamma\cup S}\). The solutions
are split to double cosets of the group \(H\) in the group
\(G\). Here, \(\Gamma'\) is assumed to be isomorphic to \(\Gamma\).
\item For each obtained double coset check that
\(\Gamma''=\overline{\Gamma\cup S}\setminus \Gamma'\) is also an
SRG with the requested parameters:
in fact we just take \(\Gamma,\Gamma'\) and stabilize.
\item Then the color graph with vertex set \(\Omega\) and unicolor
antireflexive graphs \(S,\Gamma, \Gamma', \Gamma''\) provides an
example of a Jordan scheme.
\item All Jordan schemes obtained this way should be sorted up to
isomorphism and checked if they provide a proper Jordan scheme.
\end{itemize}

At the first stage of the experimentation in the role of both
\(\Gamma\) and \(\Gamma'\) the graph SRG \# 12 was used. In all found
solutions the third SRG was also isomorphic to SRG \# 12.

The preliminary result is that there exist at least 340
non-isomorphic rank 5 Jordan schemes of order 45. Striking news is
that all detected structures provide example of proper Jordan
schemes. 

\begin{remark}
A similar algorithm was tested in its rudimentary form for the
enumeration of all Jordan schemes of order \(n\le 14\), see for more
details the next section.
\end{remark}

\begin{remark}
The obtained results provide some background for the conjecture
that any rank 5 Jordan scheme of order 45 and valencies
\(1,1,12,12,12\) with three pseudo geometric basic graphs is
proper. The attempts to justify such a conjecture are in the
proceeding. Clearly such task belongs to the theory of ASs.
\end{remark}

\begin{remark}
In fact the developed program was first tested for the case when
\(\Gamma\) is the geometric graph with a large rank 3 group of
order 51840. Not a single rank 5 Jordan scheme was found which
contains at least one copy of this SRG \# 5 according to the
catalogue. Thus this portion of computations adds a bit of evidence
to consider condition \((\alpha)\) for SRGs as close to be a
necessary one in order to get a proper rank 5 Jordan scheme of
order 45.
\end{remark}

\section{Computer aided classification of small Jordan schemes}
\label{sec:orge37516c}
\label{org0919997} \label{org5dcbc7f}
This section provides a brief report about computer activities
toward the enumeration of Jordan schemes of order \(n\le 14\). The
results belong to the author S.R., they were discussed with the two
other authors, subject to possible algorithmic and theoretical
improvements. 

The goal was to enumerate all proper Jordan schemes of order \(n\),
provided they exists.

Initial simple assumptions:
\begin{itemize}
\item We know on a theoretical level that the smallest possible rank of
such a scheme is equal to 5;
\item All basic graphs are regular and symmetric;
\item For odd orders \(n\) all non-reflexive valencies \(n_i\) are even.
\end{itemize}

\subsection{Preliminary algorithmic experiments}
\label{sec:org1767b94}

Prior to the current project the third author was writing his own
program for the orderly generation of all AS's of a given order
\(n\). (In fact this project was part of a more ambitious one related
to the enumeration of all small CC's, cf. \cite{ZA18}.) The
algorithm so far works for orders \(n\le 15\). The obtained results,
clearly, agree with those by predecessors in this kind of activity,
see \cite{HanM}, where all available results in the enumeration of
small AS's are presented and discussed. 

The accumulated experience was, in principle, enough to quickly
modify the existing algorithm to the program of the search of all
Jordan schemes.

In fact, there exist two modifications of the program, a naive one and
a more advanced. A naive program, based on a very simple brute force
search, worked for \(n\le 11\).  A more advanced program was also
tested on small cases; the results were the same as for the naive
program.

Below we briefly discuss the obtained results. Each time the
possible valencies are mentioned together with the actual
results. Some simple arguments, based on the theoretical
reasonings, related to reasonably evident conditions for the
existence of a proper Jordan scheme, may remain hidden at the
course of consideration.

\subsection{Case \(n=8\)}
\label{sec:org6112228}

The following multisets of valencies and related solutions appear:  
\begin{itemize}
\item \(1,1^7\), Rank 8, this is \(V(E_8)\), theoretical;
\item \(1,1^5,2\), rank 7, no solutions;
\item \(1,1^3,2^2\), rank 6, one solution: \(V(D_4\times{\mathbb Z}_2)\);
\item \(1,1^4,3\), rank 5, no solution;
\item \(1,1,2^3\), rank 5, two Schurian solutions;
\item \(1,1^2,2,3\), rank 5, no solutions;
\item \(1,1^3,4\), rank 5, one Schurian solution.
\end{itemize}

\subsection{Case \(n=9\)}
\label{sec:org6a05dc6}

The only possibility is
\begin{itemize}
\item \(1,2^4\), rank 5, two Schurian solutions.
\end{itemize}

\subsection{Case \(n=10\)}
\label{sec:orgd55a29f}

There are 18 possible valency multisets,
only four of them satisfy an extra condition: less than two thin
relations. 
\begin{itemize}
\item \(1,1,2^4\), rank 6, one solution \(V(D_{10})\);
\item \(1,1,2,3^2\), rank 5, no solution;
\item \(1,2^3,3\), rank 5, no solutions;
\item \(1,1,2^2, 4\), rank 5, one solution, symmetrization of $V(AGL(1,5))$ (acting on pairs).
\end{itemize}

Assume that there are two 1-factors. Together they form either a
10-cycle or a 4-cycle and a 6-cycle. The 10-cycle generates a rank
8 Jordan scheme which is the symmetrization of the full S-ring over
\(D_5\), so it is not a proper Jordan scheme. The union of two cycles
of different sizes does not survive Jordan stabilization.

\subsection{Cases \(n=11,13\)}
\label{sec:org22e5073}

Because the rank should be at least 5 there exists a basic relation
of valency 2. Such a relation survives stabilization if all
connected components have the same size. Taking into account that
here \(n\) is prime this means that we have an \(n\)-cycle
\(C_n\). Clearly its Jordan stabilization is a symmetric AS, which is
a non-proper Jordan scheme, and a maximal symmetric scheme. Hence,
no proper Jordan schemes of orders 11 and 13 exist.

\subsection{Cases \(n=12,14\)}
\label{sec:org7f530a8}
For the remaining cases \(n=12,14\) more advanced search techniques
were exploited. However so far no conclusive results were obtained.
A reasonably clear sketch of related activities is expected in the
next update of the current preprint or elsewhere.

\section{Miscellanea}
\label{sec:orgd2aa00f}
\label{org495439e}
A number of issues which do not fit into the mainstream line of the
presentation but may add additional information of a potential
interest for the reader are discussed in this section.

\subsection{}
\label{sec:orgdba5c9b}
One can distinguish a few kinds of color graphs, which are
considered in the current paper. The most general kind may be
associated with the term \emph{rainbow}, after D.G. Higman
\cite{Hig90}. The two most significant kinds of color graphs in our
paper are coherent configurations (CCs) and what we call
(coherent) Jordan configurations.

The notion of a pregraph, introduced by us, is a particular
concrete kind of rainbow for which the required properties are
fulfilled. 

We also considered regular colorings, or regular color graphs,
where each unicolor graph is regular.

\subsection{}
\label{sec:org5d91725}
Let \({\frak X} = (\Omega, {\cal R})\), where \({\cal R}\) is a
partition of \(\Omega^2\), be a color graph with vertex set
\(\Omega\). We associate with it two permutation groups \(\aut({\frak
   X})\) and \(\caut({\frak X})\), both acting on \(\Omega\). The first
group preserves each color, while the second one may interchange
them. Clearly, \(\aut({\frak X}) \normal \caut({\frak X})\). In case
\({\frak X}\) is a CC, the problem of computer-aided calculation of
\(\caut({\frak X})\) was considered for a reasonably long while. Thus
we were able to successfully exploit the accumulated
experience. For the case of Jordan configurations the problem of
computing of related groups appeared quite suddenly. It is clear
that in future it will require more serious and systematical
efforts. 

Note that for the calculation of \(\caut({\cal J})\), where \({\cal
   J}\) is a Jordan scheme, knowledge of \(\aaut(J)\) is very helpful,
like in the classical case of CCs, see again \cite{KliPRWZA10}. 

We were also facing the necessity to compute groups of an arbitrary
color graph, in particular, the color group. Here sometimes also
appear unexpected difficulties. One of the reasons is related to
the fact that the absence of structure constants does not provide
in advance an efficient possibility of the embedding of the
quotient group of the desired group to the corresponding algebraic
group.

\subsection{}
\label{sec:orgcb6324a}
As was announced at the beginning, one of the essential features of
the developed text is balancing between three kinds of
arguments. Visual proofs were provided in the main part of the
paper in a larger proportion, probably even too large.

Patterns of theoretical reasonings also appear frequently. A reader
who is interested to learn more details is welcome to await the text
\cite{KliMR} to appear soon.

Ongoing interactions between computational activities and simple
(on the edge of trivial) theoretical arguments, on purpose, in some
cases were hidden from the reader, like it was in the previous
section.

Below are a couple of illustrations of such collisions.

\begin{examp}
Both WL and Jordan stabilization destroy the truncated tetrahedron
TT of order 12. Indeed, TT contains two kinds of edges: those which
appear inside of triangles, and those which do not. Therefore, TT
cannot appear as a basic graph of a CC or a Jordan scheme, in spite
of the fact that it is vertex transitive.
\end{examp}

\begin{examp}
The Hoffman graph \(H_{16}\) of order 16 and valency 4 was described
in \cite{Hof63} It is cospectral and switching equivalent to the
4-dimensional cube \(Q_4\). The graph is walk regular. This is why
there was sense to investigate it a bit with the aid of a
computer. 

It turns out that \(\aut(H_{16})\cong S_4\times S_2\) is a group of
order 48 with three orbits of lengths 2, 6 and 8 on the vertex
set. \(WL(H_{16})\) is Schurian; the Jordan closure coincides with
\(SWL(H_{16})\). 

We think there is a potential to pay more attention to this PWRG;
the obtained results might be published elsewhere. 
\end{examp}

\subsection{}
\label{sec:org5e8e9c2}
As was mentioned in Section \ref{orgf276f8a}, we were simultaneously
using \gap and COCO. This was dictating some features of the
presentation of technical data. For example, the reference to file
15aa.map aims to show that the computations were done twice, and step
aa was in a certain sence a clarification of step a.

On the other hand we were sometimes hesitating to show literally
COCO file with the main protocol of results. One of the reasons is
that the name for 2-orbits seems to be  obsolete, in particular for
the case of intransitive induced actions. A few outputs were
provided in traditional \gap format, which should be familiar
to a  \gap user. Also, sometimes we were facing the necessity
to switch to the COCO style, where for a permutation group
\((H,\Omega)\) the smallest element of \(\Omega\) is 0, while it is 1 in
 \gap. 

The latter issue was relevant in particular when we were using the classical
catalog \cite{Sim70}, in fact, one of the first catalogs of small
permutation groups. Here we also shifted certain sets of generators
to adjust it to COCO notation. 

Another small issue is the notation \({\mathbb Z}_n\) for the cyclic
group of order \(n\), while by \(C_n\) we understand an undirected
connected cycle of order \(n\).
\subsection{}
\label{sec:orgb2425ca}

The alternating group \(A_5\) of order 60 is the initial hero of the
development of the current paper. Its imprimitive transitive action
of degree 15 and rank 6 allows to observe existence of both
\(NJ_{15}\) and \(J_{15}\). For this purpose we were relying on the
property of the number six, as it is treated in
\cite{CamvL91}. Other transitive and intransitive actions of \(A_5\)
are also of an essential interest in AGT. The paper \cite{KliZA13}
investigates all S-rings over \(A_5\); among which there are many
quite attractive non-Schurian objects. Part of them still require
deeper understanding. 

In the paper \cite{GyuKZA19} the central object of interest is an
intransitive action of \(A_5\) of degree 26 with two orbits of
lengths 6 and 20. In these terms a quite nice SRG of order 26 and
valency 10 appears, which was selected as the logo of the recent
conference "Symmetry versus regularity. First 50 years since
Weisfeiler-Leman stabilization", see \cite{WeiL18}.

\subsection{}
\label{sec:org4b8c9bb}

Our next proper example of a Jordan scheme appears on 24
vertices. Here the main structure is the Klein graph of valency 7,
aka distance regular cover of the complete graph \(K_8\). The
uniqueness of this DRG was proved in the PhD thesis \cite{Jur95} of
Alexandar Jurišić. We were already discussing the restriction of
the graph \(\operatorname{Kle}_{24}\) on the continent of size 21,
due to \cite{Jur95}. 

The Klein graph \(\operatorname{Kle}_{24}\) is of an independent
interest in modern AGT and map theory, not only due to its amazing
links with the classical papers by F. Klein, mentioned in Section
\ref{org5e0938c}. 

It was considered in the pioneering paper \cite{HaeS95}, where it
was proved, using a computer, that there exist exactly 10
non-isomorphic graphs, besides \(\operatorname{Kle}_{24}\), which are
cospectral with \(\operatorname{Kle}_{24}\) (and are not
DRGs). Before, the Hoffman graph of order 16 was the smallest known
such graph. This line of investigations was extended in
\cite{vDamHKS06}. 

Different aspects of the genus 3 map related to
\(\operatorname{Kle}_{24}\) have been considered in map theory for a quite
long while. The paper \cite{SchW85} describes a polyhedral
realization of this map. The survey \cite{Ned01} touches the Klein
map only implicitly, however it played a cornerstone role in the
systematical promotion of regular map theory. 

For J.J. van Wijk the Klein map was a source of inspiration in
creation of a computer aided approach for the 3D-visualization of
regular maps. About 45 new examples, see \cite{Wij14}, were
elaborated on this way.

A nice feature of \(\operatorname{Kle}_{24}\) is, in comparison with
say the line graph of the Petersen graph, that it appears as a
Cayley graph, namely over \(S_4\). This allows to treat the Jordan
scheme \(NJ_{24}\) as an analogue of an S-ring. This property of
\(\operatorname{Kle}_{24}\) explains its appearance in
\cite{vDamJ19}. Proposition 7.2 in \cite{vDamJ19} claims that there
exist exactly nine Cayley DRGs with girth 3 and valency 6 or 7. The
number of vertices in these graphs varies from 6 to 27. Clearly,
\(\operatorname{Kle}_{24}\) is in this family. (As opposed to the
other eight DRGs, \(\operatorname{Kle}_{24}\) can be regarded as a
Cayley graph only over a non-abelian group.)

There is also more texts which consider some aspects of
\(\operatorname{Kle}_{24}\) from diverse viewpoints of AGT. Among
these we cite \cite{Dej12}, where \(\operatorname{Kle}_{24}\) and its
dual graph of order 56 are considered in the context of their
symmetry properties. 

\subsection{}
\label{sec:orge8da00a}

Recall that \(\operatorname{Kle}_{24}\) is a DRG, however not a
DTG. In other words, it generates an imprimitive rank 4 AS, which is
non-Schurian: its automorphism group of order 336, isomorphic to
\(\operatorname{PGL}(2,7)\), acts as rank 6 permutation group of
degree 24.  

This fact, which is well-known in AGT, was easily obtained by us
with the aid of COCO. After that, acting in an ad hoc manner and
using the Heawood graph \({\cal H}_{14}\) as an intermediate
structure, we managed to find a reasonably clear computer-free
justification of the fact that the metric scheme generated by
\(\operatorname{Kle}_{24}\) is not Schurian. 

The justification remains on the periphery of the presentation in
the current paper. Nevertheless, it seems it may lead to a very
helpful pattern for the consideration of so-called Tatra ASs, see
\cite{Rei19a} and extra comments in the next section.

\subsection{}
\label{sec:orgba8ed74}

A significant concept in our paper is related to what is usually
called a Delsarte clique in an SRG \(\Gamma\), or, equivalently, a
Hoffman coclique in \(\overline\Gamma\). A spread \(S\) consisting of
such cliques/cocliques is frequently called a \emph{Hoffman coloring.}
The removal of \(S\) from \(\Gamma\) leads to an imprimitive rank 4 AS.
(In general, the removal of any spread from an SRG does not
necessarily imply a rank 4 AS.) It seems that the terminology was
established in \cite{HaeT96}, where also relevant references and
interesting examples are provided. A helpful brief consideration of
the central facts related to Hoffman colorings appears in Sections
3.5 and 3.6 of \cite{BroH12}. In our paper, Hoffman colorings are
playing a crucial role, being exploited for the formulation of
sufficient conditions for the existence of rank 5 Jordan schemes.

\subsection{}
\label{sec:org8b01d83}

Following \cite{BanI84}, in this paper by an association scheme we
understand an arbitrary CC with one fiber, that is a homogeneous CC
in the sense of D.G. Higman. It is necessary to stress that during
the last two decades many experts prefer to use the wording AS in a
more restrictive sense, that is, requiring in addition that it is
commutative or even symmetric. One of the initial sources of this
modern tendency can be traced to a very helpful text
\cite{Cam99}. Here the author justifies such a tendency, appealing
to at least two kinds of arguments:
\begin{itemize}
\item the classical theory of ASs, developed by R.C. Bose and his
colleagues, was working in conjunction with the design of
statistical experiments, where all associate classes were
symmetric by default, see, e.g., \cite{BosM59};
\item the theory of ASs, developed by Delsarte in \cite{Del73},
together with its applications to codes and designs, does not
obligatorily require symmetry, however assumes that the basic
matrices of the adjacency algebra of the considered CC commute.
\end{itemize}

The relevant point of view is further brilliantly justified in the book
\cite{Bai04}, as well as in some further papers by Cameron, like
\cite{Cam03}. 

Fully respecting and understanding this position, we however prefer
to follow the habits, terminology and traditions which were
accepted in the framework of the former Soviet school of algebraic
combinatorics since its inception. 

In particular, the author M.K. participated in the preparation to
the publication of
the Russian translation (1987) of the seminal book \cite{BanI84} by
E. Bannai and T. Ito. For us the notion of a "non-commutative AS"
sounds very natural; it was used in the major part of relevant
publications. Thus we decided to maintain this tradition also in
the current text.

A number of text in AGT, mentioned in Section \ref{orgf276f8a}, are
written in a style and notation compatible with the current
preprint. Besides them there are many other significant sources,
e.g., \cite{CheP19}, \cite{God18}, \cite{Hig87}, see also the next
section. 

\section{Historical remarks}
\label{sec:org607ce07}
\label{org03b7322}

A number of hints of historical nature to a few fields of research
touched in the current text are collected in this section together
with some references.

\subsection{Jordan algebras}
\label{sec:org9d08348}
\label{org7b1a622}
\begin{itemize}
\item The authors are novices in this area. Thus our starting
\emph{disclaimer} in the current subsection aims to stress that the
presented information reflects initial impressions from the
acquaintaince with considered literature. The readers who are
more proficient in Jordan algebras are welcome to put here their
virtual smile and to go ahead.
\item There are at least two reasons which motivated our coming to this
field. First, Jordan algebras and Jordan schemes are naturally
related to coherent algebras and ASs, which are at the permanent
agenda of the authors' interests. Here the roots go to the
intersection of group theory with computer algebra (D.G. Higman
and the Moscow group guided by G.A. Adel'son-Vel'skiǐ and
I.A. Faradžev). 

The second reason has traces in the design of statistical
experiments, see below. (Note that the participation by M.K. at
the R.C. Bose memorial conference, Calcutta 1988, was the first
scientific event abroad that he ever attended. In comparison, in
1978 the invitation to him to attend a conference in Hungary as
an invited speaker was brutally rejected by Soviet authorities.)
\item Already at the intial stage of the project it became necessary to
rely on a few texts which served as origins of the background
information about Jordan algebras. These were \cite{Jac68},
\cite{See71} and \cite{Mal86}. On the further stages we became
more curious, looking for papers and books which aim to provide a
wider panorama of this field.

In this context we definitely wish to mention a nice expository
paper \cite{McC78}, directed toward a general mathematical
audience, as well as the book \cite{McC04}, which was published
25 years later but is still written in a friendly manner. Its
initial chapters were really helpful to get a taste of Jordan
algebras, as well as to understand the roots of the creation of
the corresponding concepts.

Two other background texts are written for a more refined
audience. 

The book \cite{Koe99}, in fact its second edition, deals with the
domains of positivity in a given real Jordan algebra \(A\). Such a
domain appears from the "positive" cone \(A_+=\{a^2\mid a\in
     A\}\). The basic discovery by Koecher was the existence of a
bijection between formally real Jordan algebras and domains of
positivity. This topic is considered in a more general
context of semisimple Jordan algebras and so-called domains of
invertibility.

We also mention the book \cite{FarK94}, which touches links of
Jordan algebras with geometry of symmetric cones and, in
particular, with harmonic analysis on some significant cases of
cones.

\item The appearance of the concept of a Jordan algebra is traced very
well to the fundamental paper \cite{JorNW34}. This is a purely
mathematical text with rigorous claims and detailed accurate
proofs, though not so easily accessible to a contemporary
reader. Diverse sources, from Wikipedia to blogs of experts allow
to trace the motivation of the authors and their goals.

A helpful attempt to clarify the goals of the authors of
\cite{JorNW34} is presented in the book \cite{McC04} by Kevin
McCrimmon. The claim is
that, roughly speaking, Quantum Mechanics was developed
simultaneously with modern linear algebra. In this context the
physical observables are represented by certain matrices
(operators on a Hilbert space). 

The trouble however is that some of the basic operations on
matrices are not "observable" for a physicist. This is why
as early as 1932 the physicist Pascual Jordan proposed to change
the existing paradigm in Linear Algebra and to create a more
suitable algebraic platform for Quantum Mechanics.

It took Jordan a couple of years of hard work to convince his
future coauthors to join him and to prepare this text which is
now considered classical. The main issue was to substitute the
usual multiplication of matrices by what they called
\emph{quasi-multiplication} in 1934, while nowadays it is commonly
accepted as the Jordan product. Its main advantage was to obtain
a symmetric bilinear operation which \emph{preserves} observables. We
will discuss a bit later on how long this marriage of Physics and
Algebra was respected in Physics.

\item Reasonably quickly the field of Jordan algebras, as well as the
related categories of alternative algebras and other structures
became significant in algebra
which is flourishing today, not depending on any relevance to
applications, though the latter might be significant. 

A dozen further names below provide just a sample of experts who made
a serious input to the development of the theory of Jordan
algebras and its various applications:

A.A. Albert, P.M. Cohn, H. Freudenthal, R.L. Griess,
Jr., I.G. Herstein, I.L. Kantor, I.G. MacDonald, R. Moufang,
T.A. Springer, J. Tits, E.B. Vinberg, E. Zel'manov.

The selection of names is based on very subjective criteria,
relying on diverse issues, such as familiarity with a paper, books,
understanding of inputs, short personal acquaintance, or,
simply, occasional email correspondence. A couple of these names
will be mentioned again below in various contexts.

\item Recall that a \emph{special} Jordan algebra appears from a usual
associative algebra via substituting the initial multiplication by
Jordan multiplication. Any Jordan subalgebra of special algebra
is also regarded as a special one. All other Jordan algebras are
called \emph{exceptional}. 

The main result in the above-mentioned fundamental paper
\cite{JorNW34} was that any finite-dimensional so-called formally
real Jordan algebra can be obtained with the aid of direct sums
from five kinds of building blocks. The first three kinds are, in a
sense, traditional special simple Jordan algebras. The fifth kind of
these algebras denoted by \(J(Q)\) will not be discussed in this text. 

A really surprising object is the exceptional algebra \(H_3({\mathbb O})\) of
rank 27 over the octonions.

\item Recall that there exist exactly four finite-dimensional \correction{alternative} division
algebras over the field \({\mathbb R}\), having dimension 1, 2, 4,
8 and denoted by \({\mathbb R}, {\mathbb C}, {\mathbb H}, {\mathbb
     O}\),
respectively. While the real and complex numbers as well as the
quaternions form associative algebras and are familiar to every
educated mathematician, the situation with the algebra \({\mathbb O}\) is
less common.

This structure with a bit orthodoxal name Cayley-Dickson algebra
is nowadays typically jargonically called \emph{octonions}. The
algebra \({\mathbb O}\) is non-associative: There is a famous \emph{Fano plane
mnemonic,} which helps to memorize the rule of multiplication of
units of octonions, also called octaves. We refer to the
fascinating book \cite{ConS03} and the comprehensive survey
\cite{Bae02} as two very helpful sources of information about the
octonions. The style and kind of exposition in these two texts
are complementary to each other. 

The fact that the algebra \(H_3({\mathbb O})\) is really an exceptional
Jordan algebra was rigorously proved by Albert \cite{Alb34}
in 1934. The algebra \(H_3({\mathbb O})\) belongs to a wider class of
simple Jordan algebras, which are also exceptional. An accurate formal
description of this class of structures \emph{(Albert algebras)} is
out of the scope of the current text. We only mention that acting in the
framework of Albert algebras, one may conjecture that, under natural
additional assumptions, the algebra \(H_3({\mathbb O})\) is the only
exceptional Jordan algebra.

A remarkable table, which is closely related to this conjecture,
is usually called the Freudenthal-Tits Magic Square, see, e.g.,
\cite{McC78}, \cite{Bae02} for more details.

Finally, the uniqueness of \(H_3({\mathbb O})\) as the exceptional
simple Jordan
algebra in the terms formulated above, was proved by Efim
Zel'manov, see \cite{Zel79}, a future Fields medalist. Thus
nowadays the algebra \(H_3({\mathbb O})\) is frequently called \emph{the
exceptional Jordan algebra} or the (not an) \emph{Albert algebra}.

Note that our consideration here is deliberately presented in a
slightly simplified form. For example, over the field \({\mathbb
     R}\), up to isomorphism, there exist three exceptional simple
algebras. All of them, in a certain sense, are close to the
algebra \(H_3({\mathbb O})\).

The paper \cite{McC84}, published a few years after the
breakthrough results in \cite{Zel79}, provides a detailed
introduction, mainly oriented to experts in algebra, to what
Kevin McCrimmon calls "The Russian revolution in Jordan
algebras."

Paradoxically, this remarkable page in the history of the
considered subject seems not to be strictly relevant to our
ongoing interests. Indeed, any exceptional algebra cannot appear
as a subalgebra of a suitable matrix algebra, and thus of a
coherent J-algebra.

\item Jordan algebras have also very nice links with projective
geometry. A brief helpful account is given in Section 0.8 of
\cite{McC04}, as well as in the last few pages of
\cite{ConS03}. A comprehensive treatment is provided in
\cite{Fau14}.

\item We now touch more closely topics related to AGT. It seems that
the book \cite{Bai04} was the first where evident and implicit
links of Jordan algebras to the design of statistical experiments
were discussed with enough details. In particular, we learned of
the paper \cite{Sha59} from this book. Working with Shah's paper
we found the text \cite{Har49}. Note that the book by Bailey also
contains references to a few texts by B. Harshbarger. It should
be mentioned that the first papers by Harshbarger were written
before \cite{BosS52}, where concepts of associate classes were
coined. This is why reading of \cite{Har49} required some efforts
in order to reconstruct unicolor relations in the structure with
12 vertices, considered in the example, which goes through all
the paper. The paper \cite{BaiS86} was quite helpful to overcome
ongoing difficulties. Our understanding now is that the example,
considered in \cite{Har49} presents a rank 5 regular symmetric
color graph with valencies 1, 2, 2, 2, 5, which does not provide
a Jordan scheme. 

It is also a pleasure to stress that the book \cite{Bai04} may
literally help our readers. Just one example: Proposition
\ref{org163c63e} in Section \ref{sec:org6af43ec}, in fact, is proved
as part of Proposition 12.5 from Bailey's book.

It should be mentioned that some serious attempts of applications
of Jordan algebras for statistical analysis were done in the
past. Two such texts are mentioned in \cite{Cam18}, namely
\cite{See72} and \cite{Mal86}.

\item During the last decade there are ongoing activities to use
non-proper (in our terms) Jordan algebras with one generator for
serious applications inside of AGT. A typical goal is to
interpret/improve known feasibility conditions for the parameters
of a DRG, in particular of a SRG. One  such text
\cite{CarV04} was mentioned earlier.

Below we list a number of further publications: 
\cite{ManV07}, \cite{ManVM13}, \cite{ManMV14}, \cite{VieM15},
\cite{Vie11}, \cite{Vie19}. The authors consider a special class
of Jordan algebras, which they call \emph{Euclidean} Jordan algebra. In
the simplest cases this is a rank 3 non-proper algebra, which is
generated by a primitive SRG. Using classical matrix analysis,
knowledge of the theory of Jordan algebras on the level of say
\cite{Koe99} and some extra techniques, borrowed, e.g., from
\cite{Fay97}, \cite{Fay15}, they are struggling to establish new
inequalities for combinatorial and spectral parameters of a
putative SRG. There are two kinds of  obtained results:
asymptotical inequalities and computer aided investigations of
concrete parameter sets. Comparison of the provided examples with
the tables on the home page \cite{Bro} shows that a couple of
times the eliminated sets are even not feasible in the classical
context of AGT. Better news are that in the major part of the
reported experiments the killed parameter sets may be eliminated
by one of the Krein conditions and/or absolute bound. This
provides a hope for really good news in the future, when the developed
quite sophisticated techniques will prove their
efficiency. Clearly, the suggested techniques can be exploited
for larger ranks of Jordan algebras.

\item At this point, as a short deviation, let us return to the
consideration of new attempts to find helpful applications of
Jordan algebras in Physics. Zel'manov's proof of the
uniqueness of the exceptional rank 27 Jordan algebra was, in a
sense, for a while bad news for experts in physics: no new
striking algebraic objects are expected.

It seems it took a long time to recover from such an initial shock and
to restart new efforts. An extensive survey \cite{Ior09} is also
supplied by a very rich bibliography. A recent paper
\cite{TodD18} provides a fresh glance on the groups related to
the exceptional algebra over the octonions. See also  earlier
papers \cite{DraM10,Dram15} about the appearance of the octonions in
physics. A provocative title \cite{Ber08} justifies immediate
initial interest to this text. 

We end this very short excursion with a glance on the text
\cite{LevHS17}. In fact, this qute detailed paper is one of many
results presented by the group of researchers working on the edge of a
few branches of physics with group theory and finite
geometries. A lot of considered structures around the famous
Doily, aka \(GQ(2,2)\) and its symmetries promote expectations that
there exists some initially invisible links between a few structures
investigated in our paper and those considered in
\cite{LevHS17}. It should be mentioned that for a mathematician
who is not experienced in theories of gravity and black hole
entropy, the author's claims that their finite geometry models
are reflecting adequately the mentioned physical concepts look
very exotic on first sight.  Excellent news are coming however at
the conclusions, where the authors explain that the main topic of
the paper can be translated exactly to the language of Jordan
algebras, namely to the agenda considered in \cite{Kru07} by
S. Krutelevich (a former Ph.D. student of E. Zel'manov).

\item Let us add a few words to the justification of the terms coherent
J-algebra and coherent JC, introduced in Section
\ref{org4a0fd08}. 

Currently, the only known to us relevant concept was Euclidean
Jordan algebra. It is closely related to the symmetric cones
(domains of positivity) mentioned earlier. Euclidean Jordan
algebras are, in particular, used by the  members of a group
from Portugal, who are applying Jordan algebras to SRG's 

It seems, however, that on the axiomatical level, such algebras
are not obligatorily matrix algebras, closed with respect to SH
product.

Note that in the wording suggested by us there is no
commutativity, that is, a coherent J-algebra is not the same as
a Jordan coherent algebra.
Here by a Jordan coherent algebra one means a coherent algebra
which is simultaneously a Jordan algebra. On one hand, any
adjacency algebra of an SRG shows that these two concepts may
coincide. On the other hand, proper Jordan algebras, the
existence of which is established in the current text, are not
coherent algebras.
\end{itemize}

\subsection{Coherent algebras and WL-stabilization}
\label{sec:org9f54626}
\label{orgf71f8d9}
This is an area of expertise of all the authors and of the
essential part of the expected audience. Thus we do not face the
necessity to discuss historical issues in detail. Brief hints will
be given, as well as certain references.

\begin{itemize}
\item There are at least four relatively independent sources of the
appearance of CCs and ASs.

The first source is the seminal paper \cite{Sch33}, where the
concept, now commonly called  Schur ring or briefly S-ring, was
introduced. For a while, the paper remained unnoticed, however
after WWII, H. Wielandt, a research student of Schur, published a
few texts, including mimeographed notes, which later on appeared
as \cite{Wie64}. These activities of Wielandt's served as the
successful propaganda of Schur's ideas. Since 1964 Schur's method
is becoming popular, first, in permutation group theory and later
on in AGT. 

As was stressed earlier, the second source is activities at
Calcutta due to R.C. Bose and his followers. The first paper(s)
can be traced to 1939, while the texts \cite{BosS52},
\cite{BosM59} and \cite{Bos63} all are regarded as seminal
ones. In particular, the wording and concept of an SRG are
conceived in \cite{Bos63}. Perfect use of the tools from linear
algebra was always typical feature of the members of this
research group. 

The third origin came from the Soviet school, namely from the
paper \cite{WeiL68} and the collection \cite{Wei76}. The motivation
came from the attempts to find an efficient solution of the graph
isomorphism problem. Note that an English translation of
\cite{WeiL68} is now available from the home page \cite{WeiL18} of the conference
WL2018. Another detail is that Boris Weisfeiler is the editor,
but not the sole author of \cite{Wei76}, as one might think.

The last source is due to D.G. Higman with his first paper
\cite{Hig70}, after which there were many further publications
(\cite{Hig90} is one of them). The motivation of Higman was going
from his interests in permutation group theory (before and after
CFSG). 

The reader who is interested in more information of the
historical nature is welcome to
look through some sections of \cite{Bai04}, \cite{KliRRT99},
\cite{BanGPS09}, slides of memorial lectures at Pilsen, 2018, and
other sources, like supplement S.2 to \cite{KliG15}.

\item For a given undirected simple or color graph \(\Gamma\) we can
consider a few kinds of procedures, which are called
stabilization. 

Classical Weisfeiler-Leman stabilization returns the smallest CC
\(WL(\Gamma)\) which contains a graph \(\Gamma\). By symmetric
stabilization in this paper we mean the symmetrization \(\tilde Y=SWL(\Gamma)\)
of \(Y=WL(\Gamma)\). Clearly, the result is defined uniquely. 

The \emph{Jordan stabilization} \(JS(\Gamma)\) returns the coarsest
coherent JC which contains \(\Gamma\). Again, it is possible to
check that this notion is well-defined.

Finally, we mention the \emph{resistance distance transform}
\(RDT(\Gamma)\), which is defined with the aid of a resistance
distance in a graph. For consideration of the related concepts we
refer to, e.g., \cite{KleR93}, \cite{Big97b}, \cite{Big97a}, \cite{Bap99},
\cite{Bap10}, \cite{KagM19}. It is clear that the comparison of
the strength of these kinds of stabilization provides a challenge
for the researchers in AGT. We aim to face the appearing
questions in the future.
\end{itemize}

\subsection{From Siamese objects to Tatra schemes}
\label{sec:org56801ab}
\label{org94cd7b7}
Here we briefly discuss the history of main concepts, which were
exploited in order to reach understanding of the existence of
proper Jordan schemes. Typically, exact definitions are not
provided, the reader is welcome to look through the provided
references.

\begin{itemize}
\item The seminal text here is \cite{KhaT03}. M.K. became acquainted
with it slightly before its formal publication. Here an infinite
family of Siamese color graphs of prime-power order \(q\) is
constructed, relying on the use of weighing matrices in the
spirit of P. Gibbons and R. Mathon \cite{GibM87}. At that time
M.K. was visiting the University of Delaware and was supervising
the Ph.D. thesis of S.R., in fact his former student on the edge
between Germany and Israel.

The results observed in \cite{KhaT03} suggested very helpful line
for computer aided research with further partial interpretation
of the most attractive structures in a human-friendly
manner. Part of the obtained results are reflected in Chapter 5
of the thesis \cite{Rei03}.
\end{itemize}

After defending his thesis, S.R. for a while interrupted his
academic activity. This is why M.K. suggested to Andy Woldar
(A.W.), who should be regarded as an invisible partner in writing
the current paper, to join the tandem of two authors. Two papers
\cite{KliRW05} and \cite{KliRW09} were published; a draft
\cite{KliRW} still is awaiting finishing.

Our first examples of proper Jordan schemes of order 15 and 40
fully stem from the above-mentioned activities. In a sense, one
had to pick up the structures from a box, to dust them off with
closed eyes, and after that, already with open eyes, to figure
out that this seems to be exactly what we are looking for.

\begin{itemize}
\item To continue the line which stems from the Siamese objects, we
have to mention at least a few more issues.

Since the thesis \cite{Rei03} it was desired to get a fully
computer-free proof of the fact that there are exactly two
non-isomorphic Siamese color graphs of order 15. The result was
obtained a few years ago as a combination of efforts of
independent parties. We expect that the current return of the
interest to Siamese structures will warm up the preparation of a
paper by M.K, M. Mačaj, S.R., A. Rosa, as well as A.W. on this
subject.

M. Mačaj (2015) enumerated all Siamese color graphs of
order 40. There are exactly 25245 such structures, two of them
arising
from AS's. There is a preliminary report \cite{Mac16a}. 

Last but not least is a tremendous progress, reached again by
Mačaj, in the understanding, generalization and exploiting of the
ideas from \cite{KhaT03} for a systematical construction of
Siamese color graphs of order \(\frac{q^4-1}{q-1}\) and, in
particular, of Siamese schemes. The preliminary results were
reported at the first Pilsen conference on AGT (2016), also
organized by R. Nedela, see \cite{Mac16b}. Even a brief glance on
the tables reveals a combinatorial explosion, in particular,
hundreds and thousands of DRGs and SRGs for \(q=9\) and \(q=11\),
respectively.

There is no doubt that the results obtained by Mačaj should be
prepared for publication and submitted as soon as possible.

\item One more source of our structures is what is now called WFDF
prolific construction for SRGs. We were already paying credits to
Wallis and to Dima Fon-Der-Flaass (who died at a young age, see
book \cite{DFDF12}, dedicated to his memory). Here, in addition,
we cite three more papers: \cite{Muz07}, \cite{KliMPWZ07},
\cite{CaMS02}. Each of these texts develops the original ideas in
a new direction, both on a theoretical and computer aided
levels. 

As was mentioned, an innovative feature in the appearance of the
WFDF structures at the current text is that we are "pushing"
three disjoint SRGs of the same order, requiring that they will
share the same Hoffman coloring.

\item Recall that what we called already in \cite{KliRW09} a classical
Siamese AS of order \(\frac{q^4-1}{q-1}\) appears from a concrete
imprimitive action of the group \(PSL(2,q^2)\), which has \(q+1\)
thin and \(q+1\) thick classes. 

Quite soon it became clear that an arbitrary imprimitive action
of \(PSL(2,q)\), not only when \(q\) is a square, such that its rank
is \(2l\) for a divisor \(l\) of \(q-1\), might be considered in
order to get interesting initial ASs.

First experiments were done by M.K., who, using COCO, found a
number of new non-Schurian ASs, which appear as mergings of
classes in imprimitive Schurian ASs of even rank.

In this context it is relevant to mention that since 1995 M.K. was
running at BGU an educational seminar on algebraic combinatorics
(AC). 15 years later on or so Eran Nevo joined the Department of
Mathematics at BGU. At that time he became a coorganizer
(together with M.K.) of the already existing seminar on AC. The
ideas, stemming from the fulfilled computations, as well as from
the remarkable paper \cite{McM91} were discussed on the
seminar. It was expected to arrange a joint research project
following these ideas. The sudden death of W. Thurston, who
implicitly was involved to this research, was one of the reasons
why the project produced just one nice short text \cite{Nev15}. Soon
after, Eran returned to his Alma Mater, HUJI.

The project was resumed when S.R. found a temporary position
(until 2018) at
TU Dresden and resumed his systematical research. A first success
was reached and claimed during a bus ride from the Low Tatras to
the High Tatras to a conference organized by our kind supporter
R. Nedela. Since that time all the structures, which appear
around this stream of activity, are called \emph{Tatra schemes.} The
first draft of related results was prepared by M.K. and S.R. 4-5
years ago. Fortunately, now a paper \cite{Rei19a} is in press and
will soon be available to an interested reader. 

We will add a few words in this relation in the next section.

\item A paper \cite{Gun05} by Paul E. Gunnells was already cited in
this text. One of the structures introduced there is a graph
\(G(q)\) of order \(\frac{q^2-1}{2}\) and valency \(q\); in his text \(q\)
is prime. Gunnells claims evidently the existence of a spread
related to his graphs, though in other wording.

The authors of \cite{DeDLM07} call the graphs \(G(p)\) \emph{platonic},
referring to the fact that for \(p=3\) and \(p=5\) we are getting
skeletons of the tetrahedron and the icosahedron,
respectively. For us the first interesting case corresponds to
\(q=7\), leading to the beloved
Klein graph of valency 7. Clearly, in the case of \(q=9\) we face
the now classical Siamese graph of order 40.

The fact that the graphs \(G(q)\) are Ramanujan requires some job
for confirmation. For experts working in Extremal Graph Theory
these graphs are of minor interest because their valency grows
with the order.

Good news is that these structures are now becoming of an
independent interest in AGT. Note that in \cite{Nik07} the
structures \(G(p)\) already are called the \emph{Gunnells graphs}. 

Note also that similar graphs are investigated by other
authors. A couple of references are \cite{BanST04},
\cite{BanT09}, \cite{LiM05}, though many more remain out of this
short glance. 

There is no doubt that each family of similar graphs, appearing as
2-orbits of a suitable imprimitive action of a simple group,
deserve in the future careful investigation with the purpose to
construct new examples of proper Jordan schemes.

\item We are pleased to finish this subsection with a short glance at a
recent paper \cite{KhaS18}. The authors managed to glue together
threads which go back to a few significant lines of activity,
already appearing in the current project.

We mean the seminal book \cite{BanI84}, their own paper
\cite{KhaT03}, our text \cite{KliRW05}, the so-called
"Bannai-Muzychuk criterion", \cite{Ban91,Muz87}, as well as the
ideas of Higman. In particular, the paper introduces a
generalization of the BM criterion for the case of
non-commutative ASs.

It is clear that the new ASs presented in \cite{KhaS18} have
intersection with the Tatra schemes. Comparison of differences
may be a good task for further research.

It seems that the authors of \cite{KhaS18} were not aiming to
give a comprehensive bibliography. However, the given set of
references stresses well the main lines of interest and are
helpful to the reader.
\end{itemize}
\subsection{The current project}
\label{sec:orgdda12fc}
\label{org6b70302}
\begin{itemize}
\item One can distinguish different stages in the development of this
project. The activities around Siamese structures and Tatra
schemes may be regarded as a kind of implicit preparation. As we
will see, they were quite crucial in order to allow the authors
to jump in, as soon as the necessary direction of movement became
visible.

\item There are also other  evident ingredients of the prehistory of our
project. 

A few times M.K. and M.M. were applying research proposals for a
joint cooperation with colleagues from Israel, Germany, USA and
other countries. A few of them were successful, in particular to
GIF and for the European Mobility Grant. A lot of established
research impacts, as well as a necessity to concentrate on fresh
ideas were helpful, no matter if a concrete application was
accepted.

One more ingredient is the slides of a lecture by Cameron
\cite{Cam18} already mentioned. The slides and references there
create an attractive lead to start to think properly about the
problem of the existence of proper schemes.

\item The literal start of the project itself can be traced back to a
certain 
evening in the second half of February 2019. We even remember the
exact date, namely February 22, when a proof of the existence of
\(J_{15}\) was obtained.

\item A first announcement of the results was done by S.R. at the
conference in Bled, June 2019, see \cite{Rei19b}. 
Evident positive attention to
the talk was a pleasant encouragement for our team.
\end{itemize}

\begin{itemize}
\item Summing up this quite long section (even for such a huge text)
one may claim that the reported project appeared as an
amalgamation of activities which can be split into two parts: a
kind of preparation lasting for 15 years since 2003, and half
an year starting in February 2019. 

Taking into account that the members of our trio are involved in
a few pending projects which have taken even longer parts of
their lives, probably the fate of this preprint can be considered
reasonably lucky.
\end{itemize}

\section{Concluding discussion}
\label{sec:org3d2ed22}
\label{org7fa6e54}
\begin{itemize}
\item The preprint is dedicated to the memory of Luba Bron (April 12,
1930 - April 17, 2019), mother-in-law of M.K. She passed away
within about a month after a sudden diagnosis of a pancreatic
cancer at a late stage. The members of the authors' trio deeply
shared these sad news.

\item The genre of this huge preprint is quite unusual. We warn the
reader with a surprising patience, who found time and interest to
look through the entire paper, that the current last section is
quite different in style from the main body of the text. We
discuss here non-formally diverse issues, regarding conducted and
planned research, as well as share our opinions about links of
this part of AGT with other scientific fields.

\item A possible audience of this preprint is stratified into a few
categories, according to the origin of interest, level of patience
and expected further goals. 

Jokingly speaking,  the core of the audience consists of the authors
themselves. We wish to document the history of our understanding of
the existence of the initial new objects, in particular \(J_{15}\)
and \(J_{24}\); to reflect a pleasure from the elaboration of
starting semi-formal visual proofs, to describe outlines of the
roots toward future rigorous proofs of the formulated claims. An
extra goal is to share our enthusiasm for the process of
self-education in the subject of Jordan algebras, which is new to
all of us. 

The next category of potential readers is formed by a part of the
AGT community. We have some evidence to expect that at least a few
colleagues, who are already proficient enough in Jordan algebras,
as well as those who, like us, consider themselves novices, will
express interest to go ahead along the initiated way.

There is also hope that, say, a few well-established experts in
Jordan algebras may be challenged to become acquainted with the
background concepts in AGT and after that to start a new wave of
search for proper J-algebras.

Last but not least is the community of graduate students working
on the intersection of AGT, Jordan algebras and Computer
Science. Our past experience with texts like \cite{KliG15} shows
that an essay written in the style of annotated lecture notes may
help a beginner to quickly enter into a completely new field of
expertise.

\item After putting the first version of this preprint on the arXiv we
expect to create several updates, roughly  speaking once every
couple of months. In the end the text will be conserved for a
while. Its further future currently is not clear at all. 

The announced "dry" version \cite{KliMR} of the current text is
already in preparation. Hopefully, it will be finished soon after
the current one. In lucky circumstances the time of preparation
will be measured in weeks, not months.

In principle there also exists a quite attractive lead for the
author S.R. to prepare his own text in the style of, say, the
Journal of Symbolic Computation. Such a text will aim to describe
with all details the main elaborated algorithms (together with
portions of code), to discuss results of computations and to share
difficulties to extend the area of successful search. More
detailed tables of results will be welcome. Here goes a
microscopical example: to tell in a few words why {\sf GAP}, see Table
\ref{table:proper-schemes}, does not provide for \(q=64\) a more
intelligent name of the group in comparison with all other
appearing names.

Unfortunately, as was mentioned before, S.R. is still looking for
a permanent affiliation. This is why at this stage
the mentioned lead should be regarded as an attractive dream
rather than a realistic plan.

\item There are a number of pending tasks, mainly of computational
nature. Recall that  currently we aim to prove with the aid of a computer
 that all Jordan schemes of order \(n\le 14\) are non-proper.

We do not know if \(J_{15}\) is the unique (here and below we mean
up to isomorphism) proper Jordan scheme of order 15, or if there
are more such schemes (of larger rank). 
The next step would be the enumeration of proper Jordan schemes of
order 16 and 17. For the prime order 17 the task seems to be
reasonably realistic.

Note that for order 16 we face the first occurrence of the
explosion of the number of non-Schurian ASs. Taking this into
account one may guess that for the order \(n=16\) the task of
enumeration would be quite difficult.

Starting from \(n=18\) the task of systematic enumeration of all
Jordan schemes seems to be compuationally much more difficult. For
these orders one may think of more concrete tasks, formulated in
terms of the appearance of prescribed basic graphs. For example,
some graphs of order 18, see \cite{GyuK17} might be reasonably
interesting.

For \(n=21\) possible attractive basic graphs for a putative proper
Jordan scheme seems to be the unique generalized hexagon \(H(2,1)\),
i.e., the line graph of the Heawood graph \({\cal H}_{14}\), a
primitive DRG of valency 4 and order 21. This structure is touched
very briefly in \cite{Coh83} together with the hexagons \(H(2,2)\)
and \(H(2,8)\) of orders 63 and 819, respectively. 

What looks very remarkable is that all three of these structures
are intimately related to Jordan algebras over the octonions. Note
also that implicitely we already faced \(H(2,1)\) in Section
\ref{org508eb9e}, where it was depicted in Figure \ref{figure:14.2.b}.

There are a couple of interesting tasks for order 40. One will be
discussed below. The other is related to the possibility to switch
once more in the model "island and continent" with respect to
other islands and to inspect what is going on.

Finally a realistic task seems to be to classify all rank 5
proper Jordan schemes of order 45 with valencies \(1,8,12^3\) and to
figure out which of them really appear from the prolific
construction and which probably are "quasi-prolific" in the
context discussed above.

\item We expect that the announced paper \cite{KliMR} in preparation,
which was already mentioned a few times, will include a number of
theoretical results which in this preprint were deliberately
touched only briefly.

First we mean an accurate formulation and proof of the theorem
about the prolific construction of order \(\binom{3^d+1}{2}\) and
rank 5. This should include more precise discussion of feasibility
conditions, evaluation of the amount of non-isomorphic objects. An
extra promising lead is to prove that for even \(d\) all appearing
J-schemes are proper.

We also aim to provide accurate and transparent formulations and
proofs of propositions, related to our paradigm "island and
continent", both in its simplest and general versions.

One more attractive lead is to try to elaborate some sufficient
criterion for the algebraic isomorphism of J-schemes with given
order, rank and valencies, which however are not obligatorily
combinatorially isomorphic. For this purpose the techniques in
\cite{AbiH12} by A. Abiad and W. Haemers and developed by their
followers, see, e.g., \cite{WanQH19}, look very promising.

\item A few interesting questions and research tasks appear on the way
of further investigations. From the first sight it is not so
immediate which of them turn out to be almost trivial and which
require much more serious efforts. 

Recall that the smallest non-Schurian CC has order 14, rank 11,
and contains two fibers of size 6 and 8, both are Schurian, see
\cite{KliZA17}, \cite{ZA18}. We are not aware if, in principle,
it is possible to construct proper coherent JC such that its
restriction to each fiber is a non-proper J-scheme.

We still do not have even one example of a \emph{proper primitive
J-scheme}. In such a scheme each basic graph should be connected. 

A possible way to construct such an example is to rely on the
results from \cite{HerM17}. Namely, to take a feasible parameter
set of a primitive non-commutative AS of rank 6 (denoted by \(Y\))
and to try to get a rank 5 JS algebraically isomorphic to the
symmetrization \(\tilde Y\). Note that currently the smallest order
of such a feasible set is equal to 81.

One more issue. Let \(\Gamma\) be a DRG (in particular SRG). Let us
call it \emph{JS-friendly} if \(\Gamma\) appears as a basic graph of a
suitable proper JS. To try to formulate reasonably strong
necessary conditions to carry such a property. The particular case
when \(\Gamma\) is primitive seems to be more difficult. To express
things simply, we are aware that the Petersen graph \(\Pi\) is not
JS-friendly while its line graph \(L(\Pi)\) is.

\item A more systematical way to start to work with the tensors of
structure constants of putative Jordan algebras (not obligatorily
proper) lies through the simultaneous imitation, generalization
and development of what can be called the theory of \correction{\emph{half-integral
table J-algebras}.} We refer here to \cite{AraFM99} as a reasonably
good protagonist of such an approach, see also \cite{Bla09}. 

Here the table algebras are defined in a suitable manner as
abstract analogs of the intersection algebras of ASs and other
similar algebraic structures. Bright, though quite rare examples
of really beautiful reasonings appear, e.g., when a suitable
feasibility condition for the existence of an AS appears to be
proved without the use of any assumptions about the existence of
the corresponding adjacency algebra.

It should be mentioned that, according to our experience, in the
search for proper Jordan algebras it is easier to require the
existence of at least two generators for an algebra (in the sense of
Jordan stabilization), see also extra discussion below.

\item We also mention the existence of results which are already clear
to the authors but will not appear in \cite{KliMR} to avoid
unnecessary explosion of size of that text. For example, relying
on some computer experiments, the question of the characterization
of thin Jordan schemes was posed. Finally it was proved that each
such structure appears as the centralizer algebra of a regular
action of an elementary abelian 2-group.

\item As a first reasonably realistic step in the development of the new
branch of the theory of table algebras we suggest to create a full
parametrization of rank 5 table algebras corresponding to Jordan
schemes, both primitive and imprimitive. In addition to the issues
mentioned above, the paper \cite{KliMZA09}, which goes back to
some results and ideas of D.G. Higman, might be helpful.

\item As was already mentioned, we hope that making this text publicly
available may promote interest of a few of our colleagues to come
back to the topic of Siamese combinatorial objects.

Finishing the enumeration of those proper Jordan schemes of order
40, which relies on careful inspection of all Siamese graphs of
order 40, fully enumerated by M. Mačaj, is very welcome. Similar
activities for the orders \(\frac{q^4-1}{q-1}\), \(4\le q\le 11\),
relying on his already announced results, see, e.g., \cite{Mac16a},
are expected as well.

We also assume that a couple of texts regarding Siamese structures
by M.K., S.R., A.W. and  other colleagues, will be
resumed and finished reasonably soon.

As was mentioned before, a new wave of activities around the Tatra
ASs might also be quite promising. On one hand, here is a good
potential to glue techniques and ideas which stem from the authors
with those by Hadi Kharaghani and his colleagues. On the other
hand, the trick used by us to prove that the rank 4 AS of order 24
is not Schurian seems to be promising for a wider use.

Note also that it seems that not all potentially exploitable
structures considered in \cite{Rei19a} were used in order to find
new proper Jordan algebras.

\item One more promising area of investigations has been started by
Mikhail Kagan and M.K. Here discussions with two other coauthors
of the current text were helpful.

Recall that the concept of the resistance distance transform (RDT)
is based on the use of the Moore-Penrose inverse of a square
matrix. Indeed, for a simple (connected) graph \(\Gamma\) the
resistance distance between any two of its vertices \(x,y\) is
expressed in terms of the corresponding entries of the MP-inverse
of the Laplacian matrix of \(\Gamma\).  On the other hand it is well-known
that a matrix Jordan algebra is closed with respect to taking
MP-inverses. These two facts together imply that the results of
RDT and Jordan stabilization are closely related, while both live
in the symmetrization of the WL-stabilization. More rigorous
research in this direction is now on the agenda. Many experiments
were conducted, while almost all appearing questions still remain
open.

\item At this stage of the discussion, a few issues related to the
philosophy of mathematical research appears in order.

Recall the two concepts of \emph{explanation} and \emph{interpretation} of
computer-aided results that were coined in \cite{KliPRWZA10}. In
this text a number of explanations of obtained results were
provided. Our model "island and continent" seems to be the
brightest example. Provided a concrete proper Jordan algebra is
presented the reader is welcome to make all additional efforts in
order to get a self-contained proof of the existence of the
discussed result.

On the other hand, our models for the Klein graph
\(\operatorname{Kle}_{24}\) serve as, typically quite rare, examples
of an interpretation. Here all necessary logistics is already done
and its results are formulated. The reader is just welcome to
observe the results and to agree that the presented structure
really enjoys its claimed properties.

In this context we face a quite interesting dichotomy, which
naturally appears in consideration of "beautiful" structures in
Mathematics. We mean perceptual properties of visual
representations vs. purely abstract arguments, presented in a
so-called Bourbaki style. A number of outstanding scientists,
among them P. Erdös, J. Hadamard, G.H. Hardy, J.-C. Rota,
B. Russell, expressed their (now classical) opinions on this
subject.

Developing these issues further, one may briefly touch the concept
of abductive reasoning, already mentioned in Section
\ref{orgb33f855}. For readers with even a slight interest to
this line of philosophy of science we refer to the paper
\cite{Wal01} by D.N. Walton.

For us, however, more modest concrete, though still sensitive
issues seem to be relevant, cf. e.g., \cite{Pur19},
\cite{Sta18}, where some case of graph drawings were discussed. We
mean the following. A lot of efforts were done in our text to
present the structure \(J_{15}\). In the beginning it appears as a
result of a computer. Later on we depict a couple of unicolor
graphs on the canvas of the same diagram. Then the same model is
explained in terms of island and continent. Finally, the concrete
model \(J_{15}\) is discussed as a very particular case of a general
abstract formulation. 

Which of these ways is the best for a concrete reader? Clearly, we
do not know a universal reply. However we assume that for each
mode of discussed presentation there will be readers for whom this
concrete pattern will be the best. This assumption is one of the
main justifications of the selected style and size of this
text. Of course, one may rely on such an assumption, writing a huge
preprint, however, it would be impossible to follow the same
pattern in a text like \cite{KliMR}, intended for publication in a
regular journal.

The question of the creation of nice drawings is also quite
painful for us from a purely logistical point of view. Indeed, the
preparation of each of the diagrams presented in the preprint
required a lot of effort, both for the initial draft and the final
rendering. 

Quite recently we became aware of \cite{Wol18}, an excellent Master's
Thesis by Jessica Wolz at Tübingen. A number of nice pictures,
created with the aid of the approach developed there, are already
cited on Wikipedia. We expect in the future to learn the art
presented in that thesis.

\item Remaining a bit longer on a level of philosophical discussion, let
us consider the main difference between the field of AGT, in its
classical sense, and the new field of proper Jordan algebras.

The most interesting structures in AGT are those which have one
generator. The striking kinds of such generators are DRGs and, in
particular, SRGs. Many other structures can be obtained (via
WL-stabilization) from one unicolor graph, which however can
finally split into a few basic graphs. The exceptional objects
from this point of view are exactly amorphic ASs, see, e.g.,
\cite{GolIK94}, \cite{KliKW16}. 

The field of proper Jordan schemes is different in this
context. As we are aware, every Jordan algebra polynomially
generated by  one generator
is non-proper. Thus a requirement that the structure has at least
two generators is a significant one, provided we are interested to face
proper objects.

In a sense this issue explains our essential difficulties in the
initial comprehension of these new structures, as well as further
development of the appearing theory.

\item We feel that one more issue should be briefly touched at this
place.

As was mentioned, the selection of the attribution Jordan to the
algebras and schemes considered in this text is naturally
motivated by the common name "Jordan product" for the considered
binary operation. In turn, this selection was influenced by the
texts and slides by P. Cameron cited above, who also relies on
this terminology. 

There is however another possibility to call these structure by
the name of Shah. The input of B.V. Shah to the generalization of
classical AS's is reflected in the monograph \cite{Bai04} by
Rosemary A. Bailey. In particular, she cites three relevant papers
by Shah, published in 1958-60. The selection of the name, done by
the authors of the current text, is already clear to the reader.
The distance of about 15 years between these two independent
attempts looks very convincing for the accepted decision.

\item Above some quite concrete research plans for the authors and their
close colleagues were discussed.

Here we are taking liberty to imagine a few possible promising
targets. For other experts we regard this as a wild guess, which
is based on the comparison of their results, skills, research
interests. Each time we are trying to motivate our guess by
references.

\item The famous non-associative Griess algebra, described in
\cite{Gri82}, played a background role in the proof of the
existence of the \emph{Monster group}, the largest sporadic finite
simple group. The paper \cite{Gri90} is a quite short but
reasonably friendly text, which clarifies links between the
exceptional Jordan algebra of dimension 27 over \({\mathbb O}\), Moufang loops
and cubic forms. Note also that it was Griess who invested a lot
of time and energy into the preparation of the special issue of
the  Michigan Mathematical Journal dedicated to the memory of
D.G. Higman. Such an impact creates a hope to start quite deep
efforts in order to reveal a striking mystery: is there a new kind
of combinatorial structure which in some sense shares common
properties with ASs and Jordan schemes, but which satisfies
different axioms. Clearly, a concrete example of a model for such
a structure should stem from \(H_3({\mathbb O})\).

\item Antipodal DRG covers of complete graphs are playing a crucial role
in many constructed proper Jordan schemes. Sometimes nowadays such
a graph is attributed by the abbreviation DRACKN (here N stands
for the order of the complete graph \(K_n\)). There is a continuing
interest to DRACKNs, a cornerstone paper is \cite{GodH92}. In
\cite{KliP11} an attempt of a survey of constructions known at that
time is provided, also an idea to construct DRACKNs from
generalized Hadamard matrices is coined and developed. 

A new splash of interest to DRACKNs is related to their close
links with equiangular lines (a classical concept which goes back
to J.J. Seidel, see, e.g., Chapter 11 of \cite{GodR01}) and
so-called finite tight frames. A recent book \cite{Wal18} provides
a detailed introduction to this relatively fresh field. A handful
of references below reflect recent activities, related to this
impact of ideas: \cite{FicMT12}, \cite{CouGSZ16}, \cite{Rah16},
 \cite{FicJKM17}, \cite{FicJMP18}. 

We also mention potentially promising links with such texts as
\cite{KinPV09}, \cite{CFKLT06}.

\item Let us return once more to the papers by Paul Gunnells. Recall that
the exploited graphs \(G(q)\) of order \(\frac{q^2-1}{2}\) and valency
\(q\) were constructed by him. His recent article \cite{Gun18}
strictly deals with Jordan algebras, referring to already
discussed texts \cite{JorNW34}, \cite{Koe99}, \cite{McC04} and a
few other attractive sources. This lovely text is devoted
to  some structure which is called by the author the exotic
matrix model. 
 Why not to hope that such  knowledge of  two so far almost
non-intersecting fields may lead to a progress of the mystery
which was already discussed a few times above: Can one somehow use
the exceptional Albert algebra for new strict purposes in AGT?

\item We end the last section of this essay approaching again a
philosophical issue.

It was a pleasant surprise to discover the tip of the iceberg
considered above. It was quite uncomfortable to lose many times
the feeling of balance in the course of the ongoing attempts to
stand on it quite stably.

It seems that indeed a new field of research has been started. No
doubt that in the future our followers will reach much deeper,
more transparent and general understanding of this area.

As happens frequently in mathematics our initial examples quite
probably will play a minor role in the forthcoming general theory
yet to be developed. Nevertheless we do hope that a very specific
kind of credit should be given later to this text: Coining of the
term "proper Jordan scheme".

With such positive expectations the \emph{proper} part of the essay is
finished.
\end{itemize}

\section{Acknowledgments}
\label{sec:org62a19f7}

We are much obliged to Rosemary Bailey and Peter Cameron for helpful
discussions and texts, which promoted our interest to the concept of
Jordan algebras.

Collaboration with Mikhail Kagan, Martin Mačaj, Alex Rosa and Andy
Woldar was quite crucial in preparation to the research, reported in
our essay.

We thank Roman Nedela for the ongoing attention and kind support
during many years. 

Logistical support on diverse stages of this project by Katie
Broadhead, Andries Brouwer, Chris Godsil, Stefan Gyürki, Danny
Kalmanovich, Jan Karabaš, Wolfgang Knapp and his colleagues, Eran
Nevo, Grigory Ryabov is appreciated.

Many colleagues expressed interest to some earlier activities
reflected at this essay. Among them are Eiichi Bannai, Henry Cohn,
Igor Faradžev, Gabor Gevay, Ivan Gutman, Allen Hermann, Gareth
Jones, Bill Kantor, Otto Kegel, Hadi Kharaghani, Josef Lauri, Felix
Lazebnik, Doug Klein, Ilia Ponomarenko, Milan Randič, Metod Saniga,
Bangteng Xu. We thank everybody of them.

\nocite{Wol18}

\bibliographystyle{alpha}

\bibliography{jordan}
\end{document}